\UseRawInputEncoding
\documentclass[12pt]{amsart}

\usepackage[all]{xy}
\usepackage{fullpage}
\usepackage{latexsym}
\usepackage{amsmath}
\usepackage{amsfonts}
\usepackage{amssymb}
\usepackage{amsthm}
\usepackage{eucal}
\usepackage{enumerate,yfonts}
\usepackage{mathrsfs}
\usepackage{graphicx}
\usepackage{graphics}
\usepackage{epstopdf}
\usepackage{amscd}
\usepackage{bbm}
\usepackage{hyperref}
\usepackage{url}
\usepackage{color}
\usepackage{bbm}
\usepackage{cancel}
\usepackage{enumerate}
\usepackage{amsmath,amsthm}
\usepackage{amssymb}
\usepackage{epsfig}
\usepackage{pstricks}
\usepackage{xy}
\usepackage{xypic}
\usepackage{bbm}
\usepackage{inputenc}

\newtheorem{thm}{Theorem}[section]
\newtheorem{corollary}[thm]{Corollary}
\newtheorem{lemma}[thm]{Lemma}

\newtheorem{proposition}[thm]{Proposition}
\newtheorem{prop}[thm]{Proposition}
\newtheorem{conjecture}[thm]{Conjecture}
\newtheorem{thm-dfn}[thm]{Theorem-Definition}

\theoremstyle{definition}
\newtheorem{definition}[thm]{Definition}
\newtheorem{remark}[thm]{Remark}
\newtheorem{example}[thm]{Example}

\numberwithin{equation}{section}

\newcommand{\fg}{{\mathfrak g}}

\newcommand{\fp}{{\mathfrak p}}

\newcommand{\bC}{{\mathbb C}}
\newcommand{\bP}{{\mathbb P}}

\newcommand{\bG}{{\mathbb G}}
\newcommand{\bZ}{{\mathbb Z}}

\newcommand{\mE}{\mathcal{E}}
\newcommand{\mF}{\mathcal{F}}

\newcommand{\mM}{\mathcal{M}}

\newcommand{\mO}{\mathcal{O}}
\newcommand{\mL}{\mathcal{L}}

\newcommand{\mG}{\mathcal{G}}

\newcommand{\sY}{\mathscr{Y}}

\newcommand{\on}{\operatorname}

\newcommand{\lra}{\longrightarrow}

\newcommand{\ra}{\rightarrow}
\newcommand{\la}{\leftarrow}

\newcommand{\bs}{\backslash}
\newcommand{\is}{\simeq}

\newcommand{\Loc}{\on{LocSys}}
\newcommand{\Bun}{\on{Bun}}

\newcommand{\quash}[1]{}  
\newcommand{\nc}{\newcommand}

\newcommand{\frakg}{{\mathfrak g}}

\newcommand{\bbA}{{\mathbb A}}

\newcommand{\bbC}{{\mathbb C}}

\newcommand{\bbG}{{\mathbb G}}
\newcommand{\bbH}{{\mathbb H}}

\newcommand{\bbP}{{\mathbb P}}

\newcommand{\bbR}{{\mathbb R}}
\newcommand{\bbS}{{\mathbb S}}

\newcommand{\calB}{{\mathcal B}}

\newcommand{\calE}{{\mathcal E}}
\newcommand{\calF}{{\mathcal F}}
\newcommand{\calG}{{\mathcal G}}
\newcommand{\calH}{{\mathcal H}}

\newcommand{\calK}{{\mathcal K}}
\newcommand{\calL}{{\mathcal L}}
\newcommand{\calM}{{\mathcal M}}

\newcommand{\calO}{{\mathcal O}}

\newcommand{\calQ}{{\mathcal Q}}

\newcommand{\calS}{{\mathcal S}}
\newcommand{\calT}{{\mathcal T}}
\newcommand{\calU}{{\mathcal U}}
\newcommand{\calV}{{\mathcal V}}

\newcommand{\calX}{{\mathcal X}}
\newcommand{\calY}{{\mathcal Y}}
\newcommand{\calZ}{{\mathcal Z}}

\nc{\al}{{\alpha}} \nc{\be}{{\beta}} \nc{\ga}{{\gamma}}
\nc{\ve}{{\varepsilon}} \nc{\Ga}{{\Gamma}} 
\nc{\La}{{\Lambda}}

\nc{\ad }{{\on{ad }}}

\nc{\aff}{{\on{aff}}} \nc{\Aff}{{\mathbf{Aff}}}
\newcommand{\Aut}{{\on{Aut}}}

\nc{\der}{{\on{der}}}

\nc{\diag}{{\on{diag}}}

\nc{\Fl}{{\calF\ell}}
\newcommand{\Gal}{{\on{Gal}}}
\newcommand{\Gr}{{\on{Gr}}}
\nc{\Hg}{{\on{Higgs}}}
\newcommand{\Hom}{{\on{Hom}}}
\newcommand{\id}{{\on{id}}}
\nc{\Id}{{\on{Id}}}

\nc{\Ind}{{\on{Ind}}}

\nc{\Op}{{\on{Op}}}

\newcommand{\pr}{{\on{pr}}}

\nc{\res}{{\on{res}}}

\newcommand{\Spec}{{\on{Spec}}}

\nc{\tr}{{\on{tr}}}

\newcommand{\GL}{{\on{GL}}}
\nc{\GSp}{{\on{GSp}}} \nc{\GU}{{\on{GU}}} \nc{\SL}{{\on{SL}}}
\nc{\SU}{{\on{SU}}} \nc{\SO}{{\on{SO}}}

\nc{\nh}{{\Loc_{J^p}(\tau')}}
\nc{\bnh}{{\Loc_{\breve J^p}(\tau')}}

\nc{\bU}{{\overline{U}}} \nc{\IC}{{\on{IC}}}

\newcommand{\bA}{\mathbb A}

\newcommand{\beqn}{\begin{equation*}}
\newcommand{\eeqn}{\end{equation*}}

\newcommand{\beq}{\begin{equation}}
\newcommand{\eeq}{\end{equation}}

\newcommand{\sw}{\on{sw}}

\nc{\QM}{QM}
\nc{\eval}{\textup{ev}}

\begin{document}
\title{Real groups, symmetric varieties and Langlands duality}

       \author{Tsao-Hsien Chen} 
        
       \address{ School of Mathematics, University of Minnesota, Twin cities, Minneapolis, MN 55455}
       \email{chenth@umn.edu}
       \author{David Nadler} 
        
        \address{Department of Mathematics, UC Berkeley, Evans Hall,
Berkeley, CA 94720}
        \email{nadler@math.berkeley.edu}
       
\maketitle\begin{abstract}
Let $G_\bbR$ be a connected real reductive group and let $X$ be the corresponding complex symmetric variety under the Cartan bijection.
We construct a canonical equivalence between the  
relative Satake category of $G(\calO)$-equivariant $\bC$-constructible complexes on the loop space 
$X(\calK)$
and the  
real Satake category
of $G_\bbR(\calO_\bbR)$-equivariant $\bC$-constructible complexes on 
the real affine Grassmannian $\Gr_\bbR=G_\bbR(\calK_\bbR)/G_\bbR(\calO_\bbR)$.
We show that the equivalence is $t$-exact with respect to the 
natural perverse $t$-structures and is
compatible with the 
fusion products and Hecke actions.
 We further show that the relative Satake category
 is equivalent to 
the category of $\bC$-constructible complexes on the
 moduli stack $\Bun_{G_\bbR}(\bP^1(\bbR))$ of 
 $G_\bbR$-bundles 
 on the real projective line $\mathbb P^1(\bbR)$ and hence provides a connection between the  
 relative Langlands program and the geometric Langlands program for real groups.

We provide numerous applications  of the main theorems to real and relative 
Langlands duality including
the formality and commutativity conjectures for the real and relative  Satake categories
and an
identification of the  dual  groups for 
$G_\bbR$ and $X$.

\end{abstract}
\setcounter{tocdepth}{2} \tableofcontents

 \section{Introduction} 
 
 Let $G_\bbR$ be a  real form of a connected complex
reductive group $G$. Let $X=K\backslash G$ be the associated symmetric variety under Cartan's
bijection, where $K$ is the complexification of a maximal compact subgroup $K_c\subset G_\bbR$.
A fundamental feature of the representation theory of the real group $G_\bbR$ is that many
results of an analytic nature have equivalent purely algebraic geometry formulations in terms of the
corresponding symmetric variety $X$. We will call this broad phenomenon the real-symmetric
correspondence.

In this paper we
study the real-symmetric correspondence in the framework of Langlands duality.
We show that there is an equivalence between the 
relative Satake category of $X$ and the real Satake category of 
$G_\bbR$ with remarkable properties (Theorem \ref{main thm 1}).
We further show that  the relative Satake category 
is equivalent to the dg derived category of sheaves on the moduli stack
of 
$G_\bbR$-bundles on the real projective line $\bP^1(\bbR)$,
hence provides a connection between real and relative Langlands programs (Theorem \ref{main thm 2}).
The proof relies on three geometric results:
 (1) a multi-point version of 
Quillen's 
homeomorphism
between the loop spaces for compact symmetric varieties and  real affine Grassmannians (Section \ref{QMaps})
(2) Morse-theoretic construction of the Matsuki correspondence for the affine Grassmannian (Section \ref{Morse flow})
(3) uniformization of moduli stack of  quasi-maps and real bundles.
(Section \ref{uniform}).

We provide numerous applications of the main results to real and relative Langlands duality including 
semi-simplicity and 
t-exactness criteria of the Hecke actions,
the formality and commutativity conjectures for the real and relative Satake categories, 
and an identification of the (Tannakian) dual groups for $G_\bbR$ and $X$.
The last application provides an explicit description of the 
 dual group of $X$
answering a basic open question in relative Langlands duality.

We now describe the paper in more details.

\subsection{Main results}
Let $\Gr=G(\calK)/G(\calO)$ be the affine Grassmannian for $G$
and let $D(G(\calO)\backslash\Gr)$ be the Satake category of 
$G(\calO)$-equivariant $\bC$-constructible complexes on $\Gr$.
One of the foundational result in Langlands duality is the geometric  
Satake equivalence \cite{BD,G,Lu,MV}
\[\on{Perv}(G(\calO)\backslash\Gr)\is\on{Rep}(G^\vee)\]
providing a description of the category of representations of the  Langlands dual group
$G^\vee$ in terms of 
the abelian Satake category $\on{Perv}(G(\calO)\backslash\Gr)\subset
D(G(\calO)\backslash\Gr)$
of $G(\calO)$-equivariant perverse sheaves on $\Gr$.
 
Let 
$\Gr_\bbR=G_\bbR(\calK_\bbR)/G_\bbR(\calO_\bbR)$ be the real affine Grassmanian 
of $G_\bbR$ 
and let $X(\calK)$ be the loop space of $X$.
We are interested in the  real Satake category $D(G_\bbR(\calO_\bbR)\backslash\Gr_\bbR)$ 
of $G_\bbR(\calO_\bbR)$-equivariant $\bC$-constructible complexes on 
$\Gr_\bbR$ and the relative Satake category $D(X(\calK)/G(\calO))$
of $G(\calO)$-equivariant $\bC$-constructible complexes on $X(\calK)$.
Those categories are one of the main players in the 
geometric Langlands for real groups \cite{BZN} and the relative Langlands program \cite{BZSV}.

In this paper we will assume $G_\bbR$ is connected (equivalently, $K$ is connected).
Our first main result is a remarkable equivalence between the 
real and relative Satake category, called the real-symmetric equivalence:

\begin{thm}\label{main thm 1}
There is a natural equivalence 
\[D(X(\calK)/G(\calO))\is D(G_\bbR(\calO_\bbR)\backslash\Gr_\bbR)\]
which is $t$-exact with respect to the perverse $t$-structures and is compatible with the fusion products 
and Hecke actions of $\on{Rep}(G^\vee)$.
\end{thm}

Theorem \ref{main thm 1} is the combination of 
Theorem \ref{real-symmetric}, Theorem \ref{conv comp}, and Theorem \ref{fusion compatibility} 
 in the text; we refer  to Section \ref{real sym}, Section \ref{s:hecke} and Section \ref{fusion} for a
more detailed explanation of the statement, including the definition of $t$-structures, fusion products, and 
Hecke actions. The main ingredient in the proof is a multi-point version of 
Quillen's 
homeomorphism 
between the loop spaces for compact symmetric varieties and the real affine Grassmannians,
see Theorem \ref{Quillen}.

To state our second main result, 
let 
$\Bun_{G_\bbR}(\bbP^1(\bbR))$ be the real analytic stack of $G_\bbR$-bundles on the 
real projective line $\bP^1(\bbR)$.
Let $LG_\bbR\ \text{and}\  G_\bbR(\bbR[t^{-1}])\subset G_\bbR(\calK_\bbR)$ be the 
polynomial loop group and polynomial arc group  of $G_\bbR$ respectively.
Denote by 
$D(\Bun_{G_\bbR}(\bbP^1(\bbR)))$ the  dg category of 
$\bC$-constructible complexes 
on the real analytic stack $\Bun_{G_\bbR}(\bbP^1(\bbR))$,
and  by
 $D(LG_\bbR\backslash\Gr)$ (resp. $D(G_\bbR(\bbR[t^{-1}])\backslash\Gr_\bbR)$)
 the dg categories 
 of $LG_\bbR$-equivariant  (resp. $G_\bbR(\bbR[t^{-1}])$-equivariant)
$\bC$-constructible complexes on $\Gr$ respectively.

\begin{thm}\label{main thm 2}
There are natural commutative diagram of equivalences  
\[\xymatrix{D(X(\calK)/G(\calO))\ar[rr]^\simeq\ar[d]^\simeq_{\Upsilon}&&D(G_\bbR(\calO_\bbR)\backslash\Gr_\bbR)\ar[d]^\simeq_{\Upsilon_\bbR}\\
D(LG_\bbR\backslash\Gr)\ar[dr]_\simeq\ar[rr]^\simeq&&D(G_\bbR(\bbR[t^{-1}])\backslash\Gr_\bbR)\ar[ld]^\simeq\\
&D(\Bun_{G_\bbR}(\bbP^1(\bbR)))&}\]
where $\Upsilon$ is the so called affine Matsuki equivalence,
$\Upsilon_\bbR$ is the Radon transform, the horizontal equivalences 
are nearby cycles functors along quasi-maps family,
and the vertical equivalences
in the lower triangle are induced from the 
complex and real uniformizations  of $G_\bbR$-bundles
\[\xymatrix{LG_\bbR\backslash\Gr\ar[r]^{\simeq\ \ }&\Bun_{G_\bbR}(\bbP^1(\bbR))&G_\bbR(\bbR[t^{-1}])\backslash\Gr_\bbR\ar[l]_{\ \ \simeq}}.\]
Moreover, the equivalences above are compatible with the natural 
Hecke actions of $\on{Rep}(G^\vee)$.
\end{thm}

Theorem \ref{main thm 2} is restated in 
Theorem \ref{diagram}.
We refer  to Section \ref{Affine Matsuki} and Section \ref{nearby cycles and Radon TF} for a
more detailed explanation of the statement.
The main step in the proof of Theorem \ref{main thm 2} is the 
affine Matsuki equivalence $\Upsilon$ whose proof relies on a 
Morse-theoretic construction (Theorem \ref{flow}) of the Matsuki correspondence for the affine Grassmannian in \cite{N1}: an isomorphism 
between $G(\calO)$-orbits poset on $X(\calK)$ (or rather $K(\calK)$-orbits poset on $\Gr$)
and $LG_\bbR$-orbits poset on $\Gr$.

 Theorem \ref{main thm 1} and 
 Theorem \ref{main thm 2}
  provide a connection between relative Langlands program
and geometric Langlands program for real groups
and  we expect applications of such a connection to both subjects.
For example, Theorem \ref{main thm 2} has been used in \cite{CMNO}
to establish versions of the
 relative Langland duality conjecture and 
 geometric Langlands for 
 $\bP^1(\bbR)$ in the case $(G_\bbR,X)=(\on{GL}_n(\bbH),\on{Sp}_{2n}\backslash\GL_{2n})$ where $\on{GL}_n(\bbH)$
is the real quaternionic linear group (see 
 next section for more applications).

\begin{remark}
In the case when $G_\bbR\is H$ is a connected complex reductive group viewed as a real group, Theorem \eqref{main thm 2} recovers the results of V. Lafforgue 
\cite[Proposition 2.1]{La} saying that 
there are equivalences
\beq\label{Lafforgue}
\xymatrix{D(H(\calO)\backslash\Gr_H)\ar[r]^\simeq& D(H(\bC[t^{-1}])\backslash\Gr_H)\ar[r]^\simeq&D(\Bun_H(\mathbb P^1))}
\eeq
where the first equivalence is given by the Radon transform and the second equivalence 
comes from the uniformization isomorphism
$H(\bC[t^{-1}])\backslash\Gr_H\is \Bun_H(\mathbb P^1)$.
\end{remark}

\begin{remark}
The proof of Theorem \ref{main thm 2} incorporated technical material in 
\cite{CN2} on Matsuki equivalence for affine grassmannians
but with many new results, including a proof of 
a conjecture on identification of the spherical and real dual groups in \cite[Section 1.4.2]{CN2}
(see Section \ref{Identification of dual groups, intro} for more details).
On the other hand, an updated version of \cite{CN2}  will 
include  various generalizations of the main results in \emph{loc. cit.}, including the case for affine flag varieties.

 \end{remark}

\subsection{Applications}
Our main results
allow one to use powerful algebraic geometry  tools on the symmetric side (e.g., Deligne's theory of weights) to study questions
on the real side, and conversely, to use the concrete geometry on the real side (e.g., the real affine Grassmannian or moduli of real bundles) to study questions on
the symmetric side. 
Here are a few notable examples.

\subsubsection{Semi-simplicity of Hecke actions}
\begin{corollary}[Theorem \ref{semisimplicity}]
The Hecke action of $\on{Rep}(G^\vee)$ on the real Satake category $D(G_\bbR(\calO_\bbR)\backslash\Gr_\bbR)$ (resp. the relative Satake category $D(X(\calK)/G(\calO))$) 
preserves semi-simplicity, that is, it maps semi-simple objects to semi-simple objects.
\end{corollary}

In the case of relative Satake category the corollary above is a direct consequence of the 
decomposition theorem in complex algebraic geometry. However, it is not 
obvious in the case of real Satake category since 
decomposition theorem might not hold in the real analytic setting.

\subsubsection{$t$-exactness criterion of Hecke actions}
\begin{corollary}[Theorem \ref{t-exact}]
The Hecke action of $\on{Rep}(G^\vee)$ on the relative Satake category $D(X(\calK)/G(\calO))$
 (resp. the real Satake category $D(G_\bbR(\calO_\bbR)\backslash\Gr_\bbR)$) 
is $t$-exact with respect to the perverse $t$-structure if and only if 
$X$ is quasi-split (resp. $G_\bbR$ is quasi-split). 
\end{corollary}

In the case of real Satake category the corollary follows from the semi-smallness of 
the convolution morphisms for real affine Grassmannian and the $t$-exactness criterion of nearby cycles functor in 
\cite[Theorem 1.2.3]{N2}.
On the other hand, the corresponding $t$-exactness criterion for the 
relative Satake category is not obvious due the complicated spherical geometry of the loop 
space $X(\calK)$.

\quash{
\begin{remark}
We expect that 
the above
$t$-exactness criterion
is related to the 
semi-smallness theorem in
\cite[Theorem 1.3.1]{SW}, a key technical results
in their work on local unramified $L$-functions for 
affine spherical varieties.

\end{remark}
}
\subsubsection{Formality and commutativity of dg Ext algebras}
The next corollary confirms the  formality and commutativity 
conjecture for the real  and relative Satake category 
(see, e.g.,\cite[Conjecture 8.1.8]{BZSV}).
Recall the 
dg extension algebras 
$A_\bbR=\on{RHom}_{D(G_\bbR(\calO_\bbR)\backslash\Gr_\bbR)}(\delta_\bbR,\delta_\bbR\star\IC_{reg})$ 
and $A_X=\on{RHom}_{D(X(\calK)/G(\calO))}(\omega_{X(\calO)/G(\calO)},\omega_{X(\calO)/G(\calO)}\star\IC_{reg})$
for $D(G_\bbR(\calO_\bbR)\backslash\Gr_\bbR)$ and $D(X(\calK)/G(\calO))$
(see Section \ref{FC} for the precise definition).
Note that both $A_\bbR$ and $A_X$ carry natural $G^\vee$-actions 
induced from the one on $\IC_{reg}$.

\begin{corollary}[Theorem \ref{formality and comm}]\label{formality and comm intro}
(1) There is a $G^\vee$-equivariant isomorphism of dg algebras $A_\bbR\is A_X$.
(2) The dg algebras $A_\bbR\is A_X$ are formal, that is, they are quasi-isomorphic 
to the cohomology algebrass $H^\bullet(A_\bbR)\is H^\bullet(A_X)$ with trivial differentials.
(3) The cohomology algebras $H^\bullet(A_\bbR)\is H^\bullet(A_X)$ are commutative.
\end{corollary}
 The formality of $A_X$ is proved in \cite[Theorem 27]{CY}
using a pointwise purity result for $\IC$-complexes of spherical orbits  in $X(\calK)$.\footnote{In \cite[Theorem 27]{CY}, we only treat the case of 
classical symmetric varieties. 
Thanks to the work of Drinfeld and Bouthier \cite{D,B}, we now know that 
$X(\calK)$ is $G(\calO)$-ind placid and the argument in \emph{loc. cit.}
can be generalized to the general case. The details will appear in the 
revised version of \cite{CY}.} Thus the formality  for $A_\bbR$ is a bit surprising and non-obvious since Hodge theory or the theory of weights
is not available on the real analytic setting.
The proof of the commutativity of $H^\bullet(A_\bbR)$ is similar to the case of complex groups studied in \cite{ABG,BFN}.

\begin{remark}
The formality of $A_\bbR$ in the case of   the
 real quaternionic group $\GL_n(\bbH)$
 was proved in \cite{CMNO} by a different method.
It relies on an explicit computation of a 
 morphism 
  $A\to A_\bbR$ from the 
 dg Ext algebra $A$ for the Satake category of the 
 complex group
 $\GL_{2n}$ to $A_\bbR$. The argument in \emph{loc. cit.} uses 
 some particular properties of the real quaternionic group 
which might not hold for other real groups (but it provides more information
about $A_\bbR$).
  \end{remark}
  
  \subsubsection{Hamiltonian duals of $X$ and $G_\bbR$}\label{duals}
  
\begin{definition}
Introduce the affine schemes $M^\vee_X=\on{Spec}(H^\bullet(A_X))$
and $M^\vee_\bbR=\on{Spec}(H^\bullet(A_\bbR))$. 
Inspired by the work of \cite{BZSV} and \cite{BFN}
on relative Langlands duality and Coulomb branches, 
we will call $M^\vee_X$ and $M^\vee_\bbR$ the Hamiltonian dual of 
$X$ and $G_\bbR$.
\end{definition}

Let $D_c(X(\calK)/G(\calO))\subset D(X(\calK)/G(\calO))$ and $D_c(G_\bbR(\calO)\backslash\Gr_\bbR)\subset
D(G_\bbR(\calO)\backslash\Gr_\bbR)$ be the (non co-complete) full subcategories
consisting of constructible complexes that are extensions by
zero off of substacks
and let
$D_c(X(\calK)/G(\calO))_0\subset D_c(X(\calK)/G(\calO))$ and $D_c(G_\bbR(\calO_\bbR)\backslash\Gr_\bbR)_0\subset
D_c(G_\bbR(\calO_\bbR)\backslash\Gr_\bbR)$ be the  full subcategories 
generated by the irreducible direct summands of 
$\omega_{X(\calO)/G(\calO)}\star\IC_{reg}$ and $\delta_\bbR\star\IC_{reg}$ respectively.
Denote by $\on{Coh}(M^\vee_X/G^\vee)$ and $\on{Coh}(M_\bbR^\vee/G^\vee)$ the dg derived categories 
of coherent complexes on the stack $M^\vee_X/G^\vee$ and $M_\bbR^\vee/G^\vee$.
The following corollary follows from Corollary \ref{formality and comm intro}:  
  
\begin{corollary}[Theorem \ref{real-sym satake}]\label{real-symmetric derived Satake}
(1) There is a $G^\vee$-equivariant isomorphism 
$M^\vee_X\is M^\vee_\bbR$ 
(2) There are equivalences of categories
\beq\label{derived Satake intro}
D_c(X(\calK)/G(\calO))_0\is \on{Coh}(M^\vee_X/G^\vee)\ \ \ \ D_c(G_\bbR(\calO_\bbR)\backslash\Gr_\bbR)_0\is \on{Coh}(M_\bbR^\vee/G^\vee).
\eeq
\end{corollary}

\begin{remark}
(1) It would be nice if one can find a description of 
$M^\vee_X$ or  $M^\vee_\bbR$ in terms of the combinatoric structures of 
$X$ or $G_\bbR$.
In the recent work \cite{BZSV}, the authors proposed such a description 
for a certain class of symmetric varieties (in fact, in \emph{loc. cit.} they consider a more general setting of 
spherical varieties)
(2) Due to the existence of 
non-trivial equivariant local systems on $G(\calO)$-orbits in $X(\calK)$ (resp. $G_\bbR(\calO_\bbR)$-orbits in $\Gr_\bbR$), the 
 relative Satake category $D(X(\calK)/G(\calO))$
 (resp. real Satake category $D(G_\bbR(\calO_\bbR)\backslash\Gr_\bbR)$)
in general 
might be bigger than $D_c(X(\calK)/G(\calO))_0$ (resp. $D_c(G_\bbR(\calO_\bbR)\backslash\Gr_\bbR)_0$) and 
it is an interesting question to extend the spectral description in~\eqref{derived Satake intro} to 
to the entire dg derived category.
In view of Theorem \ref{main thm 2}, such a spectral description
would imply a version of geometric Langlands on the real projective line $\bP^1(\bbR)$, and vice versa.
\end{remark}

\subsubsection{Identification of dual groups}\label{Identification of dual groups, intro}
The paper  \cite{N1} associates to each real form $G_\bbR\subset G$ 
a connected complex reductive subgroup $H^\vee_{real} \subset G^\vee$ of the  dual group.\footnote{While the notation suggests
regarding $H_{real}^\vee$
 itself as a  dual group, at the moment we do not know of a concrete role for its dual group.} 
The construction of $H_{real}^\vee$ is via Tannakian formalism:
the 
tensor category of finite-dimensional representations $\on{Rep}(H_{real}^\vee)$
can be realized as a certain full subcategory $Q_\bbR\subset\on{Perv}(G_\bbR(\calO_\bbR)\backslash\Gr_\bbR)$
of perverse sheaves 
on $\Gr_\bbR$. 
In \cite[Section 10]{N2}, a concrete description 
of $H_{real}^\vee$ is given including the root datum and the Weyl group.
On the other hand, the papers~\cite{GN1, GN2} associate to every spherical subgroup $K\subset G$ a reductive
subgroup $H_{sph}^\vee \subset G^\vee$ of the dual group. Again,
the construction of $H_{sph}^\vee$ is via Tannakian formalism: its tensor category of finite-dimensional representations 
$\textup{Rep}(H_{sph}^\vee)$  can be realized as a certain full subcategory $Q_K^{glob}$  of generic-Hecke equivariant perverse sheaves on the (global) moduli stack of 
quasi-maps with target $X=K\backslash G$. 
However, unlike the real group case, a concrete description of $H_{sph}^\vee$ is not known. For example, 
the fact that the root systems or the Weyl group of $H^\vee_{sph}$ is  the same as that associated to 
$X$ in the structure theory of spherical varieties \cite{B,KS} remains conjectural.

Consider the case when $K\subset G$ is the symmetric subgroup of a real form $G_\bbR\subset G$.
Our last example
provides a tensor equivalence between $Q_\bbR\is Q_{K}^{glob}$ 
and hence an isomorphism $H_{real}^\vee\is H_{sph}^\vee$
of reductive subgroups of $G^\vee$.
In particular,
we obtain a 
concrete description of 
$H_{sph}^\vee$ 
answering a basic open question in relative Langlands duality:

\begin{corollary}\label{Q_X=Q_R, intro}
There are horizontal tensor equivalences in the following commutative diagram
of tensor functors
\[\xymatrix{&\on{Rep}(G^\vee)\ar[dl]\ar[dr]&\\
Q_\bbR\ar[rr]^\simeq&&Q_K^{glob}}\]
where the vertical arrows are given by the Hecke action of 
$\on{Rep}(G^\vee)$ on the monoidal units.
\end{corollary}

Corollary \ref{Q_X=Q_R, intro} is the combination of 
Proposition \ref{Q_X=Q_R},
Proposition \ref{H_sph}, and Theorem \ref{local=global}
 in the text; we refer  to Section \ref{identification} for a
more detailed explanation of the statement.
The main ingredient in the proof is a 
local-global comparison theorem for relative Satake category 
in Theorem \ref{local=global}.

\begin{corollary}\cite[Conjecture 7.3.2]{GN1}\label{Weyl groups, intro}
There is an isomorphism of reductive groups $H^\vee_{sph}\is H^\vee_{real}$. 
In particular, the  
root datum and
Weyl group of $H^\vee_{sph}$ 
are the
same as that associated to $X$ in the theory of symmetric varieties.
\end{corollary}

Corollary \ref{Weyl groups, intro} is restated in 
Theorem \ref{Weyl groups}. In \emph{loc. cit.} we also obtain a description 
of irreducible objects in $Q_K^{glob}$ confirming a conjecture in \cite{GN1}.

\subsubsection{Braidings}
We conclude the introduction with the following conjecture.
Note that Theorem \ref{main thm 1} implies that 
\begin{corollary}
There is an
equivalence
\beq\label{abelian equivalence}
\xymatrix{
\on{Perv}(X(\calK)/G(\calO))\ar[r]^\simeq&\on{Perv}(G_\bbR(\calO_\bbR)\backslash\Gr_\bbR)}
\eeq
of abelian categories.
\end{corollary}

We expect to prove the following.
\begin{conjecture}\label{conj}
(1)
There exists a natural geometric lift of the  fusion product $\star_f$ to make 
$\on{Perv}(X(\calK)/G(\calO))$ (resp. $\on{Perv}(G_\bbR(\calO_\bbR)\backslash\Gr_\bbR)$)
a braided monoidal category.
Moreover, there is an upgrade of~\eqref{abelian equivalence} to an 
equivalence
of  braided monoidal categories.

(2) The real weight functors in \cite{N2}
defines a fiber functor 
$\omega:\on{Perv}(G_\bbR(\calO_\bbR)\backslash\Gr_\bbR)\to\on{SVect}$
where $\on{SVect}$ is the tensor category of finite dimensional super vector spaces.
\end{conjecture}

\begin{remark}
(1) For (1) of the conjecture,
it is possible to produce the braided monoidal structure on $\on{Perv}(X(\calK)/G(\calO))$ (resp. $\on{Perv}(G_\bbR(\calO_\bbR)\backslash\Gr_\bbR)$) via  monodromy equivalences of nearby cycles (resp.   compatibilities between fusion and convolution products). What is less clear is how to upgrade~\eqref{abelian equivalence} to intertwine
these  braided monoidal structures.

(2) Unlike their  full subcategories 
$Q_\bbR$ and  $Q_K$ 
which are abelian semi-simple
and symmetric monoidal, the abelian categories $\on{Perv}(G_\bbR(\calO_\bbR)\backslash\Gr_\bbR)$ and $\on{Perv}(X(\calK)/G(\calO))$
in general are not semi-simple and we expect that the braiding is also not symmetric
in general.

(3) Via the Tannakian reconstruction for braided monoidal categories 
the conjecture would imply that there are equivalences $\on{Perv}(X(\calK)/G(\calO))\is\on{Perv}(G_\bbR(\calO_\bbR)\backslash\Gr_\bbR)\is\on{Rep}( H^\vee_q)$
where $\on{Rep}( H^\vee_q)$ is
the category of representations of a quantum supergroup (at a root of unity).
This suggests that the  dual groups of symmetric varieties
should be \emph{quantum supergroups}.

(4) We expect a version of Conjecture \ref{conj} for the whole 
real and relative Satake categories:   
both categories $D(X(\calK)/G(\calO))$ and $D(G_\bbR(\calO_\bbR)\backslash\Gr_\bbR)$
admit natural $E_2$-structures
(i.e., there are locally constant factorization categories  
on $\bbR^2$) and there is an upgrade of the real-symmetric equivalence 
in Theorem \ref{main thm 1}
to an equivalence of $E_2$-categories.
\end{remark}

\subsection{Organization}
In Section \ref{orbits}, we recall the parametrization of 
$K(\calK)$-orbits and $LG_\bbR$-orbits on the affine Grassmannian
and $G(\calO)$-orbits on $X(\calK)$.
In Section \ref{Morse flow}, we construct the 
Matsuki flow on the affine Grassmannian and we give a 
Morse-theoretic interpretation and refinement of the Matsuki correspondence for 
the affine Grassmannian. 
In Section \ref{Real BD}, we study real forms 
of Beilinson-Drinfeld Grassmannians. 
In Section \ref{Gram-Schmidt}, we study multi-point version of 
Gram-Schmidt factorization for loop groups.
In Section \ref{uniform},
we study uniformization of real bundles.
In Section \ref{QMaps}, we study moduli of quasi-maps 
and we prove a multi-point version of 
Quillen's homeomorphism.
In Section \ref{real sym}, we construct the real-symmetric equivalence
Theorem \ref{main thm 1}
In Section \ref{Affine Matsuki}, we prove the affine Matsuki correspondence for sheaves. 
In Section \ref{nearby cycles and Radon TF}, we prove 
Theorem \ref{main thm 2}.
In Section \ref{s:hecke}, we study the compatibility of Hecke actions.
In Section \ref{fusion}, we study the compatibility of fusion products.
In Section \ref{applications}, we study applications of 
main results to real and relative Langlands
duality.
In Appendix \ref{Real stacks}, we discuss semi-analytic stacks and categories of sheaves 
on semi-analytic stacks.

\subsection{Acknowledgements} 

The authors thank David Ben-Zvi, Alexis Bouthier, Pavel Etingof, Mark Macerato, John O'Brien, Yiannis Sakellaridis, 
Jeremy Taylor, Akshay Venkatesh, Ruotao Yang, and Lingfei Yi for many useful discussions.
T.-H.~Chen also thanks the Max Planck Institute for 
Mathematics and
D.~Nadler  the Miller Institute where parts of this
work were done. 
The research of
T.-H.~Chen is supported by NSF grant DMS-2143722
and that of D.~Nadler  by NSF grant DMS-2101466.

\section{Orbits on $\Gr$ and $X(\calK)$}\label{orbits}
In this section we study 
$K(\calK)$ and $LG_\bbR$-orbits on the affine Grassmannian $\Gr$
and $G(\calO)$-orbits on the formal loop space $X(\calK)$ of 
$X$.

\subsection{Loop groups}
 Let $\bbG_\mathbb R$ be a connected real reductive algebraic group, and
 $\bbG=\bbG_\bbR\otimes_\bbR\bbC$  its complexification.
From this starting point, one  constructs the following diagram of Lie groups
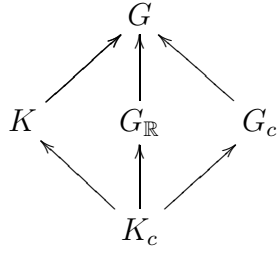
\beq\label{eq:diamond}
\xymatrix{&G&\\
K\ar[ru]\ar[ru]&G_\mathbb R\ar[u]&G_c\ar[lu]\\
&K_c\ar[lu]\ar[u]\ar[ru]&}
\eeq
Here $G=\bbG(\bC)$ and $G_\bbR=\bbG_\bbR(\bbR)$ are the Lie groups of complex and real points respectively, 
$K_c$ is a maximal compact subgroup of $G_\mathbb R$, with complexification $K$,
and $G_c$ is the maximal compact subgroup of $G$ containing $K_c$.

The real forms $G_\mathbb R$ and $G_c$ of $G$ correspond to 
anti-holomorphic involutions
$\eta$ and $\eta_c$. The involutions 
$\eta$ and $\eta_c$ commutes with each other and 
$\theta:=\eta\eta_c=\eta_c\eta$ is an involution of $G$. We have $K=G^\theta$, 
$G_\mathbb R=G^\eta$, and $G_c=G^{\eta_c}$.
We fix a maximal split tours $A_\mathbb R\subset G_\mathbb R$ and 
a maximal torus $T_\mathbb R$ such that $A_\mathbb R\subset T_\mathbb R$.
We write $A$ and $T$ for the complexification of $A_\mathbb R$ and $T_\mathbb R$.
We denote by $\Lambda_T$  the lattice of coweights of $T$ and 
$\Lambda_A$ the lattice of real coweights. We write $\Lambda_T^+$
the set of dominant coweight with respect to the Borel subgroup $B$ and 
define $\Lambda_A^+:=\Lambda_A\cap\Lambda_T^+$.
For any $\lambda\in\Lambda_T$ we define $\eta(\lambda)\in\Lambda_T$ as
\[\eta(\lambda):\bC^\times\stackrel{c}\ra\bC^\times\stackrel{\lambda}\ra T\stackrel{\eta}\ra T,\]
where $c$ is the complex conjugation of $\bC^\times$ with respect to $\mathbb R^\times$.
The assignment $\lambda\ra\eta(\lambda)$ defines an involution on $\Lambda_T$, which 
we denote by $\eta$,
and $\Lambda_A$ is the fixed points of $\eta$.

We have a natural projection map
\beq\label{sigma}
\sigma:\Lambda_T\to\Lambda_A  \ \ \ \sigma(\lambda)=\eta(\sigma)+\sigma
\eeq
whose image we denote by $\sigma(\Lambda_T)\subset\Lambda_A$.

Let 
$LG:=G(\bC[t,t^{-1}])$ be the (polynomial) loop group associated to $G$. We define 
the following involutions on $LG$:
for any $(\gamma:\bC^\times\ra G)\in LG$ we set
\[\eta^\tau(\gamma):\bC^\times\stackrel{\tau}\ra\bC^\times\stackrel{c}\ra\bC^\times\stackrel{\gamma} \ra G\stackrel{\eta}\ra G\]
\[\eta^\tau_c(\gamma):\bC^\times\stackrel{\tau}\ra\bC^\times\stackrel{c}\ra\bC^\times\stackrel{\gamma} \ra G\stackrel{\eta_c}\ra G.\]
Here $\tau(x)=x^{-1}$ is the the inverse map.
Denote by $\calK=\bC((t))$ and $\mO=\bC[[t]]$.
We have the following diagram
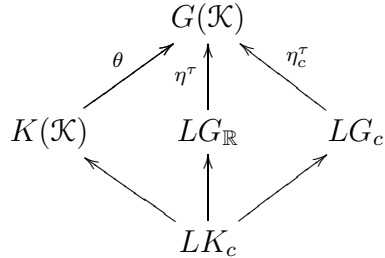
\[
\xymatrix{&G(\calK)&\\
K(\calK)\ar[ru]^\theta\ar[ru]&LG_\mathbb R\ar[u]^{\eta^\tau}&LG_c\ar[lu]_{\eta_c^\tau}\\
&LK_c\ar[lu]\ar[u]\ar[ru]&}
\]
Here $LG_\mathbb R$ and $LG_c$ are the fixed points subgroups of the 
involutions
$\eta^\tau$ and $\eta^{\tau}_c$ on $LG$ respectively.
Equivalently, $LG_\mathbb R$ (resp. $LG_c$) is the 
subgroup of $LG$ consisting of maps that take the 
unit circle $S^1\subset\bC$ to $G_\bbR$ (resp. $G_c$).
We define the based loop group $\Omega G_c$ to be the subgroup of 
$LG_c$ consisting of maps that take $1\in S^1$ to $e\in G_c$.

We define $G_{sym}\subset G$
(resp. $G_{c,sym}\subset G_c$) to be the  
fixed point subspace of the involution $\tilde\theta=\theta^{-1}$ on $G$ (resp. $G_c$)
and write $G_{sym}^0$ (resp. $G_{c,sym}^0$) be the identity component.
The map $\pi:G\to G, g\to \tilde\theta(g)g$ induces a $G$-equivariant isomorphism 
$\iota:X\is G_{sym}^0$ (resp. $G_c$-equivariant isomorphism $\iota_c:X_c\is G^0_{c,sym}$).
We have a natural embedding $A\hookrightarrow G^0_{sym}\is X$.
Let $q:G\to K\backslash G=X$ be the quotient map.

In this paper, we assume $G_{\bbR}$ (and equivalently $K$) is connected.
Throughout this paper, we will be concerned exclusively with the topology of loop spaces
and related moduli and ignore their potentially non-reduced structure.

\subsection{Parametrization of orbits}
Let $\Gr=G(\calK)/G(\calO)$ be the affine Grassmannian 
of $G$.
We recall  
results from \cite{N2} about the parametrization of 
$K(\calK)$ and $LG_\bbR$-orbits on 
$\Gr$ and $G(\calO)$-orbits on $X(\calK)$.

\begin{proposition}\label{parametrization}
We have the following.
\begin{enumerate}
\item
There is a bijection 
\[|X(\calK)/G(\calO)|\longleftrightarrow\Lambda_A^+\]
between $G(\calO)$-orbits on $X(\calK)$ and $\Lambda_A^+$
characterized by the following property.
Let $X(\calK)^\lambda$ be the $G(\calO)$-orbits corresponding to $\lambda\in\Lambda_A^+$.
Then for any $\gamma\in X(\calK)^\lambda$ we have  
$\pi(\gamma)=\tilde\theta(\gamma)\gamma\in G(\calO)t^\lambda G(\calO)$.
In addition, we have 
$\overline X(\calK)^\lambda=\bigsqcup_{\mu\leq\lambda}X(\calK)^\mu$.

\item There is a bijection \[|K(\mathcal K)\backslash\Gr|\longleftrightarrow\Lambda_A^+\]
between $K(\mathcal K)$-orbits on $\Gr$ and $\Lambda_A^+$
characterized by the following properties:
Let $\mO_K^\lambda$ be the $K(\mathcal K)$-orbits corresponding to $\lambda\in\Lambda_A^+$.
Then for any $\gamma\in\mO_K^\lambda$, thought of as an element in 
$\Omega G_c$, satisfies  
$\tilde\theta(\gamma)\gamma\in G(\bC[t])t^\lambda G(\bC[t])$.
In addition, we have 
$\overline\mO_K^\lambda=\bigsqcup_{\mu\leq\lambda}\mO_K^\mu$.

\item There is a bijection \[|LG_\bbR\backslash\Gr|\longleftrightarrow\Lambda_A^+\]
between $LG_\bbR$-orbits on $\Gr$ and $\Lambda_A^+$
characterized by the following property:
Let $\mO_\bbR^\lambda$ be the $LG_\bbR$-orbits corresponding to $\lambda\in\Lambda_A^+$.
Then for any $\gamma\in\mO_\bbR^\lambda$, thought of as an element in 
$\Omega G_c$, satisfies
$\tilde\eta^\tau(\gamma)\gamma\in G(\bC[t^{-1}])t^\lambda G(\bC[t])$.
In addition, we have 
$\overline\mO_\bbR^\lambda=\bigsqcup_{\lambda\leq\mu}\mO_\bbR^\mu$.

\item
The projection  $G\to X$ induces a bijection
 map between the set of orbits
\[ |K(\calK)\backslash\Gr|\to |X(\calK)/G(\calO)|\]
and under the bijection in (1) and (2) it is  
equal to 
the identity map $\Lambda_A^+\to\Lambda_A^+$.

\item The correspondence 
\beq\label{AMS correspondence}
|K(\mathcal K)\backslash\Gr|\longleftrightarrow |LG_\bbR\backslash\Gr|,\ \ \mO_K^\lambda
\longleftrightarrow\mO_\bbR^\lambda
\eeq
provides an order-reversing 
isomorphism from the poset $|K(\mathcal K)\backslash\Gr|$ to the poset $|LG_\bbR\backslash\Gr|$
(with respect to the closure ordering). In addition, for each 
$K(\mathcal K)$-orbit $\mO_K^\lambda$, $\mO_\bbR^\lambda$ is the 
unique $LG_\bbR$-orbit such that \[\mO_c^\lambda:=\mO_K^\lambda\cap\mO_\bbR^\lambda\]
is a single $LK_c$-orbit.

\end{enumerate}
\end{proposition}
\begin{proof}
Consider the maps $\label{pi_1}
\pi_1(G)\stackrel{q_*}\ra\pi_1(X)\stackrel{[-]}\la\Lambda_A^+$
where the first map is that induced by the projection $q:G\ra X$ and 
the second map $[-]$ assigns to a loop its homotopy class. 
We define $\calL\subset\Lambda_A^+$ to be the 
inverse image of $q_*(\pi_1(G))$ along the map
$[-]$.
According to \cite{N1},
the set of orbits in (2) and (3) are parametrized by the 
subset  $\calL\in\Lambda_A^+$.
When $K$ is connected we have $\calL=\Lambda_A^+$ and the proposition follows from
\cite[Theorem 4.2 and Theorem 10.1]{N1}.
\end{proof}

We will call~\eqref{AMS correspondence} the Affine Matsuki correspondence.
This correspondence can be viewed as an affine version of the 
classical Matskki correspondence for flag manifolds in \cite{MUV}.


\section{The Matsuki flow}\label{Morse flow}
In this section we construct a Morse flow on the 
affine Grassmannian, called the Matsuki flow, and we use it to give a 
Morse-theoretic interpretation and refinement of the 
affine 
Matsuki correspondence.

\subsection{Polynomial Loop 
spaces of $X_c$}\label{based for X_c}
Let $X_c=K_c\backslash G_c$ be the compact symmetric space.
Let  $LX=\on{X}(\bbC[t,t^{-1}])$ be the space of (polynomial) loop space of 
$X$ and  
$LX_c\subset LX$ be 
the subspace of 
$LX$ consisting of 
maps that takes $S^1$ into $X_c$. We define the based loop space 
$\Omega X_c$ to be the subspace of 
$LX_c$ consisting of maps that takes $1\in S^1$ to $e\in X_c$.
The embedding  
$X_c\is G_{c,sym}^0\subset G_c$ induces an $K_c$-equivaraint isomorphism 
$\Omega X_c\is\Omega G_{c,sym}^0\is (\Omega G_c)^{\tilde\theta}$.

\quash{
The natural inclusion 
$S\to X$ induces an imbedding 
$\Lambda_S^+\hookrightarrow X(\calK)$.

\begin{proposition}\label{parametrization of K(K)-orbits}
The embedding $\Lambda_S^+\hookrightarrow X(\calK)$ 
induces an bijection 
\[\Lambda_S^+\longleftrightarrow |X(\mathcal K)/G(\mO)|\]
between $\Lambda_S^+$ and $G(\mO)$-orbits in $X(\mathcal K)$
and, under the above bijection, 
 the image of the natural inclusion \[|K(\calK)\backslash\Gr|\to| X(\mathcal K)/G(\mO)|\] is the subset $\mL\subset\Lambda_S^+$.
\end{proposition}
}

\subsection{Geometry of orbits}
We embed 
$\Omega X_c\subset \Omega G_c$ via the 
embedding $X_c\stackrel{\iota_c}\is G_{c,sym}^0\subset G_c$.
For $\lambda\in\Lambda_T^+$, we define $P^\lambda\subset\Omega X_c$ to be the intersection 
of $\Omega X_c$ with the orbit $S^\lambda\subset\Omega G_c\is\Gr$, and we define 
$Q^\lambda\subset\Omega X_c$ to be the intersection 
of $\Omega X_c$ with the orbit $T^\lambda\subset\Omega G_c\is\Gr$. We define 
$B^\lambda$ to be the intersection 
of $\Omega X_c$ with $C^\lambda\subset\Omega G_c\is\Gr$.
The projection map $G\to X$ induces 
a projection $\Omega G_c\to\Omega X_c$.

\begin{lemma}\label{component X_c}
We have $\Omega X_c=\bigcup_{\lambda\in\Lambda_A^+} P^\lambda$. Equivalently,
$P^\lambda$ is non-empty if and only if $\lambda\in\Lambda_A^+$.

\end{lemma}
\begin{proof}
The  isomorphism $\iota:X\is G_{sym}^0$ induces 
an 
embedding $X(\calK)\subset G(\calK)$
and it follows from 
Proposition \ref{parametrization} (1)  that 
$X(\calK)^\lambda:=X(\calK)\cap G(\calO)t^\lambda G(\calO)$ is nonempty if and only if 
$\lambda\in\Lambda_A^+$.
Since $P^\lambda=\Omega X_c\cap X(\calK)^\lambda$,
it implies $P^\lambda$ is non-empty if and only if $\lambda\in\Lambda_A^+$
and we conclude that $\Omega X_c=\bigcup_{\lambda\in\Lambda_T^+} P^{\lambda}
=\bigcup_{\lambda\in\lambda_A^+} P^{\lambda}$.

\end{proof}

\begin{proposition}\label{torsors}
We have the following.
\begin{enumerate}
\item
The projection $q:\Omega G_c\to\Omega X_c$ maps $\mO_K^\lambda$ into 
$P^\lambda$ and the resulting map 
$\mO_K^\lambda\to P^\lambda$ 
induces a homeomorphism
 $\Omega K_c\backslash\mO_K^\lambda\is P^\lambda$.
\item
The projection $q:\Omega G_c\to\Omega X_c$ maps $\mO_\bbR^\lambda$ into 
$Q^\lambda$ and the resulting map 
$\mO_\bbR^\lambda\to Q^\lambda$ induces a homeomorphism 
$\Omega K_c\backslash\mO_\bbR^\lambda\is Q^\lambda$
\item 
There is a $K_c$-equivariant stratified homeomorphism 
$\Omega K_c\backslash\Omega G_c\is\Omega X_{c}$.
\end{enumerate}
\end{proposition}
\begin{proof}
Fix $\lambda\in\Lambda_A^+$.
Proposition \ref{parametrization} together with the fact that 
$\tilde\theta=\tilde\eta^\tau$ on $\Omega G_c$
imply  
$q(\mO_K^\lambda)\subset P^\lambda$ and $q(\mO_\bbR^\lambda)\subset Q^\lambda$.
By \cite[Proposition 6.3]{N1}, the restriction of the energy function $E$ to $P^\lambda$ is Bott-Morse
and $B^\lambda$ is the only critical manifold. 
Since $K_c$ is connected and 
acts transitively on $B^\lambda$, it follows that $B^\lambda$ and hence $P^\lambda$ connected.
Now  \cite[Proposition 6.4]{N2} implies
$P^\lambda\subset q(\Omega G_c)$. 
Thus 
$q(\mO_K^\lambda)=P^\lambda$ and 
part (1) follows. 
For part (2) 
we observe that  $Q^\lambda=\bigcup_{\lambda\leq\mu,\mu\in\Lambda_S^+} Q^\lambda\cap P^\mu$. 
Since $B^\lambda=Q^\lambda\cap P^\lambda$ is in the closure of $Q^\lambda\cap P^\mu$,
Lemma \ref{component X_c} implies 
$Q^\lambda=\bigcup_{\lambda\leq\mu,\mu\in\Lambda_A^+} Q^\lambda\cap P^\mu$ and 
part (1) implies $Q^\lambda\subset q(\Omega G_c)$, hence 
$Q^\lambda=q(\mO_{\bbR}^\lambda)$. Part (2) follow. 
Part (3) follows from part (1) and Lemma \ref{component X_c}.

\end{proof}

\begin{corollary}\label{transversal}
$K(\calK)$ and $LG_\bbR$-orbits on $\Gr$ are transversal. 
\end{corollary}
\begin{proof}
By Proposition \ref{torsors}, it suffices to show that the strata
$P^\lambda$ and $Q^\mu$ in $\Omega X_c$ are transversal. 
This follows from the fact that the orbits 
$S^\lambda$ and $T^\lambda$ 
on $\Omega G_c$ 
are transversal and both $S^\lambda, T^\lambda$ are invariant under the 
involution $\tilde\theta$ on $\Omega G_c$ as 
$\tilde\theta=\tilde\eta^\tau$ on $\Omega G_c$ and 
$S^\lambda$ (resp. $T^\lambda$) is $\tilde\theta$-invariant 
(resp. $\tilde\eta^\tau$-invariant). 

\end{proof}

\subsection{The energy flow on $\Omega G_c$}\label{energy flow}
We recall the construction of energy flow on $\Omega G_c$ following 
\cite[Section 8.9]{PS}.
For any $\gamma\in LG_c$ and $v\in T_\gamma LG_c$ we denote by 
$\gamma^{-1}v\in L\fg_c$ (resp. $v\gamma^{-1}\in L\fg_c$) the
image of $v\in T_\gamma LG_c$ under the 
isomorphism $T_\gamma LG_c\is T_eLG_c\is L\fg_c$
induced by the left action (resp. right action). 

Fix a $G_c$-invariant metric $\langle,\rangle$ on $\fg_c$.
Observe that the formula 
\[\omega(v,w):=\int_{S^1} \langle(\gamma^{-1}v)',\gamma^{-1}w\rangle d\theta
\]
defines a left invariant symplectic form on $T_\gamma\Omega G_c$. 
According to \cite[Theorem 8.6.2]{PS}, the composition 
$\Omega G_c\to G(\mathcal K)\to\Gr$ defines a diffeomorphism
\[
\Omega G_c\is\Gr.
\]
Let $J_\gamma$ be the automorphism of $T_\gamma\Omega G_c$ which 
corresponds to multiplication by $i$ in terms of the complex structure on 
$\Gr$. The formula 
$g(v,w)=\omega(v,J_\gamma w)$
defines a positive inner product on $T_\gamma\Omega G_c$ and 
the K$\ddot{\on{a}}$hler form on $T_\gamma\Omega G_c$ is 
given by $g(v,w)+i\omega(v,w)$.
Finally, for any smooth function $F:\Omega G_c\ra\bbR$ there corresponds 
so-called Hamiltonian vector field $R(\gamma)$ and 
gradient vector field $\nabla F(\gamma)$ on $\Omega G_c$ characterized by 
\[\omega(R(\gamma),v)=dF(\gamma)(u),\ g(\nabla F(\gamma),u)=
dF(\gamma)(u).\]
Consider the energy function on $\Omega G_c$:
\beq\label{energy function}
E:\Omega G_c\ra\bbR,\ \gamma\ra (\gamma',\gamma')_\gamma=\int_{S^1}\langle\gamma^{-1}\gamma',\gamma^{-1}\gamma'\rangle d\theta.
\eeq

We have the following well-known facts.
\begin{prop}\cite{P,PS}\label{PS}
\begin{enumerate}
\item The Hamiltonian vector field of $E$ 
is equal to the vector field induced by the rotation flow 
$\gamma_a(t)=\gamma(t+a)\gamma(a)^{-1}$ and 
is given by 
$\gamma\ra R(\gamma)=\gamma'-\gamma\gamma'(0)$. The gradient vector field of $E$ is equal to 
$\nabla E=-J\circ R$. 
\item 
The critical locus $\nabla E=0$ is the disjoint union 
$\bigsqcup_{\lambda\in\Lambda_T^+}C^\lambda$ of $G_c$-orbits of 
$\lambda\in\Omega G_c$.
\item
The gradient flow $\psi_t$ of $\nabla E$  
preserves the orbits $S^\lambda$ and $T^\lambda$. For each 
critical orbit $C^\lambda$, we have 
 \[S^\lambda=\{\gamma\in\Omega G_c|\underset{t\to\infty}{\lim}\psi_t(\gamma)\in C^\lambda\} \ \ \ \ T^\lambda=\{\gamma\in\Omega G_c|\underset{t\ra-\infty}\lim_{}\psi_t(\gamma)\in C^\lambda\}.\]
That is $S^\lambda$ and $T^\lambda$ are the stable and unstable 
manifold of $C^\lambda$.
\end{enumerate}
\end{prop}

\begin{lemma}\label{rotation}
The $K(\calK)$-orbits and $LG_\bbR$-orbits are stable under the 
rotation flow $\gamma_a(t)$ (see Proposition \ref{PS}).
\end{lemma}
\begin{proof}
We give a proof for the case of $K(\calK)$-orbits.
The proof for the $LG_\bbR$-orbits is similar.
Let $\mO_K^\lambda$ be a $K(\calK)$-orbit and let 
$\gamma=\gamma(t)\in\mO_K^\lambda$. 
By Proposition \ref{parametrization}, we need to show that 
$\tilde\theta(\gamma_a)\gamma_a\in G(\bC[t])t^\lambda G(\bC[t])$.
A direct computation shows that 
$\tilde\theta(\gamma_a)\gamma_a=\theta(\gamma(a))\tilde\theta(\gamma(t+a))\gamma(t+a)\gamma(a)^{-1}$. Note that  $\tilde\theta(\gamma(t+a))\gamma(t+a)\in G(\bC[t])t^\lambda G(\bC[t])$ as 
$\gamma(t)\in\mO_K^\lambda$, the desired claim follows.

\end{proof}

\subsection{The Matsuki flow on $\Gr$}\label{Matsuki flow}
The Cartan decomposition $\fg_\bbR=\frak k_\bbR\oplus\fp_\bbR$ 
induces a decomposition of
$\fg_c=\frak k_c\oplus i\fp_\bbR$, $\fg_\bbR=\frak k_c\oplus\fp_\bbR$
and the corresponding loop algebra 
$L\fg=L\frak k\oplus L\fp$, $L\fg_c=L\frak k_c\oplus L(i\fp_\bbR)$,
$L\fg_\bbR=L\frak k_c\oplus L\fp_\bbR$.

Recall the non-degenerate bilinear form $(,)_\gamma$ on 
$T_\gamma LG_c$ 
\[(v_1,v_2)_\gamma:=\int_{S^1}\langle\gamma^{-1}v_1,\gamma^{-1}v_2\rangle d\theta.\]
Let $\gamma\in LG_c$ and $T_\gamma (LK_c\cdot\gamma)\subset T_\gamma LG_c$ be the 
tangent space of the $LK_c$-orbit $LK_c\cdot\gamma$ through $\gamma$. 
The bilinear form above induces an 
orthogonal decomposition 
\[T_\gamma LG_c=T_\gamma LK_c\cdot\gamma\oplus (T_\gamma LK_c\cdot\gamma)^\bot\] and 
for any vector $v\in T_\gamma LG_c$ we write 
$v=v_0\oplus v_1$ where $v_0\in T_\gamma LK_c\cdot\gamma$, 
$v_1\in (T_\gamma LK_c\cdot\gamma)^\bot$.
 Note that we have 
\beq\label{decomp}
\gamma^{-1}v_0\in\on{Ad}_{\gamma^{-1}}L\frak k_c,\ \ 
\gamma^{-1}v_1\in\on{Ad}_{\gamma^{-1}}L(i\fp_\bbR).
\eeq
Recall that the loop group $\Omega G_c$ can be identified with a ``co-adjoint'' orbit
in $LG_c$ via the embedding
\[\Omega G_c\hookrightarrow L\fg_c,\ \gamma\ra\gamma^{-1}\gamma'.\]
Consider the following  functions on $\Omega G_c$
\[E:\Omega G_c\ra\bbR,\ \gamma\ra (\gamma',\gamma')_\gamma=\int_{S^1}\langle\gamma^{-1}\gamma',\gamma^{-1}\gamma'\rangle d\theta,\] 
\[E_0:\Omega G_c\ra\bbR,\ \gamma\ra
(\gamma'_0,\gamma'_0)_\gamma=
\int_{S^1}\langle\gamma^{-1}\gamma'_0,\gamma^{-1}\gamma'_0\rangle d\theta,\]
\[E_1:\Omega G_c\ra\bbR,\ \gamma\ra
(\gamma'_1,\gamma'_1)_\gamma=
\int_{S^1}\langle\gamma^{-1}\gamma'_1,\gamma^{-1}\gamma'_1\rangle d\theta.\]

\quash{
\begin{lemma}
We have  
$E=E_0+E_1$ and $E_1$ is constant on $LK_c$-orbits on 
$\Omega G_c=\Gr$.
\end{lemma}
\begin{proof}
The first claim follows from the definition and the second claim follows from the equality 
$(k\gamma)^{-1}(k\gamma)'_1=
\gamma^{-1}\gamma'_1
$, for $k\in LK_c,\gamma\in LG_c$.
\end{proof}

\begin{remark}
The functions $E$ and $E_0$ are not $LK_c$-invariant.
\end{remark}

\begin{lemma}
The Hamiltonian vector field on $\Omega G_c$ which 
corresponds to $E_1$ is given by 
\[\gamma\ra (\gamma^{-1}\gamma')_1+[(\gamma^{-1}\gamma')_1,
(\gamma^{-1}\gamma')_0]\in L\fg_c\on{mod}\fg_c\is T_\gamma\Omega G_c.\]

\end{lemma}}
Note that $E$ is the energy function in \eqref{energy function}.

\begin{lemma}\label{4}
Recall the map $\pi:\Omega G_c\ra\Omega G_c, \gamma\ra\theta(\gamma)^{-1}\gamma$.
We have 
\beq
4E_1=E\circ\pi:\Omega G_c\ra\bbR.
\eeq
In particular, the function $E_1$ is $LK_c$-invariant.

\end{lemma}
\begin{proof}
Write $||v||=\langle v,v\rangle$ for $v\in \fg_c$.
For any $\gamma\in\Omega G_c$ we have 
\[E\circ\pi(\gamma)=\int_{S^1}||\pi(\gamma)^{-1}\pi(\gamma)'||d\theta=
\int_{S^1}||\gamma^{-1}\gamma'-\gamma^{-1}\eta(\gamma)'\eta(\gamma)^{-1}\gamma|| d\theta.\]
Note that $\gamma^{-1}\gamma'-\gamma^{-1}\theta(\gamma)'\theta(\gamma)^{-1}\gamma=2\gamma^{-1}\gamma'_1$
, hence we have $||\gamma^{-1}\gamma'-\gamma^{-1}\theta(\gamma)'\theta(\gamma)^{-1}\gamma||=4||\gamma^{-1}\gamma'_1||$.
The lemma follows.

\quash{
\[\int||\gamma^{-1}\gamma'||+||\gamma^{-1}\eta(\gamma'\gamma^{-1})\gamma||
-2(\gamma^{-1}\gamma',\gamma^{-1}\eta(\gamma'\gamma^{-1})\gamma)_{\on{kil}}d\theta.\]
Since $(,)_{\on{kill}}$ is $G_c$ and $\eta$-invariant,
we have 
\[||\gamma^{-1}\gamma'||+||\gamma^{-1}\eta(\gamma'\gamma^{-1})\gamma||=
||\gamma^{-1}\gamma'||+||\eta(\gamma'\gamma^{-1})||=2||\gamma^{-1}\gamma'||.\]
On the other hand, it follows from (\ref{decomp}) that   
\[(\gamma^{-1}\gamma',\gamma^{-1}\eta(\gamma'\gamma^{-1})\gamma)_{\on{kil}}=
||\gamma^{-1}\gamma'_0||-
||\gamma^{-1}\gamma'_1||.\]
All together, we arrive 
\[E\circ\pi(\gamma)=\int2||\gamma^{-1}\gamma'||-2(
||\gamma^{-1}\gamma'_0||-
||\gamma^{-1}\gamma'_1||)d\theta=\int 4||\gamma^{-1}\gamma'_1||d\theta
=4E_1(\gamma).\]}

\end{proof}

\begin{lemma}\label{hamiltonian}
The Hamiltonian vector field on $\Omega G_c$ which 
correspond to $E_1$ (resp. $E_0$) is given by 
\[\gamma\ra R_1(\gamma)=\gamma'_1-\gamma\gamma_1'(0)\ \ (resp.\  \gamma\ra R_0=\gamma'_0-\gamma\gamma_0'(0)).\] 
In particular, we have 
\[\gamma^{-1}R_1(\gamma)\in\on{Ad}_{\gamma^{-1}}Li\frak p_\bbR+i\frak p_\bbR
\ \ (resp.\  
\gamma^{-1}R_0(\gamma)\in\on{Ad}_{\gamma^{-1}}L\frak k_\bbR+\frak k_\bbR).\]

\end{lemma}
\begin{proof}
Since $R_0(\gamma)+R_1(\gamma)=R(\gamma)=\gamma'-\gamma\gamma(0)'$,
it is enough to show that $R_1(\gamma)=\gamma'_1-\gamma\gamma_1'(0)$.
Let $\gamma\in\Omega G_c$, $x=\pi(\gamma)=\theta(\gamma)^{-1}\gamma$,
and $u\in T_\gamma\Omega G_c$. 
According to Proposition \ref{PS} and Lemma \ref{4},
we have 
\[4dE_1(\gamma)(u)\stackrel{}=\pi^*dE(\gamma)(u)=
dE(x)(\pi_*u)=
\omega(x',\pi_*u)=\omega(x^{-1}x',x^{-1}\pi_*u).\]
Using the equalities  $x^{-1}x'=2\gamma^{-1}\gamma'_1$,  
$x^{-1}\pi_*u=2\gamma^{-1}u_1$, and the fact that $\langle\gamma^{-1}\gamma'_1,(\gamma^{-1}u_0)'\rangle=0$, we get 
 \[4dE_1(\gamma)(u)=4\omega(
\gamma^{-1}\gamma'_1,\gamma^{-1}u_1)=4\int_{S^1} \langle
\gamma^{-1}\gamma'_1,(\gamma^{-1}u_1)'\rangle d\theta=4\int_{S^1} \langle
\gamma^{-1}\gamma'_1,(\gamma^{-1}u)'\rangle d\theta\]
\[=4\int_{S^1} \langle
\gamma^{-1}\gamma'_1-\gamma_1'(0),(\gamma^{-1}u)'\rangle d\theta=
4\omega(R_1(\gamma),u).\]
The lemma follows.

\end{proof}

Let $\Omega G_c=\bigcup_{\lambda\in\Lambda_A^+}\mO_K^\lambda$ and 
$\Omega G_c=\bigcup_{\lambda\in\Lambda_A^+}\mO_\bbR^\lambda$ be the $K(\calK)$-orbits and 
$LG_\bbR$-orbits stratifications of $\Omega G_c$. 
Let $\mO_c^\lambda=\mO_K^\lambda\cap\mO_\bbR^\lambda$ which is a single 
$LK_c$-orbit. 

\begin{prop}\label{Morse-Bott}
Let $E_1:\Omega G_c\ra\bbR$ be the function above and  $\nabla E_1$
be the corresponding gradient vector field. 
\begin{enumerate}
\item $\nabla E_1$ is tangential to both 
$\mO^\lambda_K$ and  $\mO^\lambda_\bbR$,
\item The union $\bigsqcup_{\lambda\in\Lambda_A^+}\mO_c^\lambda$ is the critical manifold of $\nabla E_1$.
\item For any $\gamma\in\mO_c^\lambda$, let 
$T_\gamma\Omega G_c=T^+\oplus T^0\oplus T^-$
be the orthogonal direct sum decomposition into the positive, zero, and 
negative eigenspaces of the Hessian $d^2E_1$. We have 
\[T_\gamma\mO_K^\lambda=T^+\oplus T^0,\ T_\gamma\mO_\bbR^\lambda=
T^-\oplus T^0.\]
\end{enumerate}

\end{prop}
\begin{proof}
Proof of (1). 
We first show that $\nabla E_1$ is tangential to $\mO_\bbR^\lambda=\Omega G_c\cap
LG_\bbR t^\lambda G(\bC[t])$.
Since the tangent space $T_\gamma\mO_\bbR^\lambda$ at $\gamma\in\mO_\bbR^\lambda$
is identified, by 
left translation, with the space
\[\Omega\fg_c\cap(\on{Ad}_{\gamma^{-1}}L\fg_\bbR+\fg(\bC[t]))\subset \Omega\fg_c\] it suffices to show that 
$\gamma^{-1}\nabla E_1(\gamma)\in\on{Ad}_{\gamma^{-1}}L\fg_\bbR+\fg(\bC[t])$.
Recall that, by Proposition \ref{PS}, we have $\gamma^{-1}\nabla E_1(\gamma)=J(\gamma^{-1}R_1(\gamma))
$.
Note that $J(v)+iv\in\fg(\bC[t])$ for $v\in L\fg$ and by Lemma \ref{hamiltonian} we have 
\[i\gamma^{-1}R_1(\gamma)=i(\gamma^{-1}(\gamma'_1-\gamma\gamma_1'(0))
\in\on{Ad}_{\gamma^{-1}}L\fp_\bbR+\fp_\bbR.\]  
All together, 
we get 
\[\gamma^{-1}\nabla E_1(\gamma)=
-i\gamma^{-1}R_1(\gamma)+(J(\gamma^{-1}R_1(\gamma))+
i\gamma^{-1}R_1(\gamma))\in\on{Ad}_{\gamma^{-1}}L\fp_\bbR+\fg(\bC[t])\]
which is contained in $\on{Ad}_{\gamma^{-1}}L\fg_\bbR+\fg(\bC[t])$. We are done.  The same argument as above, replacing $LG_\bbR$ by
$K(\calK)$, shows that the gradient field 
$\nabla E_0$ of $E_0$ is tangential to $\mO_K^\lambda$.
Since, by Corollary \ref{rotation}, the orbit $\mO_K^\lambda$ is a complex submanifold of $\Omega G_c=\Gr$
invariant under the 
rotation flow $\gamma_a(t)$, it follows from Proposition \ref{PS} that 
$\nabla E$ is tangential to $\mO_K^\lambda$. Since 
$\nabla E_1=\nabla E-\nabla E_0$, we conclude that 
$\nabla E_1$ is also tangential to $\mO_K^\lambda$. This finishes 
the proof of (1).

Proof of (2) and (3). 
By proposition \ref{torsors} and lemma \ref{4}, the function 
$E_1$ factors as 
\[E_1:\Omega G_c\stackrel{\pi}\ra
\Omega X_c \subset\Omega G_c\stackrel{E}\ra\bbR.\]
Thus to  
prove (2) and (3), it is enough to prove following:
\begin{enumerate}[(i)]
\item 
The union $\bigsqcup_{\lambda} B^\lambda$ is the critical manifold 
of the restriction $E$ to $\Omega X_c$,
\\
\item 
For $\gamma\in B^\lambda$
we have $T_\gamma P^\lambda=W^+\oplus W^0$, 
$T_\gamma Q^\lambda=W^-\oplus W^0$, where 
$T_\gamma\Omega X_c=W^+\oplus W^0\oplus W^-$ is the 
orthogonal direct sum decomposition into the positive, zero, and negative eigenspaces of the Hessian 
$E|_{\Omega X_c}$.
\end{enumerate}

By Proposition \ref{PS}, we have 
$T_\gamma S^\lambda=U^+\oplus U^0$, 
$T_\gamma T^\lambda=U^-\oplus U^0$, where 
$T_\gamma\Omega G_c=U^+\oplus U^0\oplus U^-$
is the orthogonal direct sum decomposition into the positive, zero, and negative eigenspaces of the Hessian 
$E$. Note that $\tilde\theta$ induces a linear map on 
$T_\gamma\Omega G_c$, which we still denoted by $\tilde\theta$,
and we have $T_\gamma\Omega X_c=
(T_\gamma\Omega G_c)^{\tilde\theta}$ is the fixed point subspace.
So to prove (i) and (ii) it suffices to show that the subspaces 
$T_\gamma S^\lambda$ and $T_\gamma T^\lambda$ are 
$\tilde\theta$-invariant. It is true, since 
$\tilde\theta=\tilde\eta^\tau$ on $\Omega G_c$ and 
$S^\lambda$ (resp. $T^\lambda$) is $\tilde\theta$-invariant 
(resp. $\tilde\eta^\tau$-invariant). 
This finished the proof of (2) and (3).

\end{proof}

\quash{
\begin{remark}
It is worth comparing this with the finite dimensional situation in \cite{MUV}. 
The functions $E_0,E_1$ here are the affine analogy of the Morse-Bott functions
$f^+,f^-:\calB\to\bbR$ in \cite[Section 3]{MUV}. 
There the sum $f=f^++f^-$ is a constant function on $\calB$ 
\end{remark}
}

\begin{thm}\label{flow}
The gradient $\nabla E_1$ and gradient-flow
$\phi_t$ associated to the $LK_c$-invariant function 
$E_1:\Gr\to\bbR$ and the $LG_c$-invariant metric 
$g(,)$ satisfy the following:
\begin{enumerate}
\item The critical locus $\nabla E_1 = 0$ is the disjoint union of 
$LK_c$-orbits $\bigsqcup_{\lambda\in\Lambda_A^+}\mO_c^\lambda$
\item
The gradient-flow $\phi_t$
preserves the $K(\calK)$-and $LG_\bbR$-orbits.
\item
The limits $\underset{t\ra\pm\infty}\lim\phi_t(\gamma)$ of the gradient-flow exist for any $\gamma\in\Gr$.
For each  $LK_c$-orbit $\mO_c^\lambda$ in the critical locus,
the stable and unstable sets 
\beq\label{eq:morse slows}
\xymatrix{
\mO^\lambda_K=\{\gamma\in\Gr|\underset{t\ra\infty}\lim_{}\phi_t(\gamma)\in\mO^\lambda_c\}
&
\mO^\lambda_\bbR=\{\gamma\in\Gr|\underset{t\ra-\infty}\lim_{}\phi_t(\gamma)\in\mO^\lambda_c\}
}\eeq
are a single $K(\calK)$-orbit and $LG_\bbR$-orbit respectively.
\item
The correspondence between orbits $\mO^\lambda_K\longleftrightarrow\mO^\lambda_\bbR$ defined by \eqref{eq:morse slows} recovers the affine Matsuki correspondence~\eqref{AMS correspondence}.

\end{enumerate}

\end{thm}

\begin{proof}
Part (1) and (2) follows from Proposition \ref{Morse-Bott}.
The $LK_c$-invariant function $E_1$, respectivley the $LG_c$-invariant metric $g(,)$, and the flow $\phi_t$, descends to a $K_c$-invariant Morse-Bott function
$\underline E_1:\Omega K_c\backslash\Gr\to\bbR$, respectivley a  
$K_c$-invariant metric $\underline g(,)$ on $\Omega K_c\backslash\Gr$,
and a flow $\underline\phi_t$. 
Since the function $\underline E_1$ is bounded below and 
the quotient $\Omega K_c\backslash\mO_K^\lambda$ is finite dimensional with $\overline{\Omega K_c\backslash\mO_K^\lambda}=
\bigcup_{\mu\leq\lambda}\Omega K_c\backslash\mO_K^\mu$, 
Proposition \ref{Morse-Bott} and 
standard results for gradient flows (see, e.g., \cite[Proposition 1.19]{AB} or \cite[Theorem 1]{P})
imply that the limit $\underset{t\ra\pm\infty}\lim\underline\phi_t(\gamma)$ 
exists for any $\gamma\in\Omega K_c\backslash\Gr$ and 
$\Omega K_c\backslash\mO_K^\lambda$
is the stable manifold for $\Omega K_c\backslash\mO_c^\lambda$
and $\Omega K_c\backslash\mO_\bbR^\lambda$ is the unstable manifold for $\Omega K_c\backslash\mO_c^\lambda$.
Part (3) and (4) follows.
\end{proof}

We will call the gradient flow $\phi_t:\Gr\to\Gr$ the 
Matsuki flow on $\Gr$.

\section{Real Beilinson-Drinfeld Grassmannians}\label{Real BD}
In this section we recall some basic facts about 
real Beilinson-Drinfeld Grassmannians. 
The main references are \cite{N1,N2}.

\subsection{Real affine Grassmannians}\label{real affine Gr}
We recall results from \cite{N1} about the real affine Grassmannian.
Let $\Gr_\mathbb R:=G_\mathbb R(\calK_\mathbb R)/G_\mathbb R(\mO_\mathbb R)$
be the real affine Grassmannian. For any $\lambda\in\Lambda_T^+$ we denote by 
$S^\lambda$ and $T^\lambda$ the $G(\mO)$ and $G(\bC[t^{-1}])$-orbit 
of $t^\lambda\in\Gr$. 
The orbits $S^\lambda$ and $T^\mu$ on $\Gr$ are transversal and 
the intersection $C^\lambda=S^\lambda\cap T^\lambda$ is isomorphic to 
the flag manifold $G/P^\lambda$ where the parabolic subgroup $P^\lambda$ is the 
stabilizer of $\lambda$. 
The affine Grassmannian $\Gr$ is the disjoint union of the orbits $S^\lambda$ (resp. $T^\lambda$)
for $\lambda\in\Lambda_T^+$
\[\Gr=\bigsqcup_{\lambda\in\Lambda_T^+} S^\lambda\ \ \ \ (resp.\ \ \ 
\Gr=\bigsqcup_{\lambda\in\Lambda_T^+} T^\lambda)\]
and we have 
\[\overline S^\lambda=\bigsqcup_{\mu\leq\lambda} S^\mu\ \ \ \ (resp.\ \ \ 
T^\lambda=\bigsqcup_{\lambda\leq\mu} T^\mu).\]
The intersection of $S^\lambda$ (resp. $T^\lambda$) with $\Gr_\bbR$ is nonempty if and only if 
$\lambda\in\Lambda_A^+$ and we write $S_\bbR^\lambda$ (resp. $T_\bbR^\lambda$), $\lambda\in\Lambda_A^+$ for the intersection. We define $C_\bbR^\lambda$ to be the intersection of 
$S_\bbR^\lambda$ and $T_\bbR^\lambda$. $S_\bbR^\lambda$ (resp. $T_\bbR^\lambda$)
is equal to the $G_\bbR(\mO_\bbR)$-orbit (resp. $G_\bbR(\bbR[t^{-1}])$-orbit) of $t^\lambda$ and 
$C_\bbR^\lambda$ is isomorphic to the real flag manifold 
$G_\bbR/P_\bbR^\lambda$ where the parabolic subgroup $P_\bbR^\lambda\subset G_\bbR$ is the 
stabilizer of $\lambda$. 
The real affine Grassmannian $\Gr$ is the disjoint union of the orbits $S_\bbR^\lambda$ (resp. $T_\bbR^\lambda$)
for $\lambda\in\Lambda_T^+$
\beq\label{strata for real Gr}
\Gr_\bbR=\bigsqcup_{\lambda\in\Lambda_A^+} S_\bbR^\lambda\ \ \ \ (resp.\ \ \ 
\Gr_\bbR=\bigsqcup_{\lambda\in\Lambda_A^+} T_\bbR^\lambda)
\eeq
and we have 
\[\overline S_\bbR^\lambda=\bigsqcup_{\mu\leq\lambda} S_\bbR^\mu\ \ \ \ (resp.\ \ \ 
T_\bbR^\lambda=\bigsqcup_{\lambda\leq\mu} T_\bbR^\mu).\]

\subsection{Beilinson-Drinfeld Grassmannians}\label{BD Gr}
Let $C$ be a smooth curve over $\bC$.
Consider the functor 
$G(\mO)_{C^n}$ from the category of affine schemes to sets
\[S\ra G(\mO)_{C^n}(S):=\{(x,\phi)|x\in C^n(S), \phi\in G(\hat\Gamma_x)\}.\]
Here $\hat\Gamma_x$ is the formal completion of the graphs $\Gamma_x$ of 
$x$ in $C\times S$.
Similarly, we define 
$G(\calK)_{C^n}$ to be the functor 
from the category of affine schemes to sets
\[S\ra G(\calK)_{C^n}(S):=\{(x,\phi)|x\in C^n(S), \phi\in G(\hat\Gamma_x^0)\}.\]
Here $\hat\Gamma_x^0:=\hat\Gamma'_x-\Gamma_x$ and 
$\hat\Gamma'_x=\on{Spec}(R_x)$ is the spectrum of  
ring of functions $R_x$ of $\hat\Gamma_x$.
$G(\mO)_{C^n}$ is represented by a formally smooth group scheme over $C^n$
and $G(\calK)_{C^n}$ is represented by a formally smooth group ind-scheme over $C^n$.
The quotient ind-scheme 
\[\Gr_{C^n}:=G(\calK)_{C^n}/G(\mO)_{C^n}\] 
is 
called the Beilinson-Drinfeld Grassmannian. We have 
\[\Gr_{C^n}(S)=\{(x,\mE,\phi)|x\in C^n(S), \mE\text{\ a}\ G\text{-torsor\ on\ } 
C\times S, \phi\text{\ a trivialization of}\ \mE\text{\ on}\ C\times S-\Gamma_x\}.\]
For any $x\in C^n$, we denote by 
$\Gr_{C^n,x}=\Gr_{C^n}|_x$ the fiber over $x$.

Let $p(n)$ be the set of partition of the set $\{1,...,n\}$. For any $\fp=\fp_1\sqcup\cdot\cdot\cdot\sqcup\fp_k\in p(n)$
let $C^\fp$ be the subset consisting of 
$(z_1,...,z_n)\in C^n$ such that $z_i=z_j$ if and only if $i,j\in\fp_l$ 
for some $l$. The collection $\{C^\fp\}_{\fp\in p(n)}$
forms a Whitney stratification of $C^n$. We denote by $C^n_0=C^n_{1\sqcup\cdot\cdot\cdot\sqcup n}$ the open stratum
of distinct points.
We have the following well-known factorization properties \cite[Section 5.3.12]{BD}:
there are canonical isomorphisms
\beq\label{complex factorization}
G(\calK)_{C^n}|_{C^\fp}\is(\prod_{i=1}^kG(\calK)_{C})|_{C^{k}_0}\ \ \ 
G(\calO)_{C^n}|_{C^\fp}\is(\prod_{i=1}^kG(\calO)_C)|_{C^{k}_0}
\ \ \ 
\Gr_{C^n}|_{C^\fp}\is(\prod_{i=1}^k\Gr_C)|_{C^{k}_0}
\eeq

Recall  the standard stratification of $\Gr_{C^n}$.
For any map $\lambda_\fp:\fp\to\Lambda_T^+$, we denote 
by $\Gr_{C^n}^{\lambda_\fp}\subset\Gr_{C^n}$ the spherical stratum 
of a $G$-bundle on $\mE$ on $C$ a point in $z\in C^\fp$
and a section $\gamma:C\setminus\{z_1,...,z_n\}$
of modification type $\lambda_\fp$. The collection $\{\Gr_{C^n}^{\lambda_\fp}\}$
provides a stratification of $\Gr_{C^n}$
compatible with that of $C^n$.

We describe another stratification of $\Gr_{C^n}$ coming from the natural
$K(\calK)_{C^n}$-action. The strata of $\calO^{\lambda}_{K,C}$
are indexed by $\Lambda_A^+$. The stratum $\calO_{K,C}^{\lambda}$
consists of the union of all fibers $\Gr^{}_{C,x}$ of the points which map to the 
$K(\calK)_{C,x}$-orbit in $\calO_K^\lambda$ via the canonical isomorphism 
$\Gr_{C,x}\is \Gr $.
The strata of $\calO_{K,C^n}^{\lambda_\fp}$ are indexed by 
maps $\lambda_\fp:\fp\to\Lambda_A^+$.
The stratum $\calO_{K,C^n}^{\lambda_\fp}$
consists of points of $\Gr_{C^n}|_{C^\fp}$
which maps via the isomorphism~\eqref{complex factorization} to the product stratum
indexed by $\lambda_\fp$.

\subsection{Real forms of Beilinson-Drinfeld Grassmannians}\label{real Gr}
We consider the case $C=\mathbb P^1=\bC\cup\infty$. We 
write $\Gr^{(n)}=\Gr_{C^n}$, $G(\calK)^{(n)}=G(\calK)_{C^n}$,
$G(\calO)^{(n)}=G(\calO)_{C^n}$,  
$\Gr^{(n)}_x=\Gr^{(n)}|_x$, 
$S^{(n),\lambda_\fp}=\Gr_{C^n}^{\lambda_\fp}$,
and $\calO_{K,C^n}^{\lambda_\fp}=\mO^{(n),\lambda_\fp}_{K}$.

Let $c:\mathbb P^1\ra \mathbb P^1$ be the complex conjugation.
Consider the following anti-holomorphic involution 
$c^n:(\mathbb P^1)^n\ra (\mathbb P^1)^n, c^n(z_1,...,z_n)=(c(z_1),...,c(z_n))$. The involution $c_n$ together with the involution 
$\eta$ on $G$ defines anti-holomorphic involutions
$G(\calK)^{(n)}$, $G(\calO)^{(n)}$, 
$\Gr^{(n)}$
and 
we write 
$G(\calK)^{(n)}_\bbR$, $G(\calO)^{(n)}_\bbR$, 
$\Gr^{(n)}_\bbR=G(\calK)^{(n)}_\bbR/G(\calO)^{(n)}_\bbR$ for the corresponding semi-analytic 
spaces.

We have a natural projection map
\[\Gr^{(n)}_\bbR\to (\mathbb P^1(\bbR))^n\]
and the factorization isomorphism in~\eqref{complex factorization}
restricts to an isomorphism
\beq
\Gr^{(n)}_\bbR|_{\mathbb P^1(\bbR)^\fp}\is(\prod_{i=1}^k\Gr^{(1)}_\bbR)|_{\mathbb P^1(\bbR)^k_0}
\eeq

Consider the case when $n=2m$ is even.
Let $\sigma_{2m}:(\bbP^1)^{2m}\to (\bbP^1)^{2m}$ be the complex conjugation given by 
$\sigma_{2m}(z_1,...,z_{2m})=(c(z_{m+1}),...,c(z_{2m}),c(z_1),...,c(z_m))$.
Then $\sigma_{2m}$ together with $\eta$ 
defines anti-holomorphic involutions of
$G(\calK)^{(2m)}$, $G(\calO)^{(2m)}$, 
$\Gr^{(2m)}$
and 
we write 
$G(\calK)^{(\sigma_{2m})}_\bbR$, $G(\calO)^{(\sigma_{2m})}_\bbR$, 
$\Gr^{(\sigma_{2m})}_\bbR=G(\calK)^{(\sigma_{2m})}_\bbR/G(\calO)^{(\sigma_{2m})}_\bbR$ for the corresponding real forms.
The projection map 
$(\mathbb P^1)^{2m}\to (\mathbb P^1)^m$ sending 
$(z_1,...,z_{2m})\to (z_1,...,z_m)$ restricts to a real analytic 
isomorphism 
$((\mathbb P^1)^{2m})^{(\sigma_{2m})}\is (\mathbb P^1)^m$
and we have a natural projection map
\beq
\Gr^{(\sigma_{2m})}_\bbR\to (\mathbb P^1)^m.
\eeq

Write $\calH=\{z\in\mathbb P^1|\on{im}(z)>0\}$ be the upper half-plane.
According to \cite[Proposition 4.3.1]{N2}
we have the following isomorphisms
\beq\label{fibers over real and complex}
\Gr^{(\sigma_{2m})}_\bbR|_{
\calH^m}\is\Gr^{(m)}|_{
\calH^m}\ \ \ \ \ \ \ 
\Gr^{(\sigma_{2m})}_\bbR|_{\mathbb P^1(\bbR)^m}\is\Gr^{(m)}_\bbR
\eeq
Moreover, for any  $\fp=\fp_1\sqcup\cdot\cdot\cdot\sqcup\fp_k$
partition of $\{1,...,m\}$, 
there is a natural isomorphism
\beq\label{real complex factorization}
\Gr^{(\sigma_{2m})}_\bbR|_{(\mathbb P^1)^\fp}\is(\prod_{i=1}^k\Gr^{(\sigma_2)}_\bbR)|_{(\mathbb P^1)^k_0}
\eeq

The stata $S^{(n),\lambda_\fp}$ (resp. $S^{(2m),\lambda_\fp}$)
restricts to strata $S^{(n),\lambda_\fp}_\bbR$ (resp. $S^{(\sigma_{2m}),\lambda_\fp}_\bbR$)
of $\Gr^{(m)}_\bbR$  (resp. $\Gr^{(\sigma_{2m})}_\bbR$).
It follows from~\eqref{strata for real Gr} that $S^{(n),\lambda_\fp}_\bbR$ is non-empty if and only if 
$\lambda_\fp:\fp\to\Lambda_A^+\subset\Lambda^+_T$.

\begin{example}\label{m=1}
Consider the case when $m=1$. Then the fibers of 
the family $\Gr_\bbR^{(\sigma_2)}\to\mathbb P^1$ over a point 
$z\in\mathbb P^1$ are given by
$\Gr_\bbR^{(\sigma_2)}|_{z=\bar z}\is\Gr_{z,\bbR}$
and $\Gr_\bbR^{(\sigma_2)}|_{z\neq \bar z}\is \Gr_z$.
\end{example}

\section{Multi-point generalizations}\label{Gram-Schmidt}
We study  multi-point version of  loop groups and loop spaces and we establish a 
 multi-point generalization of Gram-Schmidt factorization for loop groups.
The main reference is \cite{CN1}.

\subsection{Multi-point version of real loop groups}
Consider the functor $LG^{(n)}$ that assigns to an affine scheme $S$
the set of sections 
\[S\ra LG^{(n)}(S)=\{(x,\gamma)|x\in (\bbP^1)^n(S),\ \gamma\in G(\bbP^1\times S-\Gamma_x).\}.\]
Assume $n=2m$. The conjugations $\sigma_{2m}$
and $\eta$ together define an anti-holomorphic involution on 
$LG^{(n)}$
and we write
\[LG^{(\sigma_{2m})}_\bbR\ra(\mathbb P^1)^m\] 
for the corresponding semi-analytic space.
We will write 
\[LG^{(m)}_\bbR:=LG^{(\sigma_{2m})}_\bbR|_{\calH^m}\ra\calH^m\ \ \ \ \ \ L^-G^{(m)}_\bbR:=LG^{(\sigma_{2m})}_\bbR|_{\bbP^1(\bbR)^m}\to\bbP^1(\bbR)^m\]
for the restriction of $LG^{(\sigma_{2m})}_\bbR$
to 
$\calH^m$ and $\bbP^1(\bbR)^m$ respectively.
Concretely,
we have 
\[LG^{(m)}_\bbR=\{(x,\gamma)|x\in \calH^m,\ \gamma:\bbP^1\setminus |x|\cup|\bar x|\to G
\text{\ satisfying\ }
\gamma(\bbP^1(\bbR))\subset G_\bbR\}.\] 
\[L^-G^{(m)}_\bbR=\{(x,\gamma)|x\in\bbP^1(\bbR)^m,\ \gamma:\bbP^1\setminus |x|\to G
\text{\ satisfying\ }
\gamma(\bbP^1(\bbR)\setminus |x|)\subset G_\bbR\}.\]

We consider the subgroup indscheme 
$\Omega G^{(n)}\subset LG^{(n)}|_{\bC^n}$
consisting of 
$(x,\gamma)\in LG^{(n)}|_{\bC^n}$ such that 
$\gamma(\infty)=e$.
Since the conjugations $c_n$ and $\sigma_{2m}$
fixed $\infty\in\mathbb P^1$, they induces an involution of 
$\Omega G^{(2m)}$
and let  
\[\Omega G_\bbR^{(\sigma_{2m})}\to\bC^m\] 
be the associated semi-analytic  space
and the subspaces
\[\Omega G_\bbR^{(m)}:=\Omega G_\bbR^{(\sigma_{2m})}|_{\calH^m}\ra\calH^m\ \  \ \ \ 
L^-_\infty G_{\bbR}^{(m)}:=\Omega G_\bbR^{(\sigma_{2m})}|_{\bbR^m}\to\bbR^m.\]

 \begin{example}\label{the case m=1}
 Assume $m=1$. Then the fiber of $LG_\bbR^{(1)}$
 over $x\in\calH$ is equal to 
 \[LG_\bbR^{(1)}|_x=L_xG_\bbR:=\{\gamma:\bbP^1\setminus x\cup\bar x\to G
 \text{ satisfying } \gamma(\bbP^1(\bbR))\subset G_\bbR\} \]
The change of coordinate  $t=\frac{z-x}{z-\bar x}$
induces isomorphism
$L_xG_\bbR\is LG_\bbR$.
The fiber of $L^-G^{(1)}_\bbR$
 over $x\in\mathbb P^1(\bbR)$ is equal to 
 \[{L^-G}^{(1)}_\bbR|_x={L^-_xG}_\bbR:=\{\gamma:\bbP^1\setminus x\to G
 \text{ satisfying } \gamma(\bbP^1(\bbR)\setminus x)\subset G_\bbR\} \]
The change of coordinate  $t=z-x$
induces isomorphism
${L^-_xG}_\bbR\is G_\bbR(\bbR[t^{-1}])$.
 \end{example}

\subsection{Generalization of Gram-Schmidt factorization}
Consider the case of 
compact real from 
 $G_\bbR=G_c$
 and write $\Omega G_c^{(\sigma_{2m})}=\Omega G_\bbR^{(\sigma_{2m})}$,
 $\Omega G_c^{(m)}=\Omega G_\bbR^{(m)}$, $L_\infty^-G_c^{(m)}=L_\infty^-G_\bbR^{(m)}$.
 
 \begin{lemma}\label{fibers: compact case}
There is an isomorphism $L^-G_c^{(m)}\is G_c\times\bP^1(\bbR)^m$
(resp. $L_\infty^-G_c^{(m)}\is \{e\}\times\bbR^m$) of group schemes.

\end{lemma}
 \begin{proof}
 This is proved in \cite[Lemma 3.2]{CN1}.
 \end{proof}

\begin{corollary}\label{poly in K_c}
We have 
$G_c(\mathbb R[t^{-1}])=G_c$ and $G_c(\calK_\mathbb R)=G_c(\mO_\mathbb R)$.
\end{corollary}
\begin{proof}
Indeed the lemma above (in the case $m=1$) implies 
$G_c(\mathbb R[t^{-1}])=G_c$.
Since the $G_c(\mathbb R[t^{-1}])$-orbit in $\Gr_c=G_c(\calK_\bbR)/G_c(\mO_\bbR)$ 
through the based point is open dense 
(see Section \ref{real affine Gr}), it implies 
$\Gr_c$ is a point and hence
 $G_c(\calK_\mathbb R)=G_c(\mO_\mathbb R)$.
\end{proof}

\begin{proposition}\label{multi Gram-Schmidt}
(1)
For any partition $\fp\in p(m)$, there is a natural homeomorphism
\beq\label{factorization of Omega1}
\Omega G_c^{(\sigma_{2m})}|_{(\bP^1)^\fp}\is(\prod_{i=1}^k\Omega G_c^{(\sigma_2)})|_{(\bP^1)^k_0}
\eeq
(2) The homeomorphism~\eqref{factorization of Omega1}
restricts to a natural a homeomorphism 
\beq\label{factorization of Omega2}
\Omega G_c^{(m)}|_{\calH^\fp}\is(\prod_{i=1}^k\Omega G_c^{(1)})|_{\calH^k_0}
\eeq

(3) There is a natural homeomorphism
\[\Omega G_c^{(m)}\is\Gr^{(m)}|_{\calH^m}.\]

(4) There is a natural homeomorphism
\[\Omega G_c^{(m)}\times (G(\calO)^{(m)}|_{\calH^m})\is G(\calK)^{(m)}|_{\calH^m}\]

\end{proposition}
\begin{proof}
Proof of (1) and (2).
\cite[Theorem 3.3 (3)]{CN1} implies that the natural multiplication map
$(\prod_{i=1}^k\Omega G_c^{(\sigma_2)})|_{(\bP^1)^k_0}\to \Omega G_c^{(\sigma_{2m})}|_{(\bP^1)^\fp}$
is a homeomorphism.

There is a natural map $\Omega G_c^{(m)}\ra G(\calK)^{(m)}|_{\calH^m}$
sending $(x,\gamma)$ to $(x,\gamma|_{\hat\Gamma^0_x})$, 
where $\gamma|_{\hat\Gamma^0_x}$ is the restriction of 
the section $\gamma:\bbP^1\setminus |x|\cup|\bar x|\to G$ to $\hat\Gamma^0_x$.
It induces a map
\[\Omega G_c^{(m)}\to G(\calK)^{(m)}|_{\calH^m}\to G(\calK)^{(m)}/G(\calO)^{(m)}|_{\calH^m}
=\Gr^{(m)}|_{\calH^m}\] and \cite[Theorem 3.3 (2)]{CN}
implies that it is a homeomorphism. Part (3) follows.
Since $\Gr^{(m)}\is G(\calK)^{(m)}/G(\calO)^{(m)}$, part (3) implies part (4).
\end{proof}

\begin{example}
Assume $m=1$.
Then   
Proposition \ref{multi Gram-Schmidt} (4) over  a point $z\in\calH$ specializes to
the well-known 
Gram-Schmidt factorization 
for loop groups $\Omega_z G_c\times G(\calO_z)\is G(\calK_z)$.

\end{example}

\subsection{Multi-point version of $\Omega X_c$}\label{multi loop of X}
For any non-negative integer $m$, 
we define 
\[\Omega X^{(m)}_c=\{(x,\gamma)|x\in\calH^m,\ \gamma:\bbP^1\setminus |x|\cup|\bar x|\to X
\text{\ satisfying\ }
\gamma(\bbP^1(\bbR))\subset X_c, \gamma(\infty)=e\}.\] 
We have natural projection map
\[\Omega X_c^{(m)}\to\calH^m\]
and we denote by 
$\Omega_zX_c^{(m)}$ the fiber over 
$z\in\calH^m$. 
The involution $\tilde\theta$ on $G$ induces an 
involution 
on $\Omega G_c^{(m)}$, still denoted by $\tilde\theta$,
and the isomorphism
$\iota:X\is G_{sym}^0= (G^{\tilde\theta})^{0}$ induces an isomorphism 
\[\iota^{(m)}:\Omega X_c^{(m)}\is(\Omega G_c^{(m)})^{\tilde\theta}.\]
We define
\[P^{(m),\lambda_\fp}=(\iota^{(m)})^{-1}(S^{(m),\lambda_\fp})^{\tilde\theta})\]
The collection $\{P^{(m),\lambda_\fp}\}_{\lambda_\fp:\fp\to\Lambda_A^+}$
forms a stratifiaction of 
$\Omega X_c^{(m)}$.
We have the following multi-point version of 
Proposition \ref{torsors}.

\begin{proposition}\label{components of multi loops of X}
The quotient map 
$q:G\to X$ induces a
 $K_c$-equivariant stratified homeomorphisms
\[\Omega K_c^{(m)}\backslash\Omega G_c^{(m)}\is \Omega X_{c}^{(m)}\]
which restricts to $K_c$-equivariant homeomorphisms on stratum
\[\Omega K_c^{(m)}\backslash\calO_K^{(m),\lambda_\fp}\is
P^{(m),\lambda_\fp}\]
\end{proposition}
\begin{proof}
Consider the natural map
$q^{(m)}:\Omega G_c^{(m)}\to\Omega X_c^{(m)}$.
We shall show that, for any $z\in\calH^m$,
the  map $q^{(m)}:\Omega_z G_c^{(m)}\to\Omega_z X_c^{(m)}$
on the corresponding fiber 
is surjective. 
Since $\Omega G_c^{(m)}$ is ind-proper, 
we conclude that the natural map 
$q^{(m)}:\Omega G_c^{(m)}\to\Omega X_c^{(m)}$
is surjective and closed, and it follows that the induced map
$\Omega K_c^{(m)}\backslash\Omega G_c^{(m}\to\Omega X_c^{(m)}$
is a continuous, closed, bijective map, and hence a homeomorphism.
Since $q^{(m)}(\mO_K^{(m),\lambda_\fp})\subset P^{(m),\lambda_\fp}$, it implies 
the induced map $\Omega K_c^{(m)}\backslash\mO_K^{(m),\lambda_\fp}\to P^{(m),\lambda_\fp}$ is also a homeomorphism.
The $K_c$-equivariance is clear. 

Proof of the claim. Let $P_z^{(m),\lambda_\fp}=P^{(m),\lambda_\fp}\cap\Omega_zX_c^{(m)}$.
It suffices to show $P_z^{(m),\lambda_\fp}\subset q^{(m)}(\Omega_z G_c^{(m)})$.
Let $\Gr^{(m)}_z\tilde\times\Gr^{(m)}_z=G(\calK)^{(m)}_z\times^{G(\calO)^{(m)}_z}\Gr^{(m)}_z$.
The multi-point Gram-Schmidt factorization 
\[\Omega_z G_c^{(m)}\times G(\calO)^{(m)}_z\is G(\calK)^{(m)}_z,\ \ \ \ 
\Omega_z G_c^{(m)}\is\Gr_z^{(m)}\]
in Proposition \ref{multi Gram-Schmidt} implies that the natural inclusion 
$\Omega_z G_c^{(m)}\times \Omega_z G_c^{(m)}\to G(\calK)_z^{(m)}\times G(\calK)_z^{(m)}$ induces a
natural homeomorphism 
\[\Omega_z G_c^{(m)}\times \Omega_z G_c^{(m)}\is \Gr^{(m)}_z\tilde\times\Gr^{(m)}_z.\]
Moreover, we have the following commutative diagram
\[\xymatrix{\Omega_z G_c^{(m)}\times \Omega_z G_c^{(m)}\ar[r]^{\ \ \simeq}\ar[d]^{mult}&\Gr^{(m)}_z\tilde\times\Gr^{(m)}_z\ar[d]^p\\
\Omega_z G^{(m)}_c\ar[r]^{\simeq}&\Gr^{(m)}_z}\]
where $mult$ is the multiplication map and 
$p$ is the convolution map.
Pick a $s\in\Omega_z G^{(m)}_c$
such that $\gamma=\tilde\theta(s)s\in (S_z^{(m),\lambda_\fp})^{\tilde\theta}\is P_z^{(m),\lambda_\fp}$.
Consider the $G(\calO)_z^{(m)}$-action on $\Gr^{(m)}_z\tilde\times\Gr^{(m)}_z$
given by $g_+(\gamma_1,\gamma_2)=(g_+\gamma_1,\gamma_2)$
and 
$S$ be the $G(\calO)_z^{(m)}$-orbit of $(\tilde\theta(s),s)\in\Omega_z G_c^{(m)}\times \Omega_z G_c^{(m)}\is \Gr^{(m)}_z\tilde\times\Gr^{(m)}_z$.
Let $Z$ be the intersection of the closure $\overline S$ of $S$
with the pre-image $p^{-1}(S_z^{(m),\lambda_\fp})$.
Note that $Z$ is a disjoint union of $G(\calO)_z^{(m)}$-orbits
and $p:Z\to S_z^{(m),\lambda_\fp}$ is proper and restricts to a
submersion on each orbit.

Consider the involution 
$\on{sw}\circ\tilde\theta$ on $\Omega_z G_c^{(m)}\times \Omega_z G_c^{(m)}\is
\Omega_z G_c^{(m)}\is
\Gr^{(m)}_z\tilde\times\Gr^{(m)}_z$
given by $\on{sw}\circ\tilde\theta(\gamma_1,\gamma_2)=(\tilde\theta(\gamma_2),\tilde\theta(\gamma_1))$.
The lemma below implies that 
$Z$ is stable under the involution $\on{sw}\circ\tilde\theta$ and the map 
$p:Z\to S_z^{(m),\lambda_\fp}$ is compatible with the involutions
$\on{sw}\circ\tilde\theta$ on $Z$
 and $\tilde\theta$ on $S_z^{(m),\lambda_\fp}$.
 Now we can apply the general lemma \cite[Proposition 6.4]{N1}
 to conclude that $P_z^{(m),\lambda_\fp,\circ}\subset p(Z^{\on{sw}\circ\tilde\theta})$ where $P_z^{(m),\lambda_\fp,\circ}\subset
 P_z^{(m),\lambda_\fp}\is
(S_z^{(m),\lambda_\fp})^{\tilde\theta}$
is a connected component of containing 
 $\gamma=\tilde\theta(s)s$. By Theorem \ref{Quillen}, 
 there is a homeomorphism between $P_z^{(m),\lambda_\fp}$ and a fiber of the strata $S_\bbR^{(m),\lambda_\fp}\subset\Gr^{(m)}_\bbR$
 in the real Beilinson-Drinfeld grassmannian.
 Since the latter fiber is a  vector bundle over the real 
 flag manifold $G_\bbR/P_\bbR$ (see \cite[Proposition 3.6.1]{N2}) which is connected (we assume $G_\bbR$ is connected), we conclude that 
 $P_z^{(m),\lambda_\fp}=P_z^{(m),\lambda_\fp,\circ}
\subset p(Z^{\on{sw}\circ\tilde\theta})$.
Now we can conclude the proof by noting there is natural identification of 
the $\on{sw}\circ\tilde\theta$-fixed points of 
$\Gr_z^{(m)}\tilde\times\Gr_z^{(m)}$ with $\Omega_z G_c^{(m)}$ so that 
$p$ restricts to $\Omega_z G_c^{(m)}$ coinsides with
$q^{(m)}$.

\end{proof}

\begin{lemma}
The involution
$\on{sw}\circ\tilde\theta$ on $\Omega_z G_c^{(m)}\times \Omega_z G_c^{(m)}\is
\Gr^{(m)}_z\tilde\times\Gr^{(m)}_z$
sends the $G(\calO)_z^{(m)}$-orbit of 
$(\gamma_1,\gamma_2)$ to the $(\tilde\theta(\gamma_2),\tilde\theta(\gamma_1))$.
\end{lemma}
\begin{proof}
Let $(a,b)=g(\gamma_1,\gamma_2)\in\Gr^{(m)}_z\tilde\times\Gr^{(m)}_z$ with $g\in G(\calO)^{(m)}_z$.
By the factorization homeomorphism in Proposition \ref{multi Gram-Schmidt},
it implies 
$(g\gamma_1,\gamma_2)=(au,u^{-1}bg')$ in $G(\calK)^{(m)}_z$
for some $g',u\in G(\calO)^{(m)}_z$.
Let $h=\theta(g')\in G(\calO)^{(m)}_z$.
Then a direct computation show show that we have 
\[(h\tilde\theta(\gamma_2),\tilde\theta(\gamma_1))=(\tilde\theta(b)\tilde\theta(u^{-1}),
\tilde\theta(u)\tilde\theta(a)\tilde\theta(g^{-1}))
\]
in $G(\calK)^{(m)}_z\times G(\calK)_z^{(m)}$
and it implies 
\[h(\tilde\theta(\gamma_2),\tilde\theta(\gamma_1))=(\tilde\theta(b),\tilde\theta(a))\]
in $\Gr^{(m)}_z\tilde\times\Gr^{(m)}_z$. The lemma follows.
\end{proof}

\section{Uniformizations of real bundles}\label{uniform}
In this section we 
study uniformizations of the stack of real bundles on 
$\mathbb P^1$ and use it to provide a moduli interpretation for  
the quotient
$LG_\bbR\backslash\Gr$.

\subsection{Stack of real bundles}\label{Stack of real bundles}
Let $\Bun_\bbG(\bbP^1)$ be the moduli stack of $\bbG$-bundles on 
the complex projective line $\mathbb P^1$. The standard complex conjugation 
$z\to\bar z$
on $\mathbb P^1$  together with 
the involution $\eta$ of $\bbG$ defines a real structure $c:\Bun_\bbG(\bbP^1)\to
\Bun_\bbG(\bbP^1)$ on 
$\Bun_\bbG(\bbP^1)$
with real form 
$\Bun_{\bbG_\bbR}(\bbP_\bbR^1)$, the real algebraic stack of $\bbG_\bbR$-bundles on the projective real line $\bbP^1_\bbR$.
We write 
$\Bun_{\bbG}(\bbP^1)_\bbR$ for the 
real analytic stack of real points of $\Bun_{\bbG_\bbR}(\bbP_\bbR^1)$.
By definition, we have $\Bun_{\bbG}(\bbP^1)_\bbR\is \Gamma_\bbR\backslash Y_\bbR$ where $Y\to\Bun_{\bbG_\bbR}(\bbP_\bbR^1)$ is a $\bbR$-surjective presentation\footnote{A presentation of a real algebraic stack is $\bbR$-surjective if it induces a surjective map
on the isomorphism classes of $\bbR$-points.} of 
the real algebraic stack $\Bun_{\bbG_\bbR}(\bbP_\bbR^1)$,  
$\Gamma=Y\times_{\Bun_{\bbG_\bbR}(\bbP_\bbR^1)}Y$ is the corresponding groupoid, 
and $X_\bbR,\Gamma_\bbR$ are the real analytic spaces of 
real points of $X,\Gamma$
(see Appendix \ref{Real stacks}).

A point of $\Bun_{\bbG}(\bbP^1)_\bbR$ is 
a $\bbG_\bbR$-bundle $\mE_\bbR$ on $\bbP_\bbR^1$ and, by descent,
corresponds to  
a pair $(\mE,\gamma)$
where $\mE$ is a $\bbG$-bundle on $\bbP^1$ and 
$\gamma:\mE\is c(\mE)$ is an isomorphism such that the induced 
composition is the identity
\[\mE\stackrel{\gamma}\ra c(\mE)\stackrel{c(\gamma)}\ra c(c(\mE))=\mE.\]
We call such pair $(\mE,\gamma)$ a real bundle on $\bbP^1$ and 
$\Bun_{\bbG}(\bbP^1)_\bbR$ the stack of real bundles on $\bbP^1$.

For any $\bbG_\bbR$-bundle $\mE_\bbR$, the restriction of 
$\mE_\bbR$ to the (real) point $\infty$ is a 
$\bbG_\bbR$-bundle on $\Spec(\bbR)$ and
the assignment $\mE_\bbR\to\mE_\bbR|_\infty$
defines a morphism
\[
\Bun_{\bbG_\bbR}(\bbP_\bbR^1)\lra\mathbb B\bbG_\bbR.
\]
For each $\alpha\in H^1(\Gal(\bC/\bbR),G)$, 
let $T_\alpha$ be a $\bbG_\bbR$-torsor on $\Spec(\bbR)$ in the isomorphism class of 
$\alpha$ and we define $\bbG_{\bbR,\alpha}=\Aut_{\bbG_\bbR}(T_\alpha)$.
The collection $\{\bbG_{\bbR,\alpha},\ \alpha\in H^1(\Gal(\bC/\bbR),G)\}$
is the set of pure inner forms of $\bbG_\bbR$.
Let $G_{\bbR,\alpha}=\bbG_{\bbR,\alpha}(\bbR)$ be the real analytic group associated to 
$\bbG_{\bbR,\alpha}$.
We denote by $\alpha_0$ the isomorphism class of trivial $\bbG_\bbR$-torsor
with real group $G_{\bbR,\alpha_0}=G_\bbR$. 
By Example \ref{BG},
the morphism above induces a morphism 
\[
cl_\infty:\Bun_{\bbG}(\bbP^1)_\bbR\lra\mathbb \bigsqcup_{\alpha\in
H^1(\Gal(\bC/\bbR),G)}\mathbb BG_{\bbR,\alpha}
\]
on the corresponding real analytic stacks.
Define 
\beq
\Bun_{G_{\bbR,\alpha}}(\bbP^1(\bbR)):=(cl_\infty)^{-1}(\mathbb BG_{\bbR,\alpha})
\eeq
for the inverse image of $\mathbb BG_{\bbR,\alpha}$ under $cl_\infty$. 
Note that each $\Bun_{G_{\bbR,\alpha}}(\bbP^1(\bbR))$ is 
an union of connected components of $\Bun_{\bbG}(\bbP^1)_\bbR$ and 
we obtain the following decomposition
of
the stack of real bundles 
\[
\Bun_\bbG(\bbP^1)_{\bbR}=\bigsqcup_{\alpha\in
H^1(\Gal(\bbC/\bbR),G)}\Bun_{G_{\bbR,\alpha}}(\bbP^1(\bbR)).
\]
We will call  $\Bun_{G_{\bbR,\alpha}}(\bbP^1(\bbR))$ the stack of 
$G_{\bbR,\alpha}$-bundles on $\bbP^1(\bbR)$.
\begin{example}
Consider $G=\bC^\times$. In the case $\eta$ is 
the split conjugation,
the cohomology group
$H^1(\Gal(\bC/\bbR),G)$ is trivial and we have
\[\Bun_\bG(\bP^1)_\bbR\is\bZ\times B\bbR^\times.\]
In the case $\eta=\eta_c$ is the compact conjugation, we have 
$H^1(\Gal(\bC/\bbR),G)=\{\alpha_0,\alpha_1\}\is\bZ/2\bZ$ and 
\[\Bun_\bG(\bP^1)_\bbR\is\Bun_{G_{\bbR,\alpha_1}}(\bP^1(\bbR))\cup\Bun_{G_{\bbR,\alpha_2}}(\bP^1(\bbR)),\]
where $\Bun_{G_{\bbR,\alpha_i}}(\bP^1(\bbR)\is BS^1$.
\end{example}
\quash{Consider the non-abelian cohomology $H^1(\Gal(\bC/\bbR),G)$ that 
classifies isomorphism classes of $\bbG_\bbR$-bundles over $\Spec(\bbR)$.
We write $\alpha_0\in H^1(\Gal(\bC/\bbR),G)$ for the class of trivial $\bbG_\bbR$-bundle.
and we write 
$cl_\infty(\mE_\bbR)\in H^1(\Gal(\bC/\bbR),G)$ for the corresponding isomorphism class. 
The assignment $\mE_\bbR\ra cl_\infty(\mE_\bbR)$
defines a map
\beq
cl_\infty:\Bun_{\bbG_\bbR}(\bbR\bbP^1)\lra H^1(\Gal(\bbC/\bbR),G)
\eeq
and for each class $\alpha\in H^1(\Gal(\bbC/\bbR),G)$
}

\subsection{Uniformizations of real bundles}\label{uniformization of real bundles}
We shall introduce and study two kinds of uniformization 
of real bundles: one uses a real point of $\bbP^1$ called the real uniformization 
the other uses a complex point of $\bbP^1$ called the complex uniformization. 

\subsubsection{Real uniformizations}
The unifomization morphism 
\[u:\Gr\to\Bun_{\bG}(\bP^1)\] for $\Bun_{\bbG}(\bbP^1)$ exhibits $\Gr$ as a $G(\bC[t^{-1}])$-torsor over $\Bun_{\bbG}(\bbP^1)$, in particular,  we have an isomorphism 
\beq\label{uni iso}
G(\bC[t^{-1}])\backslash\Gr\is\Bun_{\bG}(\bP^1).
\eeq
The map $u$ is compatible with the real structures 
on $\Gr$ and $\Bun_{\bG}(\bP^1)$ and 
we denote by 
\beq\label{real uniformization}
u_{\bbR}:\Gr_{\bbR}\ra\Bun_\bbG(\bbP^1)_\bbR
\eeq
the associated map between the corresponding semi-analytic stacks of real points. 
We call the morphism $u_{\bbR}$ the real uniformization.
It follows from \eqref{uni iso} that 
$u_{\bbR}$ factors through an embedding 
\beq\label{open embedding 1}
G_\bbR(\bbR[t^{-1}])\backslash\Gr_{\bbR}\ra\Bun_\bbG(\bbP^1)_\bbR.
\eeq
We shall describe the image of $u_\bbR$.

\begin{prop}\label{uniformizations at real x}
The map $u_{\bbR}$ factors through 
\[
u_{\bbR}:\Gr_{\bbR}\ra\Bun_{G_\bbR}(\bbP^1(\bbR))\subset\Bun_\bbG(\bbP^1)_{\bbR}
\]
and induces an isomorphism of semi-analytic stacks
\[G_\bbR(\bbR[t^{-1}])\backslash\Gr_{\bbR}\stackrel{\sim}\lra\Bun_{G_\bbR}(\bbP^1(\bbR)).\]

\end{prop}
\begin{proof}
Since every $\bbG_\bbR$-bundle 
$\mE_\bbR$ in the image of $u_\bbR$ is trivial over 
$\bP^1_\bbR-\{0\}$, in particular at $\infty$, we have 
$\mE_\bbR\in\Bun_{G_\bbR}(\bbP^1(\bbR))$. 
Thus the map $u_\bbR$ factors through $\Bun_{G_\bbR}(\bbP^1(\bbR))$.
We show that the resulting morphism $u_\bbR:\Gr_{\bbR}\to\Bun_{G_\bbR}(\bbP^1(\bbR))$ is
surjective. Let $f:S\to\Bun_{G_\bbR}(\bbP^1(\bbR))$ be a smooth presentation
(note that $S$ is smooth as $\Bun_{G_\bbR}(\bbP^1(\bbR))$ is smooth). It suffices 
to show that, \'etale locally on $S$, $f$ admits a lifting to $\Gr_{\bbR}$.
Consider the fiber product $Y:=S\times_{\Bun_{G_\bbR}(\bbP^1(\bbR))}\Gr_{\bbR}$ 
and we denote by  $h:Y\to S$ the natural projection map. 
It suffices to show that $h$ is surjective and admits a section 
\'etale locally on $S$.
By Theorem 1.1 in \cite{MS},  
every $\bbG_\bbR$-bundle 
$\mE_\bbR$
on $\bbP_\bbR^1$ 
which is trivial at $\infty$ admits a trivialization on 
$\bP^1_\bbR-\{0\}$. It implies $h$ is surjective.
To show that $h$ admits a section, we observe that 
$Y$ is a real analytic ind-space smooth over $\Gr_{\bbR}$  
and, as $u_{\bbR}$ is formally smooth, for any $y\in Y$ and $s=h(y)\in S$,
the tangent map $dh_y:T_yY\to T_sS$ is surjective. 
Choose a finite dimensional subspace $W\subset T_yY$ such that 
$dh_y(W)=T_sS$. We claim that there exists a smooth real analytic space 
$U\subset Y$ such that $y\in U$ and $T_yU=W$. This implies 
$h|_U:U\to S$ is smooth around $y$, thus $f$ admits a section \'etale locally around 
$s=h(y)$. 
Finally, by \eqref{open embedding 1}, we obtain an isomorphism $G_\bbR(\bbR[t^{-1}])\backslash\Gr_\bbR\is
\Bun_{G_\bbR}(\bbP^1(\bbR))$. 

To prove the claim, we observe that $Y$ is locally isomorphic to 
$\Gr_{\bbR}$ times a smooth real analytic space.  So it suffices to show for any 
finite dimensional subspace $W\subset T_e\Gr_{\bbR}$, there exists 
a smooth real analytic space $U$ such that $T_eU=W$. 
This
follows from the fact that 
the exponential map $\exp:T_e\Gr_{\bbR}\to\Gr_\bbR$ associated to the 
metric $g(,)|_{\Gr_\bbR}$ (here $g(,)$ is the metric on $\Gr$ in Section \ref{energy flow}) is a local diffeomorphism.

\end{proof}

We can also consider the 
uniformization map 
$\Gr_z\to\Bun_{\bbG}(\bbP^1)$ for 
any real point $x\in\mathbb P^1(\bbR)$ and the discussion above applies to this setting and hence 
we have obtain real uniformization at $x$
\beq\label{real uniform at x}
u_{x,\bbR}: L_x^-G_\bbR\backslash\Gr_{x,\bbR}\is\Bun_{G_\bbR}(\bbP^1(\bbR)).
\eeq
where $ L_x^-G_\bbR$ is real analytic group in
 Example \ref{the case m=1}.

\quash{
Since there is a bijection 
between the set $\Lambda_S^+$ of real dominant coweights and the set 
of $G_\bbR[t_x^{-1}]$-orbits on $\Gr_{x,\bbR}$, we obtain the following: 
\begin{corollary}
There is a bijection
$\Lambda_S^+\leftrightarrow |G_\bbR(\bbR[t_x^{-1}])\backslash\Gr_{x,\bbR}|\leftrightarrow\on{}|\Bun_\bbG(\bbP^1)_{\bbR,\alpha_0}|$.

\end{corollary}}

\subsubsection{Complex uniformizations}
We now discuss complex uniformizations. 
The natural map 
\[u^{(2)}:\Gr^{(2)}\ra\Bun_\bbG(\bbP^1)\ \ \ \ (x,\mE,\phi)\to \mE\] 
 is compatible with 
the complex conjugation
$\sigma_2$ on 
$\Gr^{(2)}$ and the natural one on
$\Bun_\bbG(\bbP^1)$
and we denote by  
\beq\label{family of uniformizations}
u^{(\sigma_2)}_\bbR:\Gr^{(\sigma_2)}_\bbR\longrightarrow\Bun_\bbG(\bbP^1)_\bbR
\eeq
the map between the corresponding real analytic stacks.

Let us study the restriction of $u^{(\sigma_2)}_\bbR$
to the fiber $\Gr^{(\sigma_2)}_{\bbR}|_x$ over some $x\in\mathbb P^1$.
If $x\in\mathbb P^1(\bbR)$, then
we have $\Gr^{(\sigma_2)}_{\bbR}|_{x}\is \Gr_{x,\bbR}$ (see Example \ref{m=1})
and 
the restriction 
\[u_{x,\bbR}:=u^{(\sigma_2)}_{\bbR}|_x:\Gr_{x,\bbR}\is\Gr^{(\sigma_2)}_{\bbR}|_{x}\to\Bun_\bG(\bP^1)_\bbR\] of \eqref{family of uniformizations}  is isomorphic to the real unifomization map in 
\eqref{real uniformization}. 
If $x\in\calH$, then
we have $\Gr^{(\sigma_2)}_{\bbR}|_x\is\Gr_x$
and we will call 
the map
\beq\label{Complex uniformization}
u_{x,\bbC}:=u^{(\sigma_2)}_{\bbR}|_x:\Gr_x\is\Gr^{(\sigma_2)}_{\bbR}|_x\longrightarrow\Bun_\bbG(\bbP^1)_\bbR
\eeq
the complex uniformization 
associated to $x$. 

We shall give a description of $u_{x,\bbC}$.
Let $(\mE,\phi)\in\Gr_x$ where $\mE$
is a $G$-bundle on $\bP^1$ and $\phi:\mE|_{\bP^1-\{x\}}\is G\times(\bP^1-\{x\})$ is a trivialization of $\mE$ over $\bP^1-\{x\}$.  
Let $c(\mE)$ be complex conjugation of $\mE$ (see Sect.~\ref{Stack of real bundles}) and 
let $\mF$ be the $G$-bundle on 
$\bP^1$ obtained from gluing of 
$\mE|_{\bP^1-\{\bar x\}}$ and $c(\mE)|_{\bP^1-\{x\}}$
using the isomorphism $c(\phi)^{-1}\circ\phi:\mE|_{\bP^1-\{x,\bar x\}}\is c(\mE)|_{\bP^1-\{x,\bar x\}}$. By construction, there is a canonical isomorphism 
$\gamma:\mF\is c(\mF)$ and the resulting real bundle 
$(\mF,\gamma)\in\Bun_{\bG}(\bP^1)_\bbR$ is the image $u_{x,\bC}((\mE,\phi))$.
Note that the cohomology class in $H^1(\Gal(\bC/\bbR),G)$
given by 
the restriction of the real bundle 
$\mF$ to $\infty$ is represented by the co-boundary $c(\phi(v))^{-1}(\phi)(v)$ (here $v\in\mE|_{\infty}$), hence is trivial. Thus the complex uniformization $u_{x,\bC}$ factors as 
\[u_{x,\bC}:\Gr_x\to\Bun_{G_\bbR}(\bbP^1(\bbR)).\]

We shall describe the image of $u_{x,\bbC}$. 
For each $z\in\bC^\times$ let $a_z:\mathbb P^1\ra\mathbb P^1$
be the multiplication map by $z$.
Consider the flows on $\Gr^{(2)}$ and $\Bun_\bG(\bP^1)$:
\beq\psi_z:\Gr^{(2)}\ra
\Gr^{(2)},\ 
(x,\mE,\phi)\ra (a_z(x),(a_{z^{-1}})^*\mE,(a_{z^{-1}})^*\phi).
\eeq
\[\psi_z:\Bun_\bG(\bP^1)\to\Bun_\bG(\bP^1),\ \ \mE\to (a_{z^{-1}})^*\mE\]
For $z\in\mathbb R_{>0}$ the flows above restrict to flows
\beq\label{flow on Gr_R^2}
\psi_z^1:\Gr^{(\sigma_2)}_\bbR\ra
\Gr^{(\sigma_2)}_\bbR\ \ \ \ \ \ \psi_z^2:\Bun_\bG(\bP^1)_\bbR\to\Bun_\bG(\bP^1)_\bbR
\eeq
and we have the following commutative diagram
\beq\label{diagram for psi_z}
\xymatrix{\Gr^{(\sigma_2)}_\bbR\ar[r]^{\psi_z^1}\ar[d]^{u_\bbR^{(\sigma_2)}}&\Gr^{(\sigma_2)}_\bbR\ar[d]^{u_\bbR^{(\sigma_2)}}
\\\Bun_\bG(\bP^1)_\bbR\ar[r]^{\psi_z^2}&\Bun_\bG(\bP^1)_\bbR}.
\eeq

\begin{lemma}\label{flows on Gr^(2)}
We have the following properties of the flows:
\begin{enumerate}
\item The critical manifold of the flow $\psi_z^1$ are the cores 
$C_\bbR^\lambda\subset \Gr_\bbR\is\Gr^{(\sigma_2)}_{\bbR}|_0$ 
and the stable manifold for $C_\bbR^\lambda$ is the 
strata $S_\bbR^\lambda\subset\Gr_\bbR$. 
\item For each $\lambda\in\Lambda_A^+$, we denote by 
\[\tilde T_\bbR^\lambda=\{\gamma\in\Gr^{(\sigma_2)}_\bbR|
\underset{z\ra 0}\lim\ \psi^1_z(\gamma)\in C_\bbR^\lambda\}\]
the corresponding unstable manifold. 
We have $\tilde T_\bbR^\lambda|_0\is T_\bbR^\lambda\subset\Gr_\bbR$ for 
$\lambda\in\Lambda_A^+$. The isomorphism 
$\Gr^{(\sigma_2)}_\bbR|_x\is\Gr_x,\ x\in\calH$, 
restricts to an isomorphism 
\[\tilde T_\bbR^\lambda|_x\is\mO_\bbR^\lambda\]
for $\lambda\in\Lambda_A^+$
\quash{
\item
The $\bbR_{>0}$-acton $\bbR_{>0}\times\Gr^{(2)}_\bbR\to\Gr^{(2)}_\bbR,\ (z,x)\to\psi^1_z(x)$
 extends to a 
$\bbR_{\geq 0}$-action:
\[\psi^1:\bbR_{\geq 0}\times\Gr^{(2)}_\bbR\to\Gr^{(2)}_\bbR,\]
such that $\psi^1(0,x)=\underset{z\ra 0}\lim\ \psi^1_z(x)$.}
\end{enumerate}
\end{lemma}
\begin{proof}
This is proved in \cite[Proposition 8.4]{N1}. 
\end{proof}

\begin{lemma}\label{flows on Bun_R}
\begin{enumerate}
\item
For any $\gamma\in\Gr_\bbR^{(\sigma_2)}$,
the map $\bbR_{>0}\to\Gr_\bbR^{(\sigma_2)},\ z\to\psi_z^1(\gamma)$
given by the flow $\psi_z^1$
extends to a map $a_\gamma:\bbR_{\geq 0}\to\Gr_\bbR^{(\sigma_2)}$ such that 
$a_\gamma(0)=\underset{z\to 0}\lim\ \psi_z^1(\gamma)$.
\item
For any $\mE\in\Bun_{G_\bbR}(\bbP^1(\bbR))$,
the action map $\bbR_{>0}\to\Bun_{G_\bbR}(\bbP^1(\bbR)),\ z\to\psi_z^2(\mE)$
given by the flow $\psi_z^2$
extends to a map 
\beq\label{extended action}
a_\mE:\bbR_{\geq 0}\to\Bun_{G_\bbR}(\bbP^1(\bbR)).
\eeq
Moreover, we have 
$a_\mE(z)\is\mE$ for all $z\in\bbR_{\geq 0}$, and 
for any $\gamma\in\Gr_\bbR^{(\sigma_2)}$,
there is a commutative 
diagram 
\beq\label{comm of psi^1}
\xymatrix{\bbR_{\geq 0}\ar[r]^{a_\gamma}\ar[rd]_{a_\mE}&\Gr^{(\sigma_2)}_\bbR\ar[d]^{u_\bbR^{(\sigma_2)}}\\
&\Bun_{G_\bbR}(\bbP^1(\bbR))}
\eeq
where $\mE=q(\gamma)\in\Bun_{G_\bbR}(\bbP^1(\bbR))$.

\end{enumerate}
\end{lemma}
\begin{proof}
Part (1) follows from Lemma \ref{flows on Gr^(2)} (2).
Proof of part (2).
Let $\gamma\in\Gr_\bbR$ and let 
$\mE=q(\gamma)\in\Bun_{G_\bbR}(\bbP^1(\bbR))$. Consider the 
the composed map 
\[
a_\mE:\bbR_{\geq 0}\stackrel{a_{\gamma}}\to\Gr_\bbR\to G_\bbR(\bbR[t^{-1}])\backslash\Gr_\bbR\is\Bun_{G_\bbR}(\bbP^1(\bbR))
\]
where $a_\gamma$ is the map in part (1) and the last isomorphism is the real uniformization (see Prop.\ref{uniformizations at real x}).
It is elementary to check that the map
$a_\mE$ only depends on $\mE$ and 
$a_\mE(z)=\psi_z^2(\mE)$ for $z\in\bbR_{>0}$, hence 
 defines the desired map in \eqref{extended action}. Moreover,
since $G_\bbR(\bbR[t^{-1}])$-orbits $T_\bbR^\lambda$ on $\Gr_\bbR$ are unstable manifolds for the flow
$\psi_z^1$, we have $a_\gamma(\bbR_{\geq 0})\subset T_\bbR^\lambda$ if $\gamma\in T_\bbR^\lambda$, and it implies 
$a_\mE(z)\is\mE$ for all $a\in\bbR_{\geq 0}$.
The commutativity of diagram \eqref{comm of psi^1} follows from
the construction of $a_\mE$.

\end{proof}

\begin{proposition}\label{uniformizations at complex x}
Let $x\in\calH$.
The complex uniformization $u_{x,\bbC}:\Gr_x\to\Bun_{G_\bbR}(\bbP^1(\bbR))$ 
induces an isomorphism  
\[LG_{x,\bbR}\backslash\Gr_x\stackrel{\sim}\longrightarrow\on{}\Bun_{G_\bbR}(\bbP^1(\bbR))\]
of semi-analytic stacks.
\end{proposition}
\begin{proof}
Let $\gamma\in\Gr_x$
and 
let $\mE=u_{x,\bbC}(\gamma)\in\Bun_{G_\bbR}(\bbP^1(\bbR))$
be the image of the complex uniformization map. By Lemma \ref{flows on Bun_R}(2)
we have 
\beq\label{eq 1}
|\mE|=|a_\mE(0)|=|
q(a_{\gamma}(0)|=|u_\bbR^{(\sigma_2)}(\underset{z\to 0}\lim\ \psi_z^1(\gamma))|,
\eeq
\beq\label{eq 2}
\underset{z\to 0}\lim\ \psi_z^1(\Gr_x\is\Gr^{(2)}_\bbR|_x)=\bigcup_{\lambda\in\Lambda_A^+}C_\bbR^\lambda.
\eeq
 As the image $\bigcup_{\lambda\in\Lambda_A^+}|q(C_\bbR^\lambda)|$ of the critical manifolds under 
 $u_\bbR^{(\sigma_2)}$
is equal to $|\Bun_{G_\bbR}(\bbP^1(\bbR))|$, equations \eqref{eq 1} and \eqref{eq 2} imply that 
$u_{x,\bC}$ factors through $u_{x,\bC}:\Gr\to\Bun_{G_\bbR}(\bbP^1(\bbR))$ and induces a surjection 
between the sets of isomorphism classes of objects.
Now a similar argument as in the proof Proposition \ref{uniformizations at real x} shows that 
$u_{x,\bC}:\Gr_x\to\Bun_{G_\bbR}(\bbP^1(\bbR))$ is surjective and hence   
an isomorphism $LG_{x,\bbR}\backslash\Gr_x\is\Bun_{G_\bbR}(\bbP^1(\bbR))$.
\end{proof}

\subsection{Multi-point uniformization}
We can also consider multi-point uniformization
\[u^{(n)}:\Gr^{(n)}\to\Bun_{\bbG}(\bbP^1)\times (\bP^1)^n, \ \ (x,\mE,\phi)\to (\mE,x)\] which 
exhibit $\Gr^{(n)}$ as a $LG^{(n)}$-torsor over $\Bun_{\bbG}(\bbP^1)\times (\bP^1)^n$.
Assume $n=2m$. 
the map $u^{(2n)}$ induces a map
\[u^{(\sigma_{2m})}_\bbR:\Gr^{(\sigma_{2m})}_\bbR\to\Bun_{\bbG}(\bbP^1)_\bbR\times (\bP^1)^m\]
on the corresponding real analytic stacks. 
Then  the same discussions in the previous sections apply to the multi-points setting and we have the following:

\begin{prop}\label{open in family}
The map $u^{(\sigma_{2m})}_\bbR$ 
induces  an isomorphism 
\[LG_\bbR^{(\sigma_{2m})}\backslash\Gr^{(\sigma_{2m})}_{\bbR}\stackrel{\sim}\lra\Bun_{G_\bbR}(\bbP^1(\bbR))\times(\bbP^1)^m.\]
of semi-analytic stacks.
\end{prop}


\section{Quasi-maps and Quillen's homeomorphism}\label{QMaps}
In this section we 
study the stack of quasi-maps and its real forms 
and use them to provide moduli interpretation for  
the quotient 
$LK_c^{}\backslash\Gr^{}$. Using our previous work \cite{CN},
we show that 
the 
family of real quasi-maps is topologically trivial
and deduce from it a multi-point version and refinement of 
of Quillen result \cite{M} on homotopy equivalences
 between 
loop spaces of symmetric varieties and  real affine Grassmannians.

\subsection{Definition of quasi-maps}
For $n\geq 0$, define the stack of quasi-maps with poles
$\QM^{(n)}(\bbP^1,X)$ to classify triples $(z, \calE, \sigma)$ comprising a point $z = (z_1, \ldots, z_n)\in (\bbP^1)^n$, a  $G$-torsor $\calE$ over $\bbP^1$, and a section $\sigma:\bbP^1 \setminus |z|\to \mE\times^GX$ where we write $|z| = \cup_{i =1}^n z_i \subset \bbP^1$. 
According to \cite{GN1}, $\QM^{(n)}(\bbP^1,X)$ is an ind-stack of ind-finite type.
Note the natural  maps  $\QM^{(n)}(\bbP^1,X) \to (\bP^1)^n$ and $\QM^{(n)}(\bbP^1,X) \to \Bun_\bbG(\bbP^1)$. 
Let $\QM^{(n)}(\bbP^1,X,\infty)$ be the ind-scheme classifies 
quadruple $(z,\mE,\phi,\iota)$ where $z\in \bC^n$,
$\mE$ is a $G$-bundle on $\bP^1$, $\phi:\bP^1-|z|\ra\mE\times^GX$, and 
$\iota:\mE_K|_{\infty}\is K$, here $\mE_K$ is the $K$-reduction of
$\mE$ on $C-|z|$ given by $\phi$. We have a natural map 
$\QM^{(n)}(\bbP^1,X,\infty)\ra \bC^n$. 
The ind-scheme $\QM^{(n)}(\bbP^1,X,\infty)\ra \bC^n$ is called rigidified quasi-maps.
Note that we have natural map 
\[
\QM^{(n)}(\bbP^1,X,\infty)\to\QM^{(n)}(\bbP^1,X)|_{\bC^n}
\]
sending $(z,\mE,\phi,\iota)$ to $(z,\mE,\phi)$
and it induces an isomorphism 
\beq\label{uni for QM(infty)}
K\backslash\QM^{(n)}(\bbP^1,X,\infty)\is
\QM^{(n)}(\bbP^1,X)|_{\bC^n}
\eeq
where the group $K$ acts on 
$\QM^{(n)}(\bbP^1,X,\infty)$
by changing the trivialization $\iota$.
For any $z\in(\bP^1)^n$ (resp. $z\in\bC^n$), we denote by
$\QM^{(n)}(\bbP^1,z,X)$ (resp. $\QM^{(n)}(\bbP^1,z,X,\infty)$)
the fiber at $z$.

\subsection{Real forms of quasi-maps}\label{real form of QM}
Let $n=2m$.
The twisted conjugation 
$\sigma_{2m}$ on $(\bP^1)^{2m}$ together with the involution 
$\eta$ on $G$ defines an involution on 
 $\QM^{(2m)}(\bbP^1,X)$ and we denote by
 $\QM^{(\sigma_{2m})}(\bbP^1,X)_\bbR$ the corresponding semi-analytic stack.
 Note there are natural  maps  
\[\QM^{(\sigma_{2m})}(\bP^1,X)_\bbR \to(\bP^1)^m\text{\ \ \ \ \ }\QM^{(\sigma_{2m})}(\bP^1,X)_\bbR \to \Bun_\bbG(\bbP^1)_\bbR.\]

Since $c(\infty)=\infty$, 
we have 
a real form of  
$\QM^{(\sigma_{2m})}(\bbP^1,X,\infty)_\bbR$ of $\QM^{(n)}(\bbP^1,X,\infty)$.
We have natural maps
\[\QM^{(\sigma_{2m})}(\bP^1,X,\infty)_\bbR \to\bC^m\text{\ \ \ \ \ }\QM^{(\sigma_{2m})}(\bP^1,X,\infty)_\bbR \to \Bun_\bbG(\bbP^1)_\bbR.\]

 \subsection{Morphisms}\label{morphisms}
Let $G_1$ and $G_2$ be two reductive groups with 
complex conjugations $\eta_1$ and $\eta_2$ and Cartan involutions $\theta_1$ 
and $\theta_2$ respectively. 
Then the constructions of quasi-maps, rigidified quasi-maps, uniformization
morphisms, and real forms of those are functorial with respect to 
homomorphism 
$f:G_1\to G_2$ that intertwine
$\eta_1,\eta_2$ and $\theta_1,\theta_2$.

\subsection{Uniformizations of quasi-maps}\label{uniform of quasi maps}
We have a natural uniformization map 
\beq\label{uni for QM}
q^{(n)}:\Gr^{(n)}\to\QM^{(n)}(\bP^1,X)
\eeq 
exhibits 
$\Gr^{(n)}$ as 
a $LK^{(n)}$-torsor on $\QM^{(n)}(\bP^1,X)$.
In particular, there is a canonical isomorphism of ind-stacks
\beq\label{iso for QM}
LK^{(n)}\backslash\Gr^{(n)}\stackrel{\sim}\lra \QM^{(n)}(\bP^1,
X).
\eeq
Assume $n=2m$.
The morphism in \eqref{uni for QM} 
gives rise to a uniformization map
 \[q^{(\sigma_{2m})}_\bbR:\Gr^{(\sigma_{2m})}_{\bbR}\lra 
\QM^{(\sigma_{2m})}(\bP^1,X)_\bbR\]
on the corresponding semi-algebraic stacks of real points.
The following proposition follows from Proposition \ref{open in family}:

\begin{prop}\label{unif of QM}
The uniformization map $q^{(\sigma_{2m})}_\bbR$ induces  an isomorphism
\[LK^{(\sigma_{2m})}_c\backslash\Gr^{(\sigma_{2m})}_{\bbR}\is\QM^{(\sigma_{2m})}(\bP^1,X)_{\bbR}.\]

\end{prop}

\begin{example}\label{fibers of real quasi maps}
By Lemma \ref{fibers: compact case} and 
Proposition \ref{multi Gram-Schmidt}, 
we have
$LK_c^{(\sigma_{2m})}|_{\bP^1(\bbR)^m}:=L^-K_c^{(m)}|_{\bbR^m}\is K_c\times\bbR^m$
, $LK_c^{(\sigma_{2m})}|_{\calH^m}\is LK_c^{m}$
and $\Gr^{(m)}|_{\calH^m}\is\Omega G_c^{(m)}$
and hence there are isomorphisms 
\beq\label{real and imaginary fibers}
\QM^{(\sigma_{2m})}(\bP^1,X)_{\bbR}|_{\calH^m}\is  LK_c^{(m)}\backslash
\Gr^{(m)}|_{\calH^m}\is LK_c^{(m)}\backslash\Omega G_c^{(m)}
\eeq
\[\QM^{(\sigma_{2m})}(\bP^1,X)_{\bbR}|_{\bP^1(\bbR)^m}\is K_c\backslash\Gr^{(m)}_\bbR\]

\end{example}

\subsection{Quillen's homeomorphism}
We fist recall some results in \cite{CN1} 
on quasi-maps for complex groups.

Consider the
complex group $G\is G\backslash (G\times G)$ viewed as a symmetric variety for $G\times G$. 
It corresponds to the  swap involution $\on{sw}(g_1,g_2)=(g_2,g_1)$
of 
$G\times G$ and 
the corresponding conjugation on $G\times G$ is given by 
$\on{sw}_c(g_1,g_2)=(\eta_c(g_2),\eta_c(g_1))$.
Consider  the real quasi-maps space 
$\QM^{(\sigma_{2m})}(\bbP^1,G,\infty)_{\mathbb R}\to\bC^m$
associated to $G\backslash (G\times G)$.

The natural projections 
$p_1,p_2:G\times G\to G$
induces natural isomorphism 
\[\Gr_{G\times G}^{(2m)}\is\Gr_{}^{(2m)}\times_{(\bP^1)^{2m}}\Gr^{(2m)}.\]
Consider  the base changes
$\Gr^{(2m)}|_{(\bP^1)^m}$ and $\Gr^{(2m)}_{G\times G}|_{(\bP^1)^m}$
along the embedding $(\bP^1)^m\hookrightarrow(\bP^1)^{2m}, (z_1,...,z_m)\to (z_1,...,z_m,\bar z_1,...,\bar z_m)$.

We have a natural embedding
\beq\label{splitting}
\Gr^{(m)}\to\Gr^{(2m)}|_{(\bP^1)^m}
\eeq
sending $(z\in(\bP^1)^m,\mE,\phi)\to ((z,\bar z)\in(\bP^1)^{2m},\mE,\phi|_{\bP^1\setminus z\cup\bar z})$.
Note that $\Gr^{(\sigma_{2m})}_{G\times G,\bbR}\subset\Gr^{(2m)}_{G\times G}|_{(\bP^1)^m}$ and 
it was proved in \cite[Corollary 5.2 and Corollary 5.4]{CN}
that 
the natural projection map
\[\Gr^{(2m)}_{G\times G}|_{(\bP^1)^m}\stackrel{}
\is\Gr_{}^{(2m)}\times_{(\bP^1)^{2m}}\Gr^{(2m)}|_{(\bP^1)^m}\stackrel{p_1}\lra\Gr^{(2m)}|_{(\bP^1)^m}\]
restricts to an isomorphism 
\beq\label{projection}
\Gr^{(\sigma_{2m})}_{G\times G,\bbR}\is\Gr^{(2m)}|_{(\bP^1)^m}
\eeq
and the composed map 
\[\Gr^{(m)}\stackrel{\ref{splitting}}\to\Gr^{(2m)}|_{(\bP^1)^m}\stackrel{~\eqref{projection}}\is \Gr^{(\sigma_{2m})}_{G\times G,\bbR}\]
induces a real analytic isomorphism
\beq\label{trivialization}
\Gr^{(m)}|_{\bC^m}\is\Omega G_c^{(\sigma_{2m})}\backslash\Gr^{(\sigma_{2m})}_{G\times G,\bbR}|_{\bC^m}\is
\QM^{(\sigma_{2m})}(\bbP^1,G,\infty)_{\mathbb R}.
\eeq

Let 
$\bbS:=\bbR^m\times i\bbR$ and viewed as a subset of 
$\bC^m$ via the embedding 
\beq\label{iota}
\bbS\lra \bC^m,\ \ \
(b_1,...,b_m,ia)\to (b_1+ia,...,b_m+ia)\eeq
We now consider the base change of the isomorphism~\eqref{trivialization} to $\calS$:
\beq\label{v}
v:\Gr^{(m)}|_\calS\is\QM^{(\sigma_{2m})}(\bbP^1,G,\infty)_{\mathbb R}|_\calS.
\eeq
The conjugation $\eta\times\eta$
on $G\times G$ commutes with $\on{sw}_c$, together with 
the complex conjugation on $\bP^1$, it 
defines an involution $(\eta\times\eta)^{(\sigma_{2m})}$
on $\QM^{(\sigma_{2m})}(\bbP^1,G,\infty)_{\mathbb R}|_\calS$
and, via the isomorphism $v_\calS$, it induces an involution
$\eta^{(\sigma_{2m})}$ on $\Gr^{(m)}|_\bbS$.

The following description of $\eta^{(\sigma_{2m})}$ is proved in \cite[Proposition 5.12]{CN1}. 
\begin{proposition}
The involution $\eta^{(\sigma_{2m})}$ on $\Gr^{(m)}|_\bbS$ is given by the following formula:
\[\eta^{(\sigma_{2m})}=\eta^{(m)}
\text{\ \ \ \ on\ \ \ \ }\ \ \Gr^{(m)}|_{\bbR^m\times\{0\}}\]
\[\eta^{(\sigma_{2m})}=\tilde\theta
\text{\ \ \ \ \ \ on\ \ \ \ }\ \ \Gr^{(m)}|_{\bbR^m\times i\bbR^\times}\is\Omega G_c^{(2m)}|_{\bbR^m\times i\bbR^\times}.\]
\end{proposition}
Denote by
\beq
\xymatrix{
\calQ^{(m)}:=(\Gr^{(m)}|_\calS)^{\eta^{(\sigma_{2m})}}\lra\bbS=\bbR^m\times i\bbR
}
\eeq
 the fixed-points. 
 Then the isomorphism~\eqref{v}
restricts to natural isomorphisms 
\beq\label{fibers of Q}
\xymatrix{
v_0:\calQ^{(m)}|_{\bbR^m\times\{0\}}\ar[r]^-\sim&(\Gr^{(m)})^{\eta^{(m)}}|_{\bbR^m\times\{0\}}=\Gr^{(m)}_\bbR}|_{\bbR^m\times\{0\}}
\eeq 
\[
\xymatrix{
v_{a\neq0}:\calQ^{(m)}|_{\bbR^m\times\{ia\}}\ar[r]^-\sim&(\Omega G_c^{(m)})^{\tilde\theta}|_{\bbR^m\times\{ia\}}\is\Omega X_c^{(m)}|_{\bbR^m\times\{ia\}}}
\]

The spherical strata $S^{(m),\lambda_\fp}\subset\Gr^{(m)}$ 
for $\lambda_\fp
:\fp\to\Lambda_T^+$ (resp. cospherical  $T^{(1),\lambda} \subset \Gr^{(1)}$
for 
$\lambda\in \Lambda_T^+$) give rise to spherical  strata
 $\calQ^{(m),\lambda_\fp}=S^{(m),\lambda_\fp}\cap\calQ^{(m)}$ 
 of $\calQ^{(m)}$
(resp. cospherical strata
$\calQ^{(1)}_{\lambda}=T^{(1),\lambda}\cap\calQ^{(1)}$ of $\calQ^{(1)}$).
Then the isomorphisms~\eqref{fibers of Q}
restrict to
\beq\label{fiber of X}
\xymatrix{\ \calQ^{(m),\lambda_\fp}|_{\bbR^\fp\times\{0\}}\ar[r]^-\sim& S^{(m),\lambda_\fp}_{\bbR}|_{\bbR^\fp\times\{0\}}\ \ \ \  (resp.\ \ \ \ 
\calQ^{(1)}_\lambda|_{\bbR\times\{0\}}\ar[r]^-\sim&T^{(1),\lambda}_{\bbR}}|_{\bbR\times\{0\}})
\eeq
\[\xymatrix{\calQ^{(m),\lambda_\fp}|_{\bbR^\fp\times ia}\ar[r]^-\sim& P^{(m),\lambda_\fp}|_{\bbR^\fp\times\{ia\}}\ \ \ (resp.\ \ \calQ^{(1)}_{\lambda}|_{\bbR\times{ia}}\ar[r]^-\sim& Q^{(1),\lambda}|_{\bbR\times\{ia\}}.})\]
As $S_\bbR^{(m),\lambda}$ (resp. $T^{(1),\lambda}_\bbR$) is non-empty if and only if $\lambda_\fp:\fp\to\Lambda_A^+\subset\Lambda_T^+$
(resp. 
$\lambda\in\Lambda_A^+$), the same applies to the stratum 
$\calQ^{(m),\lambda_\fp}$ (resp. $\calQ^{(1)}_\lambda$).
We have the following:
\begin{prop}\label{triv of Q}
Equip
  $\calQ^{(m)}$ (resp. $\calQ^{(1)}$)
and $\bbS=\bbR^m\times i\bbR$  with the stratifications
 $\{\calQ^{(m),\fp_\lambda}\}$ (resp. $\{\calQ^{(1),\lambda}\}$ or $\{\calQ^{(1)}_\lambda\}$)
and $\{\bbR^\fp\times\bbR\}_{\fp\in p(m)}$ respectively.
There is a $K_c$-equivariant stratified trivilization of the family 
\[\calQ^{(m)}\to \bbS=\bbR^m\times i\bbR\] 
over $i\bbR$.
\end{prop}
\begin{proof}
Since the Beilinson-Drinfeld Grassmannian $\Gr^{(m)}|_\calS\to \calS$ is ind-proper and 
there is a $K_c$-equivariant trivialization 
$\Gr^{(m)}|_\calS\is\Gr^{(m)}|_{\bbR^m\times\{0\}}\times i\bbR$ over $i\bbR$
the desired claim follows from the general lemma in \cite[Lemma 6.6]{CN1}.

\end{proof}

Now the homeomorphisms in~\eqref{fibers of Q} and~\eqref{fiber of X} together with 
Proposition \ref{components of multi loops of X} and Proposition \ref{triv of Q} imply the following

\begin{thm}\label{Quillen}
There are $K_c$-equivariant strata-preserving homeomorphisms
\[\Omega K_c^{(m)}\backslash\Gr^{(m)}|_{\bbR^m\times\{i\}}\is\Omega X_c^{(m)}|_{\bbR^m\times\{i\}}\is\Gr^{(m)}_\bbR|_{\bbR^m\times\{0\}}\]
which restrict to $K_c$-equivariant homeomorphisms
\[ \Omega K_c^{(m)}\backslash\calO_K^{(n),\lambda_\fp}\is P^{(m),\lambda_\fp}\is S^{(m),\lambda_\fp}_\bbR\]
\[ \Omega K_c^{(1)}\backslash\calO_K^{(1),\lambda}\is P^{(1),\lambda}\is S_\bbR^{(1)}\ \ \ \ \ 
\Omega K_c^{(1)}\backslash\calO_\bbR^{(1),\lambda}\is Q^{(1),\lambda}\is T^{(1),\lambda}_\bbR\ \ \text{if}\ \ m=1\]

\end{thm}

\subsection{Trivialization of real quasi-maps}
Consider the morphism $f:G\to G\times G,\ g\to (g,\theta(g))$.
It is equivariant for the conjugations and Cartan involutions on 
$G$ and $G\times G$ hence, by the functoriality noted in Section \ref{morphisms}, we obtain a map 
\beq\label{f}
\QM^{(\sigma_{2m})}(\bbP^1,X,\infty)_{\mathbb R}|_\calS\lra \QM^{(\sigma_{2m})}(\bbP^1,G,\infty)_{\mathbb R}|_\calS\stackrel{\eqref{v}}\is\Gr^{(m)}|_\calS.
\eeq

\begin{lemma}\label{iso to Q}

The map \eqref{f} restricts to a $K_c$-equivariant homeomorphism
\[\xymatrix{
\QM^{(\sigma_{2m})}(\bbP^1,X,\infty)_{\mathbb R}|_{\calS}\ar[r]^-\sim&\calQ^{(m)}}\ \ \ \]
In particular, $\QM^{(\sigma_{2m})}(\bbP^1,X,\infty)_{\mathbb R}|_\calS\to\calS$ is topologically trivial
over $i\bbR$
\end{lemma}

\begin{proof}
According to Example \ref{fibers of real quasi maps}
and~\eqref{fibers of Q},
there are natural identifications
\[\QM^{(\sigma_{2m})}(\bbP^1,X,\infty)_{\mathbb R,\calL}\is\Gr^{(m)}_{\bbR},\ \ \ \ \ \ 
\calQ^{(m)}\is\Gr^{(m)}_\bbR\text{\ \ \ \ \ over\ \ \ \ }\bbR\times\{0\}\]
\[\QM^{(\sigma_{2m})}(\bbP^1,X,\infty)_{\mathbb R,\calL}\is\Omega K_c^{(m)}\backslash\Omega G_c^{(m)},\ \ \ \ \ \ 
\calQ^{(m)}\is\Omega X_c^{(m)}\text{\ \ \ \ \ over\ \ \ \ }\bbR\times i\bbR^\times\]
and under the above identifications the map~\eqref{f}
specializes to the natural inclusion map
\beq
\xymatrix{
\Gr^{(m)}_{\bbR}\ar@{^(->}[r] & \Gr^{(m)} &\text{over\ \ \ \ }  \bbR\times\{0\} 
}\eeq
and
to the natural map 
\beq\label{q_a}
\xymatrix{
\Omega K_c^{(m)}\backslash\Omega G_c^{(m)}\ar[r]^-{} &  \Omega X^{(m)}_c
& \text{over\ \ \ \ }  \bbR\times i\bbR^\times 
}
\eeq
The lemma follows from Proposition \ref{components of multi loops of X}.

\end{proof}

\subsection{Flows on quasi-maps}\label{flows}
We have a flow $\psi_z, z\in\bC^\times$ 
on $\QM^{(2)}(\bbP^1,X,\infty)$
given by: 
 \beq\psi_z:\QM^{(2)}(\bbP^1,X,\infty)\ra
\QM^{(2)}(\bbP^1,X,\infty)
\eeq\ \ \ 
\[\ 
\psi_z((x,\mE,\psi,\iota))=(a_z(x),(a_{z^{-1}})^*\mE,(a_{z^{-1}})^*\psi,\iota).
\]
\quash{
We have the following commutative diagrams 
\[\xymatrix{\QM^{(\sigma_2)}(\bbP^1,G,K)\ar[r]^{\psi_z}\ar[d]&\QM^{(2)}(\bbP^1,G,K)\ar[d]
\\(\bbP^1)^2\ar[r]^{a_z}&(\bbP^1)^2}\ \ \ \ \ \ 
\xymatrix{\QM^{(\sigma_2)}(\bbP^1,G,K,\infty)\ar[r]^{\psi_z}\ar[d]&\QM^{(2)}(\bbP^1,G,K,\infty)\ar[d]
\\\bC^2\ar[r]^{a_z}&\bC^2}.\]
}
For $z\in\mathbb R_{>0}$ the isomorphism $\psi_z$ restricts to a flow
\[\psi^{3}_z:\QM^{(\sigma_2)}(\bbP^1,X,\infty)_\bbR\ra\QM^{(\sigma_2)}(\bbP^1,X,\infty)_\bbR,\]
and we have the following commutative diagrams 
\[\xymatrix{\QM^{(\sigma_2)}(\bbP^1,X,\infty)_\bbR\ar[r]^{\psi^{3}_z}\ar[d]&\QM^{(\sigma_2)}(\bbP^1,X,\infty)_\bbR\ar[d]
\\\bbC\ar[r]^{a_z}&\bbC}\]

\begin{lemma}\label{properties of flows}
We have the following properties of the flows:
\begin{enumerate}
\item The flow $\psi^{3}_z$ on $\QM^{(\sigma_2)}(\bbP^1,X,\infty)_\bbR$ is 
$K_c$-equivariant.
\item
Recall the flow 
$\psi_z^1$ on $\Gr^{(\sigma_2)}_\bbR$~\eqref{flow on Gr_R^2}.
We have the following commutative diagram
\beq\label{compatibility with flows}
\xymatrix{\Gr^{(\sigma_2)}_\bbR|_\bC\ar[r]^{\psi_z^1}\ar[d]&\Gr^{(\sigma_2)}_\bbR|_\bC\ar[d]\\
\QM^{(\sigma_2)}(\bbP^1,X,\infty)_\bbR\ar[r]^{\psi_z^3}&\QM^{(\sigma_2)}(\bbP^1,X,\infty)_\bbR}.
\eeq

\item For each $\lambda\in\Lambda_A^+$, the core $C_\bbR^\lambda\subset \Gr_\bbR
\subset\QM^{(\sigma_2)}(\bP^1,X,\infty)_\bbR|_0$ is a union of components of 
the critical manifold of the flow $\psi_z^3$ on $\QM^{(\sigma_2)}(\bbP^1,X,\infty)_\bbR$ 
and the stable manifold for $C_\bbR^\lambda$ is the 
strata $S_\bbR^\lambda\subset\Gr_\bbR$. 
\item For each $\lambda\in\Lambda_A^+$, we denote by 
\[\tilde T_\bbR^\lambda=\{x\in\QM^{(\sigma_2)}(\bbP^1,X,\infty)_\bbR|
\underset{z\ra 0}\lim\psi_z^3(x)\in C_\bbR^\lambda\}\]
the corresponding unstable manifold. 
We have $\tilde T_\bbR^\lambda|_0\is T_\bbR^\lambda\subset\Gr_\bbR$ for 
$\lambda\in\Lambda_A^+$. The open embedding 
$\Omega K_c\backslash\Gr\to\QM^{(\sigma_2)}(\bP^1,X,\infty)_\bbR|_i$ restricts to an isomorphism 
\[\Omega K_c\backslash\mO_\bbR^\lambda\is\tilde T_\bbR^\lambda|_i\]
for $\lambda\in\Lambda_A^+$.

\end{enumerate}
\end{lemma}
\begin{proof}
Part (1) and (2) follows from the construction of the flows. 
Part (3) and (4) 
follows from Lemma \ref{flows on Gr^(2)} and diagram (\ref{compatibility with flows}).

\end{proof}

\section{Real-symmetric equivalence}\label{real sym}
In this section we construct the real-symmetric equivalence 
between the real and relative Satake categories.
The theory of  sheaves on infinite dimensional stacks 
developed in \cite{BKV} plays an important role here.

\subsection{Placid stacks}
We first review the notion of placid (ind)-schemes and (ind)-stacks.
Let $Y$ be a scheme 
acted on by an affine group scheme $H$.
we say that
$Y$ is $H$-\emph{placid} if 
\begin{itemize}
\item  $Y$  can be written as filtered limit $Y = \lim_{j} Y_j$, 
where each $Y_j$ is a $H$-scheme of finite type and the transition maps 
$Y_{j'} \to Y_{j}$ for $j\to j'$ are affine, smooth, surjective, and $H$-equivariant.
\item
the action $H \times Y_j \to Y_j$  factors through 
a group scheme $H_j$ of finite type. Moreover
the $H_j$ can be chosen so that $\{ H_j \}_{j}$ forms a projective
system with $H = \on{lim}_j H_j$ and the transition maps $H_{j'} \to H_j$
are smooth surjective with \emph{unipotent} kernel. 
\end{itemize}

Let $Z\subset Y$ a $H$-invariant subscheme.
We shall say that the inclusion $Z\to Y$
is placid 
if there is a presentation of $Y=\on{lim} Y_j$ as above and 
and index $j$ and a $H$-invariant subscheme $Z_j\subset Y_j$
such that 
$Z\cong Z_j\times_{Y_j}Y$.

Let $Y$ be an ind-scheme acted on by an affine group scheme $H$.
We say that $Y$ is $H$-ind placid if can be written as  filtered colimit
\begin{equation}\label{placid presentation}
Y\cong\on{colim}_i Y^i
\end{equation}
where each $Y^i$ is $H$-placid  
and the transition maps $Y^i\to Y^{i'}$ are
placid closed embeddings.
A presentation as in~\eqref{placid presentation}
is called a placid presentation of $Y$.

We call a stack $\calY$ a placid stack (resp.
ind-placid stack) if it is isomorphic to a stack of the form 
$\calY\is Y/H$ where $Y$ is  $H$-placid scheme (resp. $H$-ind placid scheme).

We recall the following basic result due to Drinfeld and 
 Bouthier:

\begin{prop}\label{ind-placidness}
\begin{enumerate}
\item \cite[Theorem 6.3]{D} and \cite[Proposition 2.0.1]{B}
For any smooth affine $\bC$-scheme $Y$ of finite type, the loop space 
$Y(\calK)$ is ind-placid.
\item
\cite[Proposition 3.8]{D}  and
\cite[Proposition 1.1.4]{B} Let $G$ be a complex connected reductive group and let 
$H\subset G$ be a connected reductive subgroup.
The natural map $G(\calK)\to (G/H)(\calK)$ is a 
$H(\calK)$-torsor in the  $h$-topology.
In particular, there is an isomorphism of stacks 
\[G(\calK)/H(\calK)\is (G/H)(\calK)\]
\end{enumerate}
\end{prop}

\begin{remark}
Part (2) above is equivalent to the following fact: 
Let $A$ be a $\bC$-algebra.
Any $G$-torsor on $\on{Spec}(A((t)))$ is trivial over 
$\on{Spec}(B((t)))$
for a $h$-cover $\on{Spec}(B)\to\on{Spec}(A)$.
This  fact is claimed in \cite[Proposition 3.8, Remark (b)]{D} 
with sketch of proof.
In \cite{B}, the author provided a different proof.

\end{remark}
In \cite[Proposition 26]{CY},
we deduce 
the following results from Proposition \ref{ind-placidness} (1).\footnote{In \emph{loc. cit.} we only deal with the case of classical symmetric varieties
using an explicit construction of (ind)-plaicd presentations. 
The proof of the general case, using Proposition \ref{ind-placidness}, will appear in the revised version of the paper.}
\begin{prop}\label{placid of X(K)}
 (1) $X(\calK)$  is $G(\calO)$-ind placid.
(2)  The orbits closure $\overline{X(\calK)^\lambda}$ is $G(\calO)$-placid.
(3) The inclusion $X(\calK)^\lambda\to \overline{X(\calK)^\lambda}$
is placid.
\end{prop}

Proposition \ref{ind-placidness} (2) immediately implies 
\begin{corollary}\label{iso of stacks}
(1) The stack $K(\calK)\backslash\Gr$ is ind-placid.
(2) The stack $K(\calK)\backslash\overline\mO^\lambda_K$
is placid  and the inclusion $K(\calK)\backslash\mO^\lambda_K\to
K(\calK)\backslash\overline\mO^\lambda_K$ is placid.
(3) There is an isomorphism of ind-placid stacks $X(\calK)/G(\calO)\is K(\calK)\backslash\Gr$ and isomorphisms of placid stacks
$K(\calK)\backslash\overline\mO^\lambda_K\is \overline {X(\calK)^\lambda} /G(\calO)$.

\end{corollary}

\subsection{Stacks admitting gluing of sheaves}

We recall the notion of stacks admitting gluing of sheaves in \cite[Section 5.5]{BKV}.
Let $\calY$ be stack and let 
$\eta:\calS\to \calY$ be a locally closed embedding of finite presentation (fp-locally closed embedding for short).
According to \cite[Section 5.4.4]{BKV}, there are well-defined continuous functor 
$\eta_*:D(\calS)\to D(\calY)$
 
\begin{definition}
We say a stack $\calY$ admits gluing of sheaves 
if for every  fp-locally closed embedding $\eta:\calS\to\calY$ 
the functors $\eta_*:D(\calS)\to D(\calY)$ admits a left adjoint $\eta^*:D(\calY)\to D(\calS)$.
\end{definition} 

We have following basic facts:
\begin{lemma}\cite[Lemma 5.5.3]{BKV}\label{fiber sequence}
Let $\calY$ be a stack admitting gluing of sheaves.
(1) Let $\eta:\calS\to\calY$ be a fp-locally closed embedding.
There exists a left adjoint $\eta_!$ of $\eta^!:D(\calS)\to D(\calX)$
(2) Let $j:\calU\to\calY$ be a fp-open embedding and $i:\calZ\to\calY$
be the complementary fp-closed embedding.
For every $\mF\in D(X)$ there is a fiber sequence
\[j_!j^!\mF\to\mF\to i_*i^*\mF.\]
\end{lemma}

\begin{lemma}\cite[Proposition 5.5.7]{BKV}\label{quotient is placid}
Let $H$ be an ind-placid group, that is, a group object in 
ind-placid schemes, acting on  an ind-placid scheme $Y$.
The quotient stack $Y/H$ admits gluing of sheaves.
\end{lemma}

\begin{example}
Let $H$ be a complex affine group scheme of finite type acting on 
 a complex smooth affine scheme $Y$ of finite type.
 Then the  loop group $H(\calK)$ 
is an ind-placid group and the loop space $Y(\calK)$ is an ind-placid scheme, and 
 Lemma \ref{quotient is placid} implies the 
 quotients 
 $H(\calK)\backslash Y(\calK)$ 
   admits glueing of sheaves.

\end{example}

The following technical lemma 
will be used in Section \ref{fusion product for X}.

\begin{definition}\label{strongly pro-smooth}
Let $f:\calX\to\calY$ be a morphism between stacks.
We say
$f$ is \emph{strongly pro-smooth} is 
  there is a surjective morphism 
 $Y\to \calY$ such that the 
pullback $f_Y:X:=\calX\times_\calY Y\to Y$ satisfying the following:
 there is a presentation $X\is\on{lim}_{i\in I}X_i$
 as filtered limit of ind-stacks ind-locally of finite type 
 over $Y$ such that all the projection maps 
 $X_i\to Y$ are fp-smooth morphisms
 and the transition maps $X_i\to X_{i'}$
 are fp-affine smooth morphisms.
\end{definition}

 \begin{lemma}\label{key lemma}
 Let $\calX$ is an ind-stack of ind-locally of finite type.  
 \begin{enumerate}
\item $\calX$ admits gluing of sheaves.
 \item Let $f:\calX\to\calY$ be a strongly pro-smooth morphism.  
 Then  $\calY$ admits gluing of sheaves and 
 for any fp-locally closed embedding $\eta:\calS\to \calY$ the 
 base change morphism 
 $\eta_{\calS}^* f^!\to f_\calS^!\eta^*$ (resp. $(\eta_{\calS})_!f_\calS^! \to f^!\eta_!$),
is isomorphism.
Here  $f_\calS$ and 
 $\eta_\calS$ are the natural projection maps in the following fiber product
 \[\xymatrix{\calX_\calS\ar[r]^{\eta_\calS}\ar[d]^{f_\calS}&\calX\ar[d]^f\\
 \calS\ar[r]^{\eta}&\calY}\] 
\end{enumerate}

  \end{lemma}
  \begin{proof}
  Choose a presentation 
\beq\label{presentation of X}
\calX\is\on{colim}_{\beta}\calX^\beta \is\on{colim}_{\beta\in I}\on{colim}_{\alpha}\calX^\beta_\alpha
\eeq
  as a filtered colimit, 
  where $\calX^\beta_{\alpha}$ are stacks 
of finite type and the transition maps $\calX^{\beta}\to\calX^{\beta'}$ (resp. $\calX^\beta_\alpha\to\calX^\beta_{\alpha'}$)
  are closed embeddings of finite presentation (resp. open embeddings of finite presentation).
 It follows from \cite[Example 1.3.4]{BKV} that 
 $\calX^\beta_\alpha$ is \emph{placid} in the sense of \cite[Section 1.3]{BKV}
 and \cite[Lemma 5.5.5, and Lemma 5.5.6]{BKV} implies
$\calX$ admits glueing of sheaves.

Let $\calX^{[n]}$ (resp. $\calX_\calS^{[n]}$)
be the terms of the $\check{\on{C}}$ech complex associated to 
$f$ (resp. $f_\calS$) 
and let $\eta^{[n]}_\calS:\calX_{\calS}^{[n]}\to\calX^{[n]}$
pullback of  $\eta_S$ along the projection map
$\calX_{}^{[n]}\to\calX$.
 We claim that 
\begin{enumerate}
\item each $\calX^{[n]}$
admits glueing of sheaves 
\item  for every natural projection
$f:\calX^{[m]}\to\calX^{[n]}$, with pullback $f_\calS:\calX_\calS^{[m]}\to\calX_\calS^{[n]}$,
there is a  natural isomorphism
\[f^!(\eta^{[n]}_\calS)_*\is (\eta^{[m]}_\calS)_*f_\calS^!\]
\item 
(Beck-Chevalley condition)
For every projection $f$ as in (2), 
there is a natural isomorphism
\beq\label{base change}
(\eta^{[m]}_\calS)^*f^!\is f_\calS^!(\eta^{[n]}_\calS)^*\ \ \ \ 
\eeq
\end{enumerate}

Now it follows from \cite[Proposition 5.1.8]{BKV} that 
the functor $\eta_*:D(\calS)\is\on{colim}_{[n]} D(\calX_\calS^{[n]})\to D(\calY)\is\on{colim}_{[n]}D(\calX^{[n]})$ admits a left adjoint $\eta^*$ satisfying the 
base change morphism $\eta_{\calS}^* f^!\is f_\calS^!\eta^*$.
The lemma follows.

Proof of the claim.
Claim (2) follows from \cite[Lemma 5.4.5]{BKV}.
Note that the presentation~\eqref{presentation of X} of $\calX$ induces a presentation 
 \[\calX^{[n]}\is\on{colim}_{\beta}\calX^{[n],\beta} \is\on{colim}_{\beta}\on{colim}_{\alpha}\calX^{[n],\beta}_\alpha\]
where $\calX^{[n],\beta}_\alpha=\calX^{[n]}\times_{\calX}\calX^{\beta}_{\alpha}$.
The assumption that $f$ is strongly pro-smooth implies that the projection $\calX^{[n],\beta}_\alpha\to \calX^{\beta}_\alpha$
is \emph{smooth}  in the sense of \cite[Section 1.3.1 (c)]{BKV} and, since $\calX^{\beta}_\alpha$ is \emph{placid},
it follows from 
\cite[Section 1.3.3 (d), Lemma 5.5.5, and Lemma 5.5.6]{BKV} that $\calX^{[n],\beta}_\alpha$
is \emph{plaicd} and $\calX^{[n]}$ admits glueing of sheaves.
Claim (1) follows.

To show claim (3), we first assume the support of $\mF$ is contained in some $\calX^\beta_{\alpha}$.
Then the desired claim follows from \cite[Proposition 5.3.9]{BKV}.
Since every object $\mF\in D(\calX)$ can be written as a colimit
$\mF\is\on{colim}_{\alpha}\on{colim}_{\beta}\mF^\beta_{\alpha}$ 
with respect to $!$-pushforward along open embeddings $\calX^\beta_\alpha\to\calX^\beta_{\alpha'}$
and closed embeddings $\calX^\beta\to\calX^{\beta'}$ (see, e.g., \cite[Corollary 5.1.5]{BKV}), the general case follows from the fact that 
the functors in~\eqref{base change} commutes with colimits.

The proof of the base change isomorphism 
$(\eta_{\calS})_!f_\calS^! \to f^!\eta_!$ is similar. The  only non-trivial part is the 
the Beck-Chevalley condition 
in (3) and it follows from the same argument above using \cite[Proposition 5.3.9]{BKV}.

  \end{proof}

   \begin{proposition}\label{key base change}
(1) The quotient $LK^{(n)}\backslash\Gr^{(n)}$ is an ind-stack ind-locally of finite type 
(2) the natural projection map 
  $f:LK^{(n)}\backslash\Gr^{(n)}\to K(\calK)^{(n)}\backslash\Gr^{(n)}$
  is strongly pro-smooth. 
  
 In particular, both stacks $LK^{(n)}\backslash\Gr^{(n)}$ and 
  $K(\calK)^{(n)}\backslash\Gr^{(n)}$ admit gluing of sheaves.
   
    \end{proposition}
    \begin{proof}
Since the stack of quasi maps $QM^{(n)}(\bP^1,X)$ is 
an ind-stack ind-locally of finite type, 
the uniformization isomorphism $LK^{(n)}\backslash\Gr^{(n)}\is QM^{(n)}(\bP^1,X)$ in~\eqref{uni for QM} implies 
part (1).
On the other hand, the pullback of $f$ along $\Gr^{(n}\to K(\calK)^{(n)}\backslash\Gr^{(n)}$ is isomorphic to the natural projection map 
\[\Gr^{(n)}\times_{(\bP^1)^n}K(\calK)^{(n)}/LK^{(n)}\lra \Gr^{(n)}.\]
Note that the 
quotient $K(\calK)^{(n)}/LK^{(n)}$ is the 
$K(\calO)^{(n)}$-torsor over $(\bbP^1)^n\times\Bun_K(\bbP^1)$
classifying a $K$-bundle $\mE_K\in\Bun_K(\bbP^1)$ on $\bP^1$, a point $\underline z\in(\bP^1)^n$,
and a trivialization of $\mE_K$ on the formal neighborhood of the points $\underline z$ in $\bP^1$.
Since $K(\calO)^{(n)}$ is a projective limit of smooth schemes with smooth affine transition maps
 (see, e.g., \cite[Lemma 2.5.1]{R})) and 
 $\Bun_K(\bbP^1)$ is a smooth stack locally of finite type
  it implies that 
 $f$ is strongly pro-smooth. 
 The proposition follows.

    \end{proof}

\subsection{Real and relative Satake categories}

The  loop group $K(\calK)$, arc group $G(\calO)$, and real arc group $G_\bbR(\mO_\bbR)$) act natural on $\Gr$, $X(\calK)$, and $\Gr_\bbR$ 
and we can form the quotient stacks $K(\calK)\backslash\Gr$ and  $X(\calK)/G(\calO)$
and the quotient semi-analytic stack $G_\bbR(\mO_\bbR)\backslash\Gr_\bbR$ (see Appendix \ref{quotient stack}).  
Let $D(K(\calK)\backslash\Gr)$, $D(X(\calK)/G(\calO))$, and $D(G_\bbR(\mO_\bbR)\backslash\Gr_\bbR)$) be the dg-categories of  $\bC$-constructible complexes 
on $K(\calK)\backslash\Gr$, $X(\calK)/G(\calO)$,  and $G_\bbR(\mO_\bbR)\backslash\Gr_\bbR$ respectively (see Appnedix \ref{sheaves}).

By Proposition \ref{iso of stacks},
there is  an isomorphism of stacks
$K(\calK)\backslash\Gr\is X(\calK)/G(\calO)$
and hence  a canonical equivalence
\beq\label{torsor-equ}
D(K(\calK)\backslash\Gr)\is D(X(\calK)/G(\calO))
\eeq

We will call $D(X(\calK)/G(\calO))$ the relative Satake category 
and $D(G_\bbR(\mO_\bbR)\backslash\Gr_\bbR)$ the real Satake category.


\subsection{Perverse t-structures}

Let $D_{}(G_\bbR\backslash\Gr_\bbR)$ 
be the dg category of  constructible $\bbC$-complexes on $G_\bbR\backslash\Gr_\bbR$
and let 
$D_{\calS_\bbR}(G_\bbR\backslash\Gr_\bbR)$ 
to be the full subcategory
of $D(G_\bbR\backslash\Gr)$ of complexes constructible with 
respect to the $G_\bbR(\mO_\bbR)$-orbits stratification
$\calS_\bbR=\{S_\bbR^\lambda\}_{\lambda\in\Lambda_A^+}$.
We have a natural equivalence 
\beq\label{real model equ} D_{\calS_\bbR}(G_\bbR\backslash\Gr_\bbR)\is 
D_{}(G_\bbR(\mO_\bbR)\backslash\Gr_\bbR)
\eeq
Since $G_\bbR$ is connected, 
it follows from \cite[Lemma 3.9.1]{N2} that 
the dimension  
$\dim_\bbR S_\bbR^\lambda=2\langle\lambda,\rho\rangle$
of the 
real spherical orbits $S_\bbR^\lambda$
are even numbers. Thus there is a classical perverse $t$-structure 
$(^{p_{}}D^{\leq 0}(G_\bbR(\calO_\bbR)\backslash\Gr_\bbR),{^p}D^{\geq 0}(G_\bbR(\calO_\bbR)\backslash\Gr_\bbR))$
on
$D(G_\bbR(\calO_\bbR)\backslash\Gr_\bbR)\is D_{\calS_\bbR}(G_\bbR\backslash\Gr_\bbR)$
given by the middle perversity function 
$p_\bbR:\calS_\bbR\to\bZ$, $p_\bbR(S_\bbR^\lambda)=-\frac{\dim_\bbR\Gr_\bbR^\lambda}{2}=
-\langle\lambda,\rho\rangle$. 
We will write 
\[\on{Perv}(G_\bbR(\calO_\bbR)\backslash\Gr_\bbR)={^{p_{cl}}}D^{\leq 0}(G_\bbR(\calO_\bbR)\backslash\Gr_\bbR)\cap{^{p_{cl}}}D^{\geq 0}(G_\bbR(\calO_\bbR)\backslash\Gr_\bbR)\]
for the heart of the perverse $t$-structure.

We shall introduce a $t$-structure on $D(X(\calK)/G(\calO))$
following \cite[Section 6]{BKV}.
For any finite type stack $\calY$ we denote by 
 $({^{p_{cl}}}D^{\leq0}(\calY),{^{p_{cl}}}D^{\geq0}(\calY))$
 the classical perverse $t$-structure on $D(\calY)$.
According to \cite[Section 6.3.2]{BKV}, for every placid stack $\calY\is Y/H$ there is a unique 
$!$-adapted $t$-structure on $(^pD^{\leq0}(\calY),{^p}D^{\geq0}(\calY))$ 
on
$D(\calY)$ characterized by 
\beq\label{characterization'}
^pD^{\geq0}(\calY)\is\on{colim}_{i\in I} {^{p_{cl}}}D^{\geq0}(\calY_i)[\dim\calY_i]
\eeq
where 
$\calY\is\on{lim}_{i\in I} \calY_j=Y_j/H$

Following \cite[Section 2.4.8]{BKV},
we say that an ind-stack 
$\calY$ is placid stratified if there is a  
stratification $\{\calS^\alpha\}_{\alpha\in I}$ of $\calY$ such that 
(1)
each stratum $\calS^\alpha$ is a  placid stack 
and the inclusion 
$j^\alpha:\calS^\alpha\to\calY$ is a fp-locally closed embedding 
(2) there is a presentation 
$\calY=\on{colim}_{j\in J} \calY^j$
of $\calY$ as a filtered colimit of closed substacks such that 
each $\calY$ is a finite union of the strata $\calS^\alpha$.

Recall the following construction of 
$t$-structures  functor 
for placid stratified stack, admitting gluing of sheaves
in \cite[Proposition 6.4.2]{BKV}.

\begin{prop}\label{t-structures}
Let $(\calY,\{\calS^\alpha\}_{\alpha\in I})$
be a placid stratified stack, admitting gluing of sheaves, and equipped with
a perversity $p:I\to\bZ$.
There is a unique $t$-structure $({^p}D^{\leq0}(\calY),{^p}D^{\geq0}(\calY))$
on $D(\calY)$ satisfying
\beq\label{characterization}
^pD^{\geq0}(\calY)=\{\calF\in D(\calY)| (j^{\alpha})^!\calF\in {^p}D^{\geq0}(\calS^\alpha)[p(\alpha)].\}
\eeq
The heart of the $t$-structure is denoted by
\[\on{Perv}^p(\calY)={^p}D^{\leq0}(\calY)\cap{^p}D^{\geq0}(\calY).\]

\end{prop}

Note that every ind-placid stack $\calY$
 is a placid stratified stack and 
\cite[Lemma 5.5.6]{BKV}  implies that it 
 admits glueing of sheaves.
Thus Proposition \ref{placid of X(K)}, Corollary \ref{iso of stacks} and Lemma \ref{quotient is placid} imply
that $K(\calK)\backslash\Gr\is X(\calK)/G(\calO)$ together with the orbits stratification 
$\{X(\calK)^\lambda/G(\calO)\is K(\calK)\backslash\mO^\lambda_K\}$ is a placid stratified stack, 
 admitting glueing of sheaves.
We are mainly interested in the following perversity function
\beq\label{perversity}
p_X:\Lambda_A^+\to\bZ,\ \ p_X(\lambda)=-\langle\lambda,\rho\rangle
\eeq
and we will call the unique $t$-structure 
$({^p}D^{\leq0}(X(\calK)/G(\calO)),{^p}D^{\geq0}(X(\calK)/G(\calO)))$
on $D(X(\calK)/G(\calO))$
and  $({^p}D^{\leq0}(K(\calK)\backslash\Gr),{^p}D^{\geq0}(K(\calK)\backslash\Gr))$
on $D(K(\calK)\backslash\Gr)$
in Proposition \ref{t-structures},
 the 
\emph{perverse $t$-structure}. 
We will write
\[\on{Perv}(X(\calK)/G(\calO))={^p}D^{\leq0}(X(\calK)/G(\calO))\cap{^p}D^{\geq0}(X(\calK)/G(\calO))\]
\[ \on{Perv}(K(\calK)\backslash\Gr)={^p}D^{\leq0}(K(\calK)\backslash\Gr)\cap{^p}D^{\geq0}(K(\calK)\backslash\Gr)\]
for the  heart of the $t$-structure.

Denote by 
$j^\lambda_X: X(\calK)^\lambda/G(\calO)\to X(\calK)/G(\calO)$ 
the natural inclusion. 
Let $\on{Perv}(X(\calK)^\lambda/G(\calO))$  be the category of perverse sheaves 
on the placid stack $X(\calK)^\lambda/G(\calO)$.
Then for any $\calL\in\on{Perv}(X(\calK)^\lambda/G(\calO))$ we have the 
$\IC$-complex 
\[\IC(\calL):=\on{Im}({^p}H^0(j^\lambda_{X,!}\mL[-\langle\lambda,\rho\rangle])\to {^p}H^0(j^\lambda_{X,*}\mL[-\langle\lambda,\rho\rangle]))\in\on{Perv}(X(\calK)/G(\calO))\]
Note that $\omega_{X(\calK)^\lambda/G(\calO)}\in\on{Perv}(X(\calK)^\lambda/G(\calO))$, we will write 
\[\IC^\lambda_X=\IC(\omega_{X(\calK)^\lambda/G(\calO)})\]

Similarly, denote by 
$j^\lambda_K:K(\calK)\backslash\calO^\lambda_K\to K(\calK)\backslash\Gr$ the natural inclusion.
Then for any $\mL\in\on{Perv}(K(\calK)\backslash\calO^\lambda_K)$, we have 
the $\IC$-complex 
\[\IC(\calL):=\on{Im}({^p}H^0(j^\lambda_{K,!}\mL[-\langle\lambda,\rho\rangle])\to {^p}H^0(j^\lambda_{K,*}\mL[-\langle\lambda,\rho\rangle]))\in\on{Perv}(K(\calK)\backslash\Gr)\]
and we will write 
\[\IC_K^\lambda=\IC(\omega_{K(\calK)\backslash\calO^\lambda_K}).\]

\subsection{Real-symmetric equivalence for affine Grassmannians}

Consider  the real analytic ind-scheme
$\Omega K_c\backslash\Gr$.
Introducing  the stratification $\calS_K=\{\Omega K_c\backslash\mO_K^\lambda\}_{\lambda\in\Lambda_A^+}$ with strata 
the
$\Omega K_c$-quotients of $K(\calK)$-orbits. 
Let $D(LK_c\backslash\Gr)$ 
be the dg category of $K_c$-equivariant  $\bC$-constructible on $\Omega K_c\backslash\Gr$
and 
we set $D_{\calS_K}(LK_c\backslash\Gr)$ to be the full subcategory
of $D(LK_c\backslash\Gr)$
of complexes constructible with 
respect to the stratification $\calS_K$.

Consider the natural  map 
$q:LK_c\backslash\Gr\to K(\calK)\backslash\Gr$.

\begin{lemma}\label{quasi-maps model equ}
The pullback functor 
$q^!:D(K(\calK)\backslash\Gr)\to D(LK_c\backslash\Gr)$
is fully-faithful and 
induces 
an equivalence 
\[D(K(\calK)\backslash\Gr)\is D_{\calS_K}(LK_c\backslash\Gr).\]
\end{lemma}
\begin{proof}
Consider the quotient $K(\calK)/LK_c$
of $K(\calK)$ by the subgroup $LK_c$ and the 
natural embedding 
$\Gr\to\Gr\times K(\calK)/LK_c$
sending $\gamma$ to $(\gamma, eLK_c)$.
It induces 
 an isomorphism of stacks 
$LK_c\backslash\Gr\is K(\calK)\backslash(\Gr\times K(\calK)/LK_c)$
(where $K(\calK)$ acts diagonally on $\Gr\times K(\calK)/LK_c$)
and we have  
\[q:LK_c\backslash\Gr\is K(\calK)\backslash(\Gr\times K(\calK)/LK_c)\stackrel{\pr}\to K(\calK)\backslash\Gr\]
where $\pr$ is induced by the natural projection map $\Gr\times K(\calK)/LK_c\to \Gr$.
Thus it suffices to show 
$\pr^!:D(K(\calK)\backslash\Gr)\to D(K(\calK)\backslash(\Gr\times K(\calK)/LK_c))$
is fully-faithful.
Note that we have a commutative diagram
\[\xymatrix{D(K(\calK)\backslash\Gr)\ar[r]^{\cong\ \ }\ar[d]^{\pr^!}&\on{lim}_{[n]} D(K(\calK)^n\times\Gr)\ar[d]^{\on{lim}_{[n]}(\pr^{[n]})^!}\\
D(K(\calK)\backslash(\Gr\times K(\calK)/LK_c))\ar[r]^{\cong\ \ \ \ \ }&\on{lim}_{[n]}D(K(\calK)^n\times(\Gr\times K(\calK)/LK_c))}\]
where the horizontal equivalences come from the 
 $\check{\on{C}}$ech complexes for the coverings  
$\Gr\to K(\calK)\backslash\Gr$ and $\Gr\times K(\calK)/LK_c \to K(\calK)\backslash(\Gr\times K(\calK)/LK_c)$, and the right vertical arrow are induced by the pull-back functors along the projections
$\pr^{[n]}=\id_{K(\calK)^n}\times\pr:K(\calK)^n\times (\Gr\times K(\calK)/LK_c)\to K(\calK)^n\times \Gr.$
Since the space $K(\calK)/LK_c\is K(\calO)/K_c\is K(\calO)_+\times K/K_c$ 
is contractible (here $K(\calO)_+\subset K(\calO)$ is the first congruence subgroup) the functor $(\pr^{[n]})^!$ is fully-faithful and it follows that 
$\pr^!$ is fully-faithful.

We show that the resulting functor 
$q^!:D(K(\calK)\backslash\Gr)\to D_{\calS_K}(LK_c\backslash\Gr)$
is essentially surjective. 
Since $D(K(\calK)\backslash\Gr)$ admits gluing of sheaves,
 Lemma \ref{fiber sequence} 
implies that 
 the category 
$D(K(\calK)\backslash\Gr)$ is generated by 
$j^{\lambda}_{K,*}(\mL_\lambda)$, $\lambda\in\Lambda_A^+$,
where $\mL_\lambda\in D(K(\calK)\backslash\mO_K^\lambda)$.
Consider the following Cartesian diagram
\[\xymatrix{LK_c\backslash\mO_K^\lambda\ar[r]^{i^\lambda_+}\ar[d]^{q^\lambda}&LK_c\backslash\Gr\ar[d]^q\\
 K(\calK)\backslash\mO^\lambda_K\ar[r]^{j^\lambda_K}& K(\calK)\backslash\Gr}\]
The desired claim follows from the facts that 
(1)
$D_{\calS}(LK_c\backslash\Gr)$ is generated by 
$(i^\lambda_+)_*(q_\lambda^!)(\mL_\lambda)$ under colimit 
and (2) the base change isomorphism for fp-locally closed embedding 
$q^!(j^{\lambda}_{K,*}(\mL_\lambda))\is (i^\lambda_+)_*(q^\lambda)^!(\mL_\lambda) $
(see, e.g., \cite[Lemma 5.4.5]{BKV}).

\end{proof}

\begin{thm}\label{real-symmetric}
There are natural  equivalences
\[D(X(\calK)/G(\calO))\is 
D(K(\calK)\backslash\Gr)\is D(G_\bbR(\calO_\bbR)\backslash\Gr_{\bbR})\]
which are $t$-exact with respect to the perverse $t$-structures.
\end{thm}
\begin{proof}
Theorem \ref{Quillen} (the case when $m=1$)  implies that 
there is a $K_c$-equivariant stratified homeomorphism 
$\Omega K_c\backslash\Gr\is\Gr_\bbR$
which induces 
 a natural  equivalence 
\beq\label{trivialization equ}
D_{\calS_K}(LK_c\backslash\Gr)\is D_{\calS_\bbR}(G_\bbR\backslash\Gr_\bbR).
\eeq
Now combining~\eqref{torsor-equ} and Lemma \ref{quasi-maps model equ}
we obtain the desired equivalences
\[D(X(\calK)/G(\calO))\stackrel{~\eqref{torsor-equ}}\is 
D(K(\calK)\backslash\Gr)\stackrel{\text{Lem}~\eqref{quasi-maps model equ}}\is D_{\calS_K}(LK_c\backslash\Gr)\stackrel{~\eqref{trivialization equ}}\is D_{\calS_\bbR}(G_\bbR\backslash\Gr_\bbR)\is\]
\[\stackrel{\eqref{real model equ}}\is 
D(G_\bbR(\calO_\bbR)\backslash\Gr_{\bbR}).\]
To check that the equivalences are $t$-exact it suffices to check that 
it restricts to an equivalence
$^pD^{\geq0}(K(\calK)\backslash\Gr)\is {^p}D^{\geq0}(G_\bbR(\calO_\bbR)\backslash\Gr_\bbR)$.
Note that the commutative diagram
\[\xymatrix{K_c\backslash S^\lambda_\bbR\is LK_c\backslash\calO^\lambda_K\ar[r]^{i^\lambda_+}\ar[d]^{q^\lambda}&K_c\backslash\Gr_\bbR\is  LK_c\backslash\Gr\ar[d]^q\\
K(\calK)\backslash\mO^\lambda_K\ar[r]^{j_K^\lambda}&  K(\calK)\backslash\Gr}\]
implies that there is a  commutative square of functors
\[\xymatrix{D(K(\calK)\backslash\Gr)\ar[r]_{\simeq}^{q^!}\ar[d]^{(j_K^\lambda)^!}& D(G_\bbR(\calO_\bbR)\backslash\Gr_\bbR)\ar[d]^{(i_+^\lambda)^!}\\
D(K(\calK)\backslash\mO_K^\lambda)\ar[r]^{(q^\lambda)^!}&D(G_\bbR(\calO_\bbR)\backslash S^\lambda_\bbR)}\]
In view of  the characterization of 
\[^pD^{\geq0}(K(\calK)\backslash\Gr)=\{\mF\in D^{}(K(\calK)\backslash\Gr)|(j_K^\lambda)^!\mF\in {^p}D^{\geq0}(K(\calK)\backslash\mO_K^\lambda)[p_X(\lambda)]\}\]
in~\eqref{characterization}, we need to check $(q^\lambda)^!$ induces an equivalence 
\[(q^\lambda)^!:{^p}D^{\geq0}(K(\calK)\backslash\mO_K^\lambda)[p_X(\lambda)]\is {^p}D^{\geq0}(G_\bbR(\calO_\bbR)\backslash S_\bbR^\lambda).\]
Let $K(\calK)\backslash\mO_K^\lambda\is\on{lim}_{i\in I}\calY_i$ be  a placid presentation
with natural evaluation map $ev_i:K(\calK)\backslash\mO_K^\lambda\to\calY_i$.
According to~\eqref{characterization'}, we have
\[{^p}D^{\geq0}(K(\calK)\backslash\mO_K^\lambda)\is\on{colim}_{i\in I} {^p}D^{\leq0}(\calY_i)[\dim\calY_i]\]
Since the fiber of the composition $ev_i\circ q^\lambda: K_c\backslash S_\bbR^\lambda\to K(\calK)\backslash\mO_K^\lambda\to \calY_i$ 
are contractible 
 and $\dim_\bbR S^\lambda_\bbR/2=\langle\lambda,\rho\rangle=-p_X(\lambda)$, we have 
\[(ev_i\circ  q^\lambda)^!:{^p}D^{\leq0}(\calY_i)[\dim\calY_i]\to {^p}D^{\leq0}(G_\bbR(\calO_\bbR)\backslash S_\bbR^\lambda)[\dim_\bbR S^\lambda_\bbR/2]\subset D^{\leq0}(K_c\backslash S_\bbR^\lambda)[\dim_\bbR S^\lambda_\bbR].\]
and the resulting map 
\[(q^\lambda)^!\is\on{colim}_{i\in I}(ev_i\circ  q^\lambda)^!:{^p}D^{\geq0}(K(\calK)\backslash\mO_K^\lambda)[p_X(\lambda)]\is \on{colim}_{i\in I} {^p}D^{\geq0}(\calY_i)[\dim\calY_i-\dim_\bbR S^\lambda_\bbR/2]\to\]
\[\to D^{\leq0}(K_c\backslash S_\bbR^\lambda)[\dim_\bbR S^\lambda_\bbR/2]\]
factors through an equivalence 
\[(q^\lambda)^!:{^p}D^{\geq0}(K(\calK)\backslash\mO_K^\lambda)[p_X(\lambda)]\is {^p}D^{\geq0}(G_\bbR(\calO_\bbR)\backslash S^\lambda_\bbR)\subset D^{\geq0}(K_c\backslash S_\bbR^\lambda)[\dim_\bbR S^\lambda_\bbR/2]\]
The desired claim follows. 
\end{proof}

\section{Affine Matsuki correspondence for sheaves}\label{Affine Matsuki}
In this section we prove the 
affine Matsuki correspondence for sheaves.

\subsection{The functor $\Upsilon$}
Consider the dg category  $D(LG_\bbR\backslash\Gr)$ 
of  $\bbC$-constructible complexes on the semi-analytic stacks $LG_\bbR\backslash\Gr$.
Consider the following correspondence
\[\xymatrix{K(\calK)\backslash\Gr&LK_c\backslash\Gr\ar[r]^u\ar[l]_q&LG_\bbR\backslash\Gr}\]
Define 
\beq\Upsilon=u_!q^!:D(K(\calK)\backslash\Gr)\stackrel{}\to D_{}(LK_c\backslash\Gr)\stackrel{}\ra D(LG_\bbR\backslash\Gr)
\eeq

\begin{thm}[Affine Matsuki correspondence for sheaves]\label{AM}
The functor $\Upsilon$ defines an equivalence of categories
\[\Upsilon:D(K(\calK)\backslash\Gr)\stackrel{\sim}\lra D(LG_\bbR\backslash\Gr).\]
\end{thm}

The rest of the section is devoted to the proof of Theorem \ref{AM}.

\subsection{Bijection between local systems}
Write $[\mO_K^\lambda]=LK_c\backslash\mO_K^\lambda$, $[\mO_\bbR^\lambda]=LK_c\backslash\mO_\bbR^\lambda$, 
$[\mO_c^\lambda]=LK_c\backslash\mO_c^\lambda$, and 
$[\mE^\lambda]=LG_\bbR\backslash\mO_\bbR^\lambda\in LG_\bbR\backslash\Gr$.
Recall the Matsuki flow $\phi_t:\Gr\to
\Gr$ in Theorem \ref{flow}. 
As $\phi_t$ is $LK_c$-equivariant, it descends to a flow
$\tilde\phi_{t}:LK_c\backslash\Gr\to LK_c\backslash\Gr$ and we define
\[\phi_\pm:LK_c\backslash\Gr\ra\bigsqcup_{\lambda\in\Lambda_A^+}\ [\mO_c^\lambda]\subset LK_c\backslash\Gr,\ \  \gamma\ra\underset{t\to\pm\infty}\lim\tilde\phi_t(\gamma).\]
Consider the following Cartesian diagrams:
\[\xymatrix{[\mO_K^\lambda]\ar[r]^{i_+^\lambda\ \ \ }\ar[d]^{\phi_+^\lambda}&L K_c\backslash\Gr\ar[d]^{\phi_+}
\\[\mO_c^\lambda]\ar[r]^{j_+^\lambda\ \ \ }&\bigsqcup_{\lambda\in\Lambda_A^+}[\mO_c^\lambda]}\ \ \ \ \ 
\xymatrix{[\mO_\bbR^\lambda]\ar[r]^{i_-^\lambda\ \ \ }\ar[d]^{\phi_-^\lambda}&L K_c\backslash\Gr\ar[d]^{\phi_-}
\\[\mO_c^\lambda]\ar[r]^{j_+^\lambda\ \ \ }&\bigsqcup_{\lambda\in\Lambda_A^+}[\mO_c^\lambda]}\ \ \ \ \xymatrix{[\mO_\bbR^\lambda]\ar[r]^{i_-^\lambda\ \ \ }\ar[d]^{u^\lambda}&LK_c\backslash\Gr\ar[d]^{u}
\\[\mE^\lambda]\ar[r]^{j_-^\lambda\ \ \ }&LG_\bbR\backslash\Gr}
\]
Here  
$i^\lambda_\pm$ and $j^\lambda_\pm$ are the natural 
embeddings and $\phi^\lambda_\pm$ (resp. $u^\lambda$) is the restriction of 
$\phi_\pm$ (resp. $u$) along $j^\lambda_+$ (resp. $j^\lambda_-$). 

\begin{lemma}\label{bijection}
We have the following:
\begin{enumerate}
\item
There is a bijection between isomorphism classes of 
local systems $\tau^+$ on $[\mO_K^\lambda]$,
local systems $\tau^-$ on $[\mO_\bbR^\lambda]$, 
local systems $\tau$ on $[\mO_c^\lambda]$, and local systems $\tau_\bbR$ on $[\mE^\lambda]=[LG_\bbR\backslash\mO_\bbR^\lambda]$, characterizing by the property that 
$\tau^\pm\is(\phi_\pm^\lambda)^*\tau$ and $
\tau^-\is(u^\lambda)^*\tau_\bbR$. 
\item
The map $u^\lambda$ factors as 
\beq\label{p^lambda}
u^\lambda:[\mO_\bbR^\lambda]
\stackrel{\phi_-^\lambda}\ra[\mO_c^\lambda]\stackrel{p^\lambda}\ra[\mE^\lambda]\eeq
where $p^\lambda$ is smooth of relative dimension $\on{dim}[\mE^\lambda]-\on{dim}[\mO_c^\lambda]$. Moreover, we have 
$(p^\lambda)^*\tau_\bbR\is\tau$. 
\end{enumerate}
\end{lemma}
\begin{proof}

Since the fibers of 
$\phi_\pm$ 
are contractible, pull-back along $\phi^\lambda_+$ (resp. $\phi^\lambda_-$) defines 
an equivalence between
$LK_c$-equivariant local systems on 
$\mO_c^\lambda$ 
and $LK_c$-equivariant local systems 
on $\mO_K^\lambda$ (resp. $\mO_\bbR^\lambda$).
We show that the fiber of $u^\lambda$ is contractible, hence pull back along 
$u^\lambda$ defines an equivalence between local systems on 
$[\mE^\lambda]$ and $LK_c$-equivariant local systems on $\mO_\bbR^\lambda$.
Pick $y\in\mO_c^\lambda$ and let
$LK_c(y)$, $LG_\bbR(y)$ be the stabilizers of 
$y$ in $LK_c$ and $LG_\bbR$ respectively.
The group $LK_c(y)$ acts on the fiber 
$l_y:=(\phi_-^\lambda)^{-1}(y)$ and we have 
$\mO_\bbR^\lambda\is LK_c\times^{LK_c(y)}l_y$. 
Moreover, under the isomorphism 
$[\mO_\bbR^\lambda]\is LK_c\backslash\mO_\bbR^\lambda\is
LK_c(y)\backslash l_y$, $[\mO_c^\lambda]\is LK_c(y)\backslash y$, and
$[\mE^\lambda]\is LG_\bbR(y)\backslash y$, 
the map $u^\lambda$ 
takes the form
\[
u^\lambda:[\mO_\bbR^\lambda]\is LK_c(y)\backslash l_y\stackrel{\phi_-^\lambda}\ra 
[\mO_c^\lambda]\is LK_c(y)\backslash y\stackrel{p^\lambda}\ra 
[\mE^\lambda]\is LG_\bbR(y)\backslash y,\]
where the first map is induced by the projection 
$l_y\ra y$ and 
the second map is induced by the 
inclusion $LK_c(y)\ra LG_\bbR(y)$.
We claim that the quotient $LK_c(y)\backslash LG_\bbR(y)$ is contractible, 
hence 
$u^\lambda$
has contractible fibers and $p^\lambda$ is smooth of relative dimension $\on{dim}[\mE^\lambda]-\on{dim}[\mO_c^\lambda]$. Part (1) and (2) follows.

Proof of the claim. Pick $y'\in C_\bbR^\lambda\subset\Gr_\bbR$ and let
$K_c(y')$ and $G_\bbR(\bbR[t^{-1}])(y')$ be the stabilizers of 
$y'$ in $K_c$ and $G_\bbR(\bbR[t^{-1}])$ respectively.
The composition of the complex and real uniformizations of $\Bun_{G_\bbR}(\bP^1(\bbR))$
 \[LG_\bbR\backslash\Gr\stackrel{\on{Prop.} \ref{uniformizations at complex x}}\is \Bun_{G_\bbR}(\bP^1(\bbR))
 \stackrel{\on{Prop.} \ref{uniformizations at real x}}\is G_\bbR(\bbR[t^{-1}])\backslash\Gr_\bbR\] identifies 
\[LG_\bbR(y)\backslash y\is [\mE^\lambda]\is G_\bbR(\bbR[t^{-1}])(y')\backslash y'.\] 
Hence we obtain a natural isomorphism 
\[LG_\bbR(y)\is\Aut([\mE^\lambda])\is G_\bbR(\bbR[t^{-1}])(y')\]
sending $LK_c(y)=K_c(y)\subset LG_\bbR(y)$ to $K_c(y')\subset G_\bbR(\bbR[t^{-1}])(y')$.
Thus we reduce to show that the quotient $K_c(y')\backslash G_\bbR(\bbR[t^{-1}])(y')$ is contractible.
This follows from the fact that evaluation map 
$K_c(y')\backslash G_\bbR(\bbR[t^{-1}])(y')\to K_c(y')\backslash G_\bbR(y'),\ \ \gamma(t^{-1})\to\gamma(0)$
has contractible fibers and the quotient $K_c(y')\backslash G_\bbR(y')$ is contractible as 
$K_c(y')$ is a maximal compact subgroup of the Levi subgroup of $G_\bbR(y')$.

\end{proof}

Recall the flow \[\psi_z^3:\QM^{(\sigma_2)}(\bbP^1,X,\infty)_\bbR\ra\QM^{(\sigma_2)}(\bbP^1,X,\infty)_\bbR\] in \S\ref{flows}.
For $\lambda\in\Lambda_A^+$, we have the 
critical manifold $C_\bbR^\lambda$, the stable manifold 
$S_\bbR^\lambda$, and the unstable manifold $\tilde T_\bbR^\lambda$.
We write 
\beqn 
s_\lambda^+:S_\bbR^\lambda\ra\QM^{(\sigma_2)}(\bbP^1,G,K,\infty)_\bbR,\ \ \ \ 
\tilde t_\lambda:\tilde T_\bbR^\lambda\ra\QM^{(\sigma_2)}(\bbP^1,G,K,\infty)_\bbR
\eeqn
for the inclusion maps and we write 
\beqn
c_\lambda^+:S_\bbR^\lambda\ra C_\bbR^\lambda,\ \ \ \ \tilde d_\lambda:\tilde T_\bbR^\lambda\ra C_\bbR^\lambda
\eeqn
for the contraction maps. 
Note that all the  
maps above are $K_c$-equivalent with respect to natural $K_c$-actions.
Recall that, by Lemma \ref{properties of flows}, we have isomorphisms 
$\tilde T_\bbR^\lambda|_i\is\Omega K_c\backslash\mO_\bbR^\lambda$, $\tilde T_\bbR^\lambda|_0\is T_\bbR^\lambda$,
for $\lambda\in\Lambda_A^+$ and we write 
\beqn
s_\lambda^-:T_\bbR^\lambda\ra\QM^{(\sigma_2)}(\bbP^1,G,K,\infty)_\bbR,\ \ \ 
t_\lambda:\Omega K_c\backslash\mO_\bbR^\lambda\ra\QM^{(\sigma_2)}(\bbP^1,G,K,\infty)_\bbR
\eeqn
for the restriction of $\tilde t_\lambda$ and 
\beqn
c_\lambda^-:T_\bbR^\lambda\ra C_\bbR^\lambda,\ \ \ 
d_\lambda:\Omega K_c\backslash\mO_\bbR^\lambda\ra C_\bbR^\lambda
\eeqn
for the restriction of the contractions $\tilde d_\lambda$.

We write 
$
k_\lambda:\Omega K_c\backslash\mO_c^\lambda\ra C_\bbR^\lambda$
for the restriction of $d_\lambda$ and 
$
p_\lambda: T_\bbR^\lambda\ra G_\bbR(\bbR[t^{-1}])\backslash T_\bbR^\lambda$ 
for the natural quotient map.

\quash{
We have the following diagrams
\beq
\xymatrix{C_\bbR^\lambda&T_\bbR^\lambda\ar[d]^{c_\lambda}\ar[l]
\\ S_\bbR^\lambda\ar[r]^{}\ar[u]&\Gr_\bbR}\ \ \ \ \ 
\xymatrix{\Omega K_c\backslash\mO_c^\lambda\ar[r]^{}\ar[d]^{k_\lambda}&\Omega K_c\backslash\mO_\bbR^\lambda\ar[d]^{u}
\\ C_\bbR^\lambda\ar[r]^{j_-^\lambda\ \ \ }&\Gr_\bbR}\ \ \ \ 
\xymatrix{T_\bbR^\lambda\ar[r]^{i_-^\lambda\ \ \ }\ar[d]^{\phi_-^\lambda}&
\Gr_\bbR\ar[d]^{\phi_-}
\\ G_\bbR(\bbR[t^{-1}])\backslash T_\bbR^\lambda\ar[r]^{}&G_\bbR(\bbR[t^{-1}])\backslash\Gr_\bbR}
\eeq}

\begin{lemma}\label{bijection 2}
The map $k_\lambda:\Omega K_c\backslash\mO_c^\lambda\ra C_\bbR^\lambda$ is a $K_c$-equivariant isomorphism. 
There is a bijection between isomorphism classes of 
$K_c$-equivariant local systems $\omega^+$ on $S_\bbR^\lambda$,
$K_c$-equivariant local systems $\omega^-$ on $T_\bbR^\lambda$, 
$K_c$-equivariant local systems $\omega$ on $C_\bbR^\lambda$, 
$K_c$-equivariant local systems $\tau$ on $\Omega K_c\backslash\mO_c^\lambda$, 
and local system $\omega_\bbR$ on $G_\bbR(\bbR[t^{-1}])\backslash T_\bbR^\lambda$,
characterizing by the property that 
$\omega^\pm\is(c_\lambda^\pm)^*\omega$, $
\tau\is(k_\lambda)^*\omega$
, and $(p_\lambda)^*\omega_\bbR\is (c_\lambda^-)^*\omega$
\end{lemma}
\begin{proof}
The first claim follows from the fact that 
$\Omega K_c\backslash\mO_c^\lambda\is C_\bbR^\lambda\is
K_c(\lambda)\backslash K_c$, where $K_c(\lambda)$
is the stabilizer of $\lambda$ in $K_c$, and the $K_c$-equivariant property of $k_\lambda$. 
The second claim follows from the facts that the 
contraction maps $c_\lambda^\pm$ are
$K_c$-equivariant and the 
fibers of $c_\lambda^\pm$ and the quotient
$K_c\backslash G_\bbR(\bbR[t^{-1}])$
are contractible.

\end{proof}

\subsection{Standard and co-standard sheaves}
For each $\lambda\in\Lambda_A^+$ and a local system 
$\tau$ on $[\mO_c^\lambda]$
one has the standard sheaves
\beq\label{standard}
\calS_*^+(\lambda,\tau):=(i_{+}^\lambda)_*(\tau^+) \text{\ \ and\ \ } 
\calS_*^-(\lambda,\tau):=(j_{-}^\lambda)_*(\tau_\bbR)
\eeq
and co-standard sheaves
\beq\calS_!^+(\lambda,\tau):=(i_{+}^\lambda)_!(\tau^+)
\text{\ \ and\ \ } 
\calS_!^-(\lambda,\tau):=(j_-^\lambda)_!(\tau_\bbR).
\eeq
Here $\tau^+$ and $\tau_\bbR$ are local system 
on $[\mO_K^\lambda]$ and $[\mE^\lambda]$
corresponding to $\tau$
as in Lemma \ref{bijection}.
Let $d_\lambda:=\on{dim}\Bun_{G}(\mathbb P^1)_\bbR-\on{dim}[\mO_K^\lambda]$.

Write
\beq\label{iota}
\iota_+^\lambda:[\mO_c^\mu]\to[\mO_K^\mu],\ \ \iota_-^\lambda:[\mO_c^\mu]\to[\mO_\bbR^\mu]
\eeq
for the natural embeddings.
We recall the following fact, see \cite[Lemma 5.4]{MUV}.
\begin{lemma}\label{Braden}
\begin{enumerate}
\item
Consider $[\mO_c^\mu]\stackrel{\iota_+^\mu}\to[\mO_K^\mu]\stackrel{\phi_+^\mu}\to[\mO_c^\mu]$.
Let $\mF\in D_c([\mO_K^\mu])$. If $\mF$ is 
smooth (= locally constant) on the trajectories of the flow $\tilde\phi_t$, then we have canonical isomorphisms 
$(\iota_+^\mu)^!\mF\is(\phi_+^\mu)_!\mF$  and 
$(\iota_+^\mu)^*\mF\is(\phi_+^\mu)_*\mF$.
\item
Consider $[\mO_c^\mu]\stackrel{\iota_-^\mu}\to[\mO_\bbR^\mu]\stackrel{\phi_-^\mu}\to[\mO_c^\mu]$ where $\iota_-^\mu$ is the natural embedding. 
Let $\mF\in D_c([\mO_\bbR^\mu])$. If $\mF$ is 
smooth (= locally constant) on the trajectories of the flow $\tilde\phi_t$ and is supported on a finite dimensional substack $\sY\subset [\mO_\bbR^\mu]$, then we have canonical isomorphisms 
$(\iota_-^\mu)^!\mF\is(\phi_-^\mu)_!\mF$  and 
$(\iota_-^\mu)^*\mF\is(\phi_-^\mu)_*\mF$.
\end{enumerate}

\end{lemma}

We shall show that the functor $\Upsilon$ sends 
standard sheaves to co-standard sheaves.
Introduce the following local system on $[\mO_c^\lambda]$
\beq\label{L_lambda}
\mL_\lambda:=(\iota_-^\lambda)^*\mL_\lambda'\otimes\mL_\lambda''\otimes\on{or}_{p^\lambda}^\vee
\eeq
where 
\beq
(\mL'_\mu)^\vee:=(i_-^\mu)^!(\bC)[\on{codim}[\mO_\bbR^\mu]] \ \text{\ and\ \ }  
\mL_\lambda'':=(\iota_+^\lambda)^!\bC[\on{codim}_{[\mO_K^\lambda]}[\mO_c^\lambda]]
\eeq
are local systems 
on $[\mO_\bbR^\lambda]$ and $[\mO_c^\lambda]$ respectively and 
$\on{or}_{p^\lambda}:=(p^\lambda)^!\bC[-\on{dim}[\mE^\lambda]+\on{dim}[\mO_c^\lambda]]$ is the orientation sheaf for the smooth map $p^\lambda:[\mO_c^\lambda]\to[\mE^\lambda]$ in \eqref{p^lambda}.

\begin{lemma}\label{surjective}
For any local system 
$\tau$ on $[\mO_c^\lambda]$ we have 
\[\Upsilon(S^+_*(\lambda,\tau))\is S^-_!(\lambda,\tau\otimes\mL_\lambda)[d_\lambda].\]

\end{lemma}
\begin{proof}
Let $\lambda,\mu\in\Lambda_A^+$.
Consider the following diagram 
\beq\label{key diagram}
\xymatrix{[\mO_K^\lambda\cap\mO_\bbR^\mu]\ar[r]^s\ar[d]^\iota
&[\mO_\bbR^\mu]\ar[r]^{u^\mu\ \ \ \ \ \ \ }\ar[d]^{i_-^\mu}&[\mE^\mu]=LG_\bbR\backslash\mO_\bbR^\mu\ar[d]^{j_-^\mu}\\
[\mO_K^\lambda]\ar[r]^{i_+^\lambda}&LK_c\backslash\Gr\ar[r]^{u\ }&LG_\bbR\backslash\Gr}.
\eeq
Let 
$\mG=(j_-^\mu)^*\Upsilon(S_*(\lambda,\tau))\is (j_-^\mu)^*u_!(i_+^\lambda)_*(\tau^+)\is(u^\mu)_!(i_-^\mu)^*(i_+^\lambda)_*(\tau^+)$.
It suffices to show that $\mG\is 0$ if $\lambda\neq\mu$ and $\mG\is(\tau\otimes\mL_\lambda)_\bbR$ if $\lambda=\mu$.

By Corollary \ref{transversal}, the orbits $\mO_\bbR^\mu$ and $\mO_K^\lambda$ are trasversal to each 
other, hence we have 
\beq\label{upper shriek}
(i_-^\mu)^*(i_+^\lambda)_*(\tau^+)\is(i_-^\mu)^!(i_+^\lambda)_*(\tau^+)\otimes\mL'_\mu[\on{codim}[\mO_\bbR^\mu]].
\eeq
where 
\beq\label{L'}
(\mL'_\mu)^\vee=(i_-^\mu)^!(\bC)[\on{codim}[\mO_\bbR^\mu]]
\eeq
is a local system on $[\mO_\bbR^\mu]$.
Thus 
\[
\mG\is(u^\mu)_!(i_-^\mu)^*(i_+^\lambda)_*(\tau^+)\stackrel{\eqref{upper shriek}}\is(u^\mu)_!((i_-^\mu)^!(i_+^\lambda)_*(\tau^+)\otimes\mL_\mu')[\on{codim}[\mO_\bbR^\mu]]\is
\]
\[
(u^\mu)_!(s_*\iota^!(\tau^+)\otimes\mL_\mu')[\on{codim}[\mO_\bbR^\mu]].
\]

According to Lemma \ref{bijection} the map $u^\mu$ factors as 
\[u^\mu:[\mO_\bbR^\mu]\stackrel{\phi_-^\mu}\ra[\mO_c^\mu]\stackrel{p^\mu}\ra[\mE^\mu]\]
where $p^\mu$ is smooth of relative dimension $\on{dim}[\mE^\mu]-\on{dim}[\mO_c^\mu]$. Since $s_*\iota^!(\tau^+)[\on{codim}[\mO_\bbR^\mu]]\in D_c([\mO_\bbR^\mu])$
is smooth on the trajectories of the flow $\tilde\phi_t$, by Lemma \ref{Braden}, we have 
\beq\label{mG}
\mG\is u^\mu_!(s_*\iota^!(\tau^+)\otimes\mL_\mu')[\on{codim}[\mO_\bbR^\mu]]
\is p^\mu_!(\phi_-^\mu)_!(s_*\iota^!(\tau^+)\otimes\mL_\mu')[\on{codim}[\mO_\bbR^\mu]]
\stackrel{\on{Lem}\ref{Braden}}\is
\eeq
\[\is p^\mu_!(\iota_-^\mu)^!(s_*\iota^!(\tau^+)\otimes\mL_\mu')[\on{codim}[\mO_\bbR^\mu]].\]
Here $\iota_-^\mu:[\mO_c^\mu]\ra[\mO_\bbR^\mu]$ is the embedding.

If $\lambda\neq\mu$ then $[\mO_\bbR^\mu]\cap[\mO_c^\lambda]$ is empty, thus we have \[(\iota_-^\mu)^!s_*\iota^!(\tau^+)[\on{codim}[\mO_\bbR^\mu]]=0\]
and \eqref{mG} implies $\mG=0$. 

If $\lambda=\mu$, then 
$[\mO_\bbR^\lambda]\cap[\mO_K^\lambda]=[\mO_c^\lambda]$,
$s=\iota_-^\lambda$, $\iota=\iota_+^\lambda$ are closed embeddings 
and by Lemma \ref{bijection} we have 
\[
(u^\lambda)_!(s)_*(\tau)\is (p^\lambda)_!(\tau)\is\tau_\bbR\otimes (p^\lambda)_!(\bC)\is
\tau_\bbR\otimes(\on{or}_{p^\lambda}^{\vee})_\bbR[\on{dim}[\mE^\mu]-\on{dim}[\mO_c^\mu]],
\]
\[\iota^!(\tau^+)\is \iota^*(\tau^+)\otimes \iota^!\bC\is\tau\otimes (\iota_+^\lambda)^!\bC\is\tau\otimes\mL''_\lambda[-\on{codim}_{[\mO_K^\lambda]}[\mO_c^\lambda]]\]
where $\on{or}_{p^\lambda}$ is the relative orientation sheaf on $[\mO_c^\lambda]$ associated to $p^\lambda:[\mO_c^\lambda]\to[\mE^\lambda]$
and 
\beq\label{L''}
\mL_\lambda'':=(\iota_+^\lambda)^!\bC[\on{codim}_{[\mO_K^\lambda]}[\mO_c^\lambda]]
\eeq
is a local system on $[\mO_c^\lambda]$.
Now an elementary calculation shows that 
\[\mG\stackrel{\eqref{mG}}\is (u^\lambda)_!(s_*\iota^!(\tau^+)\otimes\mL_\lambda')[\on{codim}[\mO_\bbR^\lambda]]\is
(u^\lambda)_!(s_*(\tau\otimes\mL_\lambda'')\otimes\mL_\lambda')[\on{codim}[\mO_\bbR^\lambda]-\on{codim}_{[\mO_K^\lambda]}[\mO_c^\lambda]]\is\]
\[\is(u^\lambda)_!(\iota_-^\lambda)_*(\tau\otimes\mL_\lambda''\otimes (\iota_-^\lambda)^*\mL_\lambda')[\on{codim}[\mO_\bbR^\lambda]-\on{codim}_{[\mO_K^\lambda]}[\mO_c^\lambda]]\is
(\tau\otimes\mL_\lambda)_\bbR[d_\lambda]
,\]
where 
\beq\label{L_lambda}
\mL_\lambda:=(\iota_-^\lambda)^*\mL_\lambda'\otimes\mL_\lambda''\otimes\on{or}_{p^\lambda}^\vee
\eeq
is a local sytem on $[\mO_c^\lambda]$ and
\[d_\lambda=\on{codim}[\mO_\bbR^\lambda]-
\on{codim}_{[\mO_K^\lambda]}[\mO_c^\lambda]+\on{dim}[\mE^\lambda]-\on{dim}[\mO_c^\lambda]=\on{dim}\Bun_{G_\bbR}(\mathbb P^1(\bbR))-\on{dim}[\mO_K^\lambda].\]
The lemma follows.

\end{proof}

\subsection{Fully-faithfulness}

We shall show that $\Upsilon$ is fully-faithful.
Consider a diagram of closed substacks of $LK_c\backslash\Gr$
\quash{
\[\xymatrix{U_0\ar[r]\ar[drrr]&U_1\ar[r]\ar[drr]&\cdot\cdot\cdot&U_k\ar[d]\ar[r]&\cdot\cdot\cdot\\
&&&LK_c\backslash\Gr}\]}
\[
U_0\stackrel{j_0}\lra U_1\stackrel{j_1}\ra U_2\lra\cdot\cdot\cdot\ra U_k\lra\cdot\cdot\cdot\]
such that 
\begin{enumerate}
\item $\bigcup_{i} U_i=LK_c\backslash\Gr$, 
\item Each $U_i$ is a finite union of $[\mO_K^\lambda]$, 
\item Each $j_k$ is closed embedding. 
\end{enumerate}

Let $f_i:U_i\ra LK_c\backslash\Gr$ be the natural embedding and 
we define \[s_i=u\circ f_i:U_i\ra LG_\bbR\backslash\Gr.\]
Note that each $s_i$ is of finite type.

\begin{lemma}\label{full}
For any $\mF,\mF'\in D(K(\calK)\backslash\Gr)$
we have \[\on{Hom}_{D(K(\calK)\backslash\Gr)}(\mF,\mF')
\is\on{Hom}_{D(LG_\bbR\backslash\Gr)}(\Upsilon(\mF),\Upsilon(\mF')).\]

\end{lemma}
\begin{proof}
Since the functor $u_!:D(LK_c\backslash\Gr)
\stackrel{}\to
D(LG_\bbR\backslash\Gr)$ admits a continuous right adjoint $u^!$, it 
sends compact objects to compact objects. It follows that 
the functor
\[\Upsilon=: D(K(\calK)\backslash\Gr)\stackrel{q^!}\is D_{\calS}(LK_c\backslash\Gr)
\subset D(LK_c\backslash\Gr)
\stackrel{u_!}\to
D(LG_\bbR\backslash\Gr)\]
sends compact objects to compact objects.
Thus 
 we can assume 
both $\mF,\mF'$ are compact.
Then we can 
choose $k$ such that 
$\mF=(j_k)_*\mF_k$ and $(j_k)_*\mF'_k$ for 
$\mF_k,\mF'_k\in D(K(\calK)\backslash U_k)$.
We have \[\on{Hom}_{D(K(\calK)\backslash\Gr)}(\mF,\mF')\is
\on{Hom}_{D(K(\calK)\backslash U_k)}(\mF_k,\mF'_k)\]
and \[\on{Hom}_{D(LG_\bbR\backslash\Gr)}(\Upsilon(\mF),\Upsilon(\mF'))\is
\on{Hom}_{D(K(\calK)\backslash U_k)}((s_k)_!\mF_k,(s_k)_!\mF_k')\is\]
\[\is
\on{Hom}_{D(K(\calK)\backslash U_k)}(\mF_k,(s_k)^!(s_k)_!\mF_k').\]
We have to show that the map 
\beq\label{iso}
\on{Hom}_{D(K(\calK)\backslash U_k)}(\mF_k,\mF'_k)\ra
\on{Hom}_{D(K(\calK)\backslash U_k)}(\mF_k,(s_k)^!(s_k)_!\mF_k')
\eeq
is an isomorphism.
Since $D_c(K(\calK)\backslash U_k)$ is generated by 
$w_!(\tau_\lambda^+)$ (resp. $w_*(\tau_\lambda^+)$) for 
$[\mO_K^\lambda]\subset U_k$ (here $w_\lambda:[\mO_K^\lambda]\ra U_k$ is 
natural inclusion), it suffices to verify 
(\ref{iso}) for 
\[\mF_k=(w_\lambda)_!(\tau_\lambda^+)\text{\ \ and \ \ }\mF'_k
\is (w_\mu)_*(\tau_\mu^+).\] 
Note that in this case the left hand side of \eqref{iso} becomes
\beq\label{iso Hom trivial, K}
\on{Hom}_{D(K(\calK)\backslash U_k)}(\mF_k,\mF'_k)=0\text{\ \ if \ \ }\lambda\neq\mu
\eeq
\beq\label{iso Hom, K}
\on{Hom}_{D(K(\calK)\backslash U_k)}(\mF_k,\mF'_k)\is
\on{Hom}_{D([\mO_c^\lambda])}(\tau_\lambda,\tau_\lambda)
\text{\ \ if \ \ }\lambda=\mu.
\eeq
By Lemma \ref{surjective} we have 
\[(s_k)_!((w_\mu)_*(\tau_\mu^+))
\is u_!(j_k)_*(w_\mu)_*(\tau_\mu^+))\is 
\Upsilon(S_*^+(\mu,\tau_\mu))\is
(j_-^\mu)_!(\tilde\tau_{\mu,\bbR})[d_\mu],\] 
where $\tilde\tau_{\mu,\bbR}=\tau_{\mu,\bbR}\otimes\mL_{\mu,\bbR}$.
Therefore the right hand side of \eqref{iso} becomes
\beq\label{iso on hom,1}
\on{Hom}_{D(K(\calK)\backslash U_k)}(\mF_k,
(s_k)^!(s_k)_!\mF_k')\is
\eeq
\beqn
\is\on{Hom}_{D(K(\calK)\backslash U_k)}((w_\lambda)_!(\tau_\lambda^+),
(s_k)^!(s_k)_!((w_\mu)_*(\tau_\mu^+))\is
\eeqn
\beqn
\is
\on{Hom}_{D([\mO_K^\lambda]}(\tau^+_\lambda,
w_\lambda^!(s_k)^!(j_-^\mu)_!(\tilde\tau_{\mu,\bbR})[d_\mu])\is
\eeqn
\beqn
\is\on{Hom}_{D([\mO_K^\lambda])}(\tau_\lambda^+,
(u\circ i_+^\lambda)^!(j_-^\mu)_!(\tilde\tau_{\mu,\bbR})[d_\mu]).
\eeqn
Since $u\circ i_+^\lambda$ and $j_-^\mu$
are transversal, 
we have 
\[(u\circ i_+^\lambda)^!(j_-^\mu)_!\tau_{\mu,\bbR}\is
(u\circ i_+^\lambda)^*(j_-^\mu)_!\tau_{\mu,\bbR}\otimes\mL_\lambda'''[-d_\lambda]\] 
where 
\beq\label{L}
\mL_\lambda'''=(u\circ i_+^\lambda)^!\bC[d_\lambda]
\eeq
is a local system on $[\mO_K^\lambda]$
and, in view of the diagram \eqref{key diagram}, we get 
\beq\label{iso hom, 2}
\on{Hom}_{D(K(\calK)\backslash U_k)}(\mF_k,
(s_k)^!(s_k)_!\mF_k')\stackrel{\eqref{iso on hom,1}}\is\on{Hom}_{D([\mO_K^\lambda])}(\tau_\lambda^+,
(u\circ i_+^\lambda)^!(j_-^\mu)_!\tilde\tau_{\mu,\bbR}[d_\mu])\is
\eeq
\beqn
\is
\on{Hom}_{D([\mO_K^\lambda])}(\tau_\lambda^+,
(u\circ i_+^\lambda)^*(j_-^\mu)_!\tilde\tau_{\mu,\bbR}\otimes\mL[d_\mu-d_\lambda])
\eeqn
\beqn
\is
\on{Hom}_{D([\mO_K^\lambda])}((\phi_+^\lambda)^*\tau_\lambda,
\iota_!(u^\mu\circ s)^*\tilde\tau_{\mu,\bbR}\otimes\mL[d_\mu-d_\lambda])
\eeqn
\beqn
\is\on{Hom}_{D([\mO_K^\lambda])}((\phi_+^\lambda)^*\tau_\lambda,
\iota_!(u^\mu\circ s)^*\tilde\tau_{\mu,\bbR}\otimes\mL[d_\mu-d_\lambda])
\eeqn
\beqn
\on{Hom}_{D([\mO_c^\lambda])}(\tau_\lambda,
(\phi_+^\lambda)_*(\iota_!(u^\mu\circ s)^*\tilde\tau_{\mu,\bbR}\otimes\mL)[d_\mu-d_\lambda]).
\eeqn
\\
Consider the case $\lambda\neq\mu$. Then 
by Lemma \ref{Braden} we have 
\[(\phi_+^\lambda)_*(\iota_!(u^\mu\circ s)^*\tilde\tau_{\mu,\bbR}\otimes\mL)[d_\mu-d_\lambda])\is
(\phi_+^\lambda)_*\iota_!((u^\mu\circ s)^*\tilde\tau_{\mu,\bbR}\otimes\iota^*\mL))[d_\mu-d_\lambda])\is\]
\[\is
(\iota^\lambda_+)^*
\iota_!((u^\mu\circ s)^*\tilde\tau_{\mu,\bbR}\otimes\iota^*\mL))[d_\mu-d_\lambda])=
0,\] 
here $\iota_+^\lambda:[\mO_c^\lambda]\to[\mO_K^\lambda]$,
and it follows from \eqref{iso hom, 2} that 
\beq\label{iso Hom trivial, R}
\on{Hom}_{D(K(\calK)\backslash U_k)}(\mF_k,
(s_k)^!(s_k)_!\mF_k')=0.
\eeq
Hence we have
\beqn
\on{Hom}_{D(K(\calK)\backslash U_k)}(\mF_k,\mF_k')\stackrel{}\is
\on{Hom}_{D(K(\calK)\backslash U_k)}(\mF_k,
(s_k)^!(s_k)_!\mF_k')\is 0 \text{\ \ if\ \ }\lambda\neq\mu.
\eeqn
\\
Consider the case  $\lambda=\mu$. We have 
\[(\phi_+^\lambda)_*(\iota_!(u^\lambda\circ s)^*\tilde\tau_{\lambda,\bbR}\otimes\mL)[d_\lambda-d_\lambda])\is
(u^\lambda\circ s)^*\tilde\tau_{\lambda,\bbR}\otimes\iota^*\mL_\lambda'''
\is\tau_\lambda\otimes\mL_{\lambda}\otimes \iota^*\mL'''_\lambda.\] 
We claim that $\mL_{\lambda}\otimes\iota^*\mL'''_\lambda\is\bC$ is the trivial local system
hence above isomorphism implies 
\[(\phi_+^\lambda)_*(\iota_!(u^\lambda\circ s)^*\tilde\tau_{\lambda,\bbR}\otimes\mL)[d_\lambda-d_\lambda])\is\tau_\lambda,\ \text{\ if\ }
\lambda=\mu,\]
and by \eqref{iso hom, 2}, we obtain 
\beq\label{iso Hom, R}
\on{Hom}_{D(K(\calK)\backslash U_k)}(\mF_k,
(s_k)^!(s_k)_!\mF_k')\is\on{Hom}_{D([\mO_c^\lambda])}(\tau_\lambda,\tau_\lambda)
.
\eeq
By unwinding the definition of the map in \eqref{iso}, we obtain that \eqref{iso} satisfies
\beqn
\xymatrix{
\on{Hom}_{D(K(\calK)\backslash U_k)}(\mF_k,
\mF_k')\ar[rr]^{\eqref{iso}}\ar[rd]^\sim_{\eqref{iso Hom, K}}&&\on{Hom}_{D(K(\calK)\backslash U_k)}(\mF_k,
(s_k)^!(s_k)_!\mF_k')\ar[ld]_\sim^{\eqref{iso Hom, R}}\\
&\on{Hom}_{D([\mO_c^\lambda])}(\tau_\lambda,\tau_\lambda)&},
\eeqn
hence is an isomorphism.
The lemma follows.

To prove the claim, we observe that, up to cohomological shifts, we have 
\[\mL_\lambda\stackrel{\eqref{L_lambda}}\is(s)^*\mL_\lambda'\otimes\mL_\lambda''\otimes\on{or}_{p^\lambda}^\vee
\is (p^\lambda)^*((j_-^\lambda)^!\bC)^\vee)\otimes\iota^!\bC\otimes\on{or}_{p^\lambda}^\vee[-]
\]
\[\iota^*\mL_\lambda'''\stackrel{\eqref{L}}\is\iota^*((u\circ i_+^\lambda)^!\bC)[-].\]
Using the canonical isomorphisms $\iota^!(-)\is\iota^*(-)\otimes\iota^!\bC$ and 
$(p^\lambda)^!(-)\is (p^\lambda)^*(-)\otimes\on{or}_{p^\lambda}[-]$, we see that 
\[\mL_\lambda\otimes\iota^*\mL_\lambda'''\is (p^\lambda)^*((j_-^\lambda)^!\bC)^\vee)\otimes\on{or}_{p^\lambda}^\vee\otimes
\iota^!((u\circ i_+^\lambda)^!\bC)[-]\is\]
\[\is(p^\lambda)^*((j_-^\lambda)^!\bC)^\vee)\otimes\on{or}_{p^\lambda}^\vee\otimes
(p^\lambda)^!((j_-^\lambda)^!\bC))[-]\is
\]
\[\is(p^\lambda)^*((j_-^\lambda)^!\bC)^\vee)\otimes(p^\lambda)^*((j_-^\lambda)^!\bC))[-]
\is\bC[-].\]
The claim follows.
\end{proof}

\subsection{Proof of Theorem \ref{AM}}
Since the categories
$D(K(\calK)\backslash\Gr)$
and 
$D(LG_\bbR\backslash\Gr)$ are generated by 
standard (resp. co-standard) objects, 
Lemma \ref{surjective} and Lemma \ref{full} imply that the functor 
$\Upsilon:D(K(\calK)\backslash\Gr)\ra D(LG_\bbR\backslash\Gr)$
is essentially surjective and fully-faithful, and hence an equivalence. 
This finishes the proof of Theorem \ref{AM}.

\begin{remark}
Let $D_c(K(\calK)\backslash\Gr)\subset D(K(\calK)\backslash\Gr)$ be the 
full subcategory consisting of constructible complexes that are extension by 
zero off of a substack (equivalently, supported on a finite union 
of $K(\calK)$-orbits on $\Gr$).
We define $D_!(LG_\bbR\backslash\Gr)$ 
to the be full subcategory of 
$D(LG_\bbR\backslash\Gr)$ 
consisting of all constructible complexes that are extensions by zero off of finite type substacks of 
$LG_\bbR\backslash\Gr$.
The proof of Theorem \ref{AM} show that 
$\Upsilon$ restricts to an equivalence
\[\Upsilon:D_c(K(\calK)\backslash\Gr)\is D_!(LG_\bbR\backslash\Gr).\]
\end{remark}

\quash{
\begin{lemma}
\begin{enumerate}
\item
The functor 
$\on{For}_+$ admits a right adjoint functor 
which we denote it by 
$\on{Av}_+^r:D(LK_c\backslash\Gr)\ra D(K(\calK)\backslash\Gr)$. 
That is, we have 
\[
\on{Hom}_{D(LK_c\backslash\Gr)}(\on{For}_+(\mM),\mF)\is\on{Hom}_{D(K(\mathcal K)\backslash\Gr)}(\mM,\on{Av}_+^r\mF).\]
 \item The right adjoint $\on{Av}_+^r$ has a unique extension  $\on{Av}_+^r:\on{pro}(D(LK_c\backslash\Gr))
\ra\on{pro}(D(K(K)\backslash\Gr))$ such that 
\[\on{Hom}_{\on{pro}(D(LK_c\backslash\Gr))}(\on{For}_+(\mM),\mF)\is\on{Hom}_{\on{pro}(D(K(\mathcal K)\backslash\Gr))}(\mM,\on{Av}_+^r\mF)\]

\end{enumerate}
\end{lemma}

The lemmas above imply that the functor 
\[\Upsilon_-:=\on{Av}^r_+\circ u^!:D_!(LG_\bbR\backslash\Gr)\ra D(K(\calK)\backslash\Gr)\]
defines a right adjoint of $\Upsilon_+$.
\begin{proposition}
The functor $\Upsilon_-$ is the inverse equivalence of 
$\Upsilon_+$.
\end{proposition}}

\quash{
\begin{lemma}
The left adjoint functor of $\on{For}_-$ exists on the full subcategory 
$\on{For}_+(D(K(\calK)\backslash\Gr))\subset D(LK_c\backslash\Gr)$. 
More precisely, there exists a functor 
\[\on{Av}_-^l:\on{For}_+(D(K(\calK)\backslash\Gr))\ra D_!(LG_\bbR\backslash\Gr)\]
such that for
any $\mM\in D(K(\calK)\backslash\Gr)$, $\mF\in D_!(LG_\bbR\backslash\Gr)$
we have 
\[\on{Hom}_{D(LK_c\backslash\Gr)}(\on{For}_+(\mM),\on{For}_-(\mF))\is\on{Hom}_{D_!(LG_\bbR\backslash\Gr)}(\on{Av}_-^l\circ\on{For}_+(\mM),\mF).\]
\end{lemma}}

\section{Nearby cycles functors and Radon transforms}\label{nearby cycles and Radon TF}
We study the nearby cycles functors associated to the 
quasi-maps in Section \ref{QMaps} and the Radon transform 
for the real affine Grassmannian.

\subsection{A square of equivalences}
Recall the quasi-map family 
$QM^{(\sigma_2)}(\bP^1,X)_{\bbR}\to\bP^1$ in Section \ref{real form of QM}.
Consider the base change 
$QM^{(\sigma_2)}(\bP^1,X)_{\mathbb R}|_{i\bbR_{\geq 0}}$
along the natural inclusion $i\bbR_{\geq0}\to\mathbb P^1$.
By 
 Proposition \ref{open in family} and Proposition \ref{unif of QM}, 
we have the following cartesian diagram
\[\xymatrix{(LK_c\backslash\Gr)\times i\mathbb R_{>0}\ar[r]^j\ar[d]^{f^0}&
QM^{(\sigma_2)}(\bP^1,X)_{\mathbb R}|_{i\bbR_{\geq 0}}\ar[d]^{f}&K_c\backslash\Gr_{\mathbb R}\ar[d]^{f_0}\ar[l]_{\ \ \ \ \ \ i}\\
LG_\bbR\backslash\Gr\times i\mathbb R_{>0}\ar[r]^{\bar j}\ar[d]&\Bun_{G_\bbR}(\bbP^1(\bbR))\times i\mathbb R_{\geq 0}\ar[d]&G_\bbR(\bbR[t^{-1}])\backslash\Gr_{\bbR}\ar[l]_{\bar i}\ar[d]\\
i\mathbb R_{>0}\ar[r]&i\mathbb R_{\geq 0}&\{0\}\ar[l]}.\]
Define the following nearby cycles functors 
\beq
\Psi:D(K(\calK)\backslash\Gr)\is
D_{\calS}(LK_c\backslash\Gr)\ra D(K_c\backslash\Gr_{\bbR}),\ \ \mF\ra\Psi(\mF):=i^*j_*(\mF\boxtimes\underline{\bC}_{i\bbR_{>0}}),
\eeq
\beq
\Psi_{\bbR}:D(LG_\mathbb R\backslash\Gr)\ra D(G_\mathbb R(\mathbb R[t^{-1}])\backslash\Gr_\mathbb R),\ \ 
\mF\ra \Psi_{\bbR}(\mF)=(\bar i)^*(\bar j)_*(\mF\boxtimes\underline{\bC}_{i\bbR_{>0}}).
\eeq
We also have the Radon transform 
\beq\label{Radon TF}
\Upsilon_\bbR: D(G_\bbR(\mO_\bbR)\backslash\Gr_\bbR)\ra
D(G_\bbR(\bbR[t^{-1}])\backslash\Gr_\bbR)
\eeq
given by the restriction to $D(G_\bbR(\mO_\bbR)\backslash\Gr_\bbR)\subset
D(G_\bbR\backslash\Gr_\bbR)$
of the 
push-forward $p_!:D(G_\bbR\backslash\Gr_\bbR)\ra
D(G_\bbR(\bbR[t^{-1}])\backslash\Gr_\bbR)$
along the quotient map 
$p:G_\bbR\backslash\Gr\to G_\bbR(\bbR[t^{-1}])\backslash\Gr_\bbR$.
Denote by 
$D(\Bun_{G_\bbR}(\bbP^1(\bbR)))$
be the dg category of $\bC$-constructible complexes 
on $\Bun_{G_\bbR}(\bbP^1(\bbR))$.

Here are the main results of this section.
\begin{thm}\label{equ}
The nearby cycles functors and the Radon transform
induce equivalences of categories:
\[\Psi:D(K(\calK)\backslash\Gr)\stackrel{\sim}\lra D(G_\bbR(\mO_\bbR)\backslash\Gr_{\bbR}),\] 
\[\Psi_{\bbR}:D(LG_\mathbb R\backslash\Gr)\stackrel{\sim}\lra D(G_\mathbb R(\mathbb R[t^{-1}])\backslash\Gr_{\mathbb R}),\]
\[\Upsilon_\bbR:D(G_\bbR(\mO_\bbR)\backslash\Gr_\bbR)\stackrel{\sim}\lra
D(G_\bbR(\bbR[t^{-1}])\backslash\Gr_\bbR).\]

\end{thm}

\begin{thm}\label{diagram}
We have a commutative square of equivalences 
\[
\xymatrix{D(K(\calK)\backslash\Gr)\ar[rr]^{\Psi}\ar[d]^{\Upsilon}&&D(G_\mathbb R(\mO_\mathbb R)\backslash\Gr_{\mathbb R})
\ar[d]^{\Upsilon_\bbR}\\
D(LG_{\mathbb R}\backslash\Gr)\ar[rr]^{\Psi_{\bbR}\ \ }\ar[dr]_\simeq&&D(G_\mathbb R(\mathbb R[t^{-1}])\backslash\Gr_{\mathbb R})\ar[dl]^\simeq\\
&D(\Bun_{G_\bbR}(\bbP^1(\bbR)))&}
\]
where the vertical equivalences in the lower triangle come from the 
real and complex uniformization isomorphism
\[\xymatrix{LG_{\mathbb R}\backslash\Gr\ar[r]^{\simeq\ \ }&\Bun_{G_\bbR}(\bbP^1(\bbR))&G_\mathbb R(\mathbb R[t^{-1}])\backslash\Gr_\bbR\ar[l]_{\ \ \simeq}}\]
\end{thm}

The rest of the section is devoted to the proof of Theorem \ref{equ} and Theorem \ref{diagram}.

\subsection{The nearby cycles functor $\Psi$}
For any $\lambda\in\mL$ and a $K_c$-equivaraint local system $\omega$ on $C_\bbR^\lambda$, 
one has the standard and co-standard sheaves 
\[\calT_*^+(\lambda,\omega):=(s_\lambda^+)_*(\omega^+) \text{\ \ and\ \ \ } 
\calT_!^+(\lambda,\omega):=(s_{\lambda}^+)_!(\omega^+)\]
in $D(G_\bbR(\mO_\bbR)\backslash\Gr_\bbR)$ (see Lemma \ref{bijection 2}).
Recall the standard sheaf
$\calS_*^+(\lambda,\tau)$ in $D(K(\calK)\backslash\Gr)$ (see (\ref{standard})).

\begin{proposition}\label{Psi_K}
The nearby cycles functor induces an equivalence 
$\Psi:D(K(\calK)\backslash\Gr)\is D_{}(G_\bbR(\mO_\bbR)\backslash\Gr_{\bbR})$, which is the real-symmetric equivalence in Theorem \ref{real-symmetric}.
Moreover,
we have $\Psi(\calS_*^+(\lambda,\tau))\is\calT_*^+(\lambda,\omega)$
\end{proposition}
\begin{proof}
By Lemma \ref{iso to Q},
 there is a  $K_c$-equivariant 
topological trivialization of the quasi-maps family 
$\QM^{(\sigma_{2})}(\bbP^1,X,\infty)_{\mathbb R}\to \bC$
and thus the nearby cycles functor is the same as the 
real-symmetric equivalence in Theorem \ref{real-symmetric}.
The proposition follows.

\end{proof}

\subsection{The Radon transform}
Recall the flow $\psi_z^1:\Gr_\bbR^{(\sigma_2)}\to\Gr_\bbR^{(\sigma_2)}$ 
in  \eqref{flow on Gr_R^2}. By Lemma \ref{flows on Gr^(2)}, it
restricts to a 
flow on the special fiber $\Gr_\bbR\is\Gr_\bbR^{(\sigma_2)}|_0$ with 
critical manifolds $\bigcup_{\lambda\in\Lambda_A^+} C_\bbR^\lambda$ and 
$S_\bbR^\lambda$, respectively $T_\bbR^\lambda$, is the stable manifold, respectively unstable manifold, of $C_\bbR^\lambda$.
Let 
$t_\lambda^-:G_\bbR(\bbR[t^{-1}])\backslash T_\bbR^\lambda\ra G_\bbR(\bbR[t^{-1}])\backslash\Gr_\bbR 
$ be the natural inclusion map. 
According to Lemma \ref{bijection 2}, for any $K_c$-equivariant local system 
$\omega$ on $C_\bbR^\lambda$,
we have the standard and 
co-standard sheaves
\[\calT_*^-(\lambda,\omega):=(t_\lambda^-)_*(\omega_\bbR) \text{\ \ and\ \ \ } 
\calT_!^-(\lambda,\omega):=(t_{\lambda}^-)_!(\omega_\bbR).\]
Recall the Radon transform \[\Upsilon_\bbR:
D(G_\bbR(\mO_\bbR)\backslash\Gr_\bbR)\lra
D(G_\bbR(\bbR[t^{-1}])\backslash\Gr_\bbR)\] in \eqref{Radon TF}.
The same argument as in the proof of Theorem \ref{AM}, replacing the 
Matsuki flow $\phi_t:\Gr\to\Gr$ by the
$\bbR_{>0}$-flow $\psi_z^1:\Gr_\bbR\to\Gr_\bbR$ and Lemma \ref{bijection} by Lemma \ref{bijection 2},
gives us:
\begin{prop}\label{real RT}
The Radon transform defines an equivalence of categories 
\[\Upsilon_\bbR:D(G_\bbR(\mO_\bbR)\backslash\Gr_\bbR)\stackrel{\sim}\lra
D(G_\bbR(\bbR[t^{-1}])\backslash\Gr_\bbR).\]
Moreover, for any $K_c$-equivariant local system $\omega$
on $C_\bbR^\lambda$ we have 
\[\Upsilon_\bbR(\calT_*^+(\lambda,\omega))\is\calT^-_!(\lambda,\omega\otimes\mL_\lambda)[d_\lambda].\]
Here we regard the local system $\mL_\lambda$ in \eqref{L_lambda} as 
a local system on $C_\bbR^\lambda$ via the isomorphism 
$k_\lambda:\Omega G_c\backslash\mO_c^\lambda\is C_\bbR^\lambda$
in Lemma \ref{bijection 2}.

\end{prop}

\subsection{The functor $\Psi_\bbR$}
Note that, by Proposition \ref{open in family},  the map $LG_\bbR^{(\sigma_2)}\backslash\Gr^{(\sigma_2)}_{\bbR}|_{i\bbR_{\geq0}}\is
\Bun_{G_\bbR}(\bbP^1(\bbR))\times i\mathbb R_{\geq0}
\ra i\mathbb R_{\geq0}$ is isomorphic to a constant family. It implies

\begin{prop}\label{Psi_R}
The nearby cycles functor 
\[\Psi_{\bbR}:D(LG_\mathbb R\backslash\Gr)\stackrel{}\lra D(G_\mathbb R(\mathbb R[t^{-1}])\backslash\Gr_{\mathbb R})\]
is an equivalence satisfying 
$\Psi_{\bbR}(\calS_!^-(\lambda,\tau))\is \calT_!^-(\lambda,\omega)$.
\end{prop}


\subsection{Proof  of Theorem \ref{equ} and Theorem \ref{diagram}}
We already proved Theorem \ref{equ}.
We have $\Psi(\mF)\is i^*j_*(\mF\boxtimes\underline{\bC}_{i\bbR_{>0}})\is
i^!j_!(\mF\boxtimes\underline{\bC}_{i\bbR_{>0}}[1])$
and the natural arrow 
$(f_0)_!i^!\ra(\bar i)^!f_!$ gives rise to a natural transformation
\beq\label{natural}
\Upsilon_\bbR\circ\Psi(\mF)\is(f_0)_!i^!j_!(\mF\boxtimes\bC_{i\bbR_{>0}}[1])\ra (\bar i)^!f_!j_!(\mF\boxtimes\bC_{i\bbR_{>0}}[1])
\is(\bar i)^!(\bar j)_!(f^0)_!(\mF\boxtimes\bC_{i\bbR_{>0}}[1])\is
\eeq
\[\is\Psi_{\bbR}\circ\Upsilon(\mF).
\]
Moreover, it follows from Lemma \ref{surjective}, Proposition \ref{Psi_K},  
Proposition \ref{real RT}, and Proposition \ref{Psi_R} that (\ref{natural}) is an isomorphism for
the standard sheaf $\calS_*^+(\lambda,\tau)$. 
Since the category $D_c(K(\calK)\backslash\Gr)$ is generated by 
$\calS_*^+(\lambda,\tau)$, it implies (\ref{natural}) is an isomorphism.
Theorem \ref{diagram} follows.


\section{Compatibility of Hecke actions}\label{s:hecke}
Recall the derived Satake category $D(G(\mO)\backslash \Gr)$ is naturally monoidal  with respect to convolution.
We will write $\calF_1 \star\calF_2$ for the convolution product of $\calF_1, \calF_2\in D(G(\mO)\backslash \Gr)$.

Here we enhance the equivalences and commutative square of Theorems~\ref{equ} and~\ref{diagram} to  $D(G(\mO)\backslash \Gr)$-modules. 
 Roughly speaking, we will  take advantage of  the natural right actions on the categories involved, whereas the prior Radon transforms were performed on the left.


\subsection{Hecke actions}\label{ss:conv}

First,    the affine Matsuki correspondence for sheaves 
\[
\Upsilon:D(K(\calK)\backslash\Gr)  \stackrel{\sim}\lra
D(LG_{\mathbb R}\backslash\Gr)
\]
is naturally an equivalence of $D(G(\mO)\backslash \Gr)$-modules by convolution on the right. To see this, 
recall $\Upsilon$ is the restriction to $D(K(\calK)\backslash\Gr)\subset D(LK_c\backslash\Gr)$ of the push-forward
$u_!:D(LK_\bbR\backslash\Gr)\to D(LG_\bbR\backslash\Gr)$
along the quotient map $u:LK_c\backslash\Gr\ra
LG_\bbR\backslash\Gr$.
We can equip this construction with compatibility with convolution on the right by using the commutative action diagram
\[
\xymatrix{
\ar[d] LK_c\backslash G(\calK) \times^{G(\mO)} \Gr  \ar[r]^-{u \times\id} & 
LG_\bbR\backslash G(\calK) \times^{G(\mO)} \Gr \ar[d] \\
LK_c\backslash\Gr \ar[r]^-u & 
LG_\bbR\backslash\Gr
}
\]
and its natural iterations.

Similarly,
   the Radon equivalence
\[\Upsilon_\bbR:D(G_\bbR(\mO_\bbR)\backslash\Gr_\bbR)\stackrel{\sim}\lra
D(G_\bbR(\bbR[t^{-1}])\backslash\Gr_\bbR)\]
is naturally an equivalence of $D_c(G_\bbR(\mO_\bbR)\backslash\Gr_\bbR)$-modules by convolution on the right.
 To see this, 
recall $\Upsilon_\bbR$ is
 the restriction to $D(G_\bbR(\mO_\bbR)\backslash\Gr_\bbR)\subset
D(G_\bbR\backslash\Gr_\bbR)$
of the 
push-forward $p_!:D(G_\bbR\backslash\Gr_\bbR)\ra
D(G_\bbR(\bbR[t^{-1}])\backslash\Gr_\bbR)$
along the quotient map 
$p:G_\bbR\backslash\Gr\to G_\bbR(\bbR[t^{-1}])\backslash\Gr_\bbR$.
We can equip this construction with compatibility with convolution on the right by using the commutative action diagram
\[
\xymatrix{
\ar[d] G_\bbR\backslash G(\calK_\bbR) \times^{G(\mO_\bbR)} \Gr_\bbR  \ar[r]^-{p \times\id} & 
G_\bbR(\bbR[t^{-1}])\backslash G(\calK_\bbR) \times^{G(\mO_\bbR)} \Gr_\bbR \ar[d] \\
G_\bbR\backslash\Gr_\bbR \ar[r]^-p & 
G_\bbR(\bbR[t^{-1}])\backslash\Gr_\bbR
}
\]
and its natural iterations.


\subsection{From complex to real kernels}\label{ss:kernels}

Following~\cite{N1}, nearby cycles in the real Beilinson-Drinfeld Grassmannian $\Gr^{(\sigma_2)}_\mathbb R$ 
 over  $i \bbR_{\geq 0}$
gives a  functor
\beq\label{psi}
\psi:D(G(\calO) \backslash\Gr)\lra D(G_\bbR(\mO_\bbR)\backslash\Gr_{\bbR})
\eeq
Namely, 
there is a  canonical diagram of  $G_\bbR$-equivariant maps
  \begin{equation}\label{c to r diag}
\xymatrix{
 \Gr & \ar[l]_-\pi \Gr\times i\bbR_{>0} \is \Gr^{(\sigma_2)}_\bbR|_{
i\mathbb R_{>0}} \ar@{^(->}[r]^-j &  \Gr^{(\sigma_2)}_\bbR|_{
i\mathbb R_{\geq 0}} & \ar@{_(->}[l]_-i \Gr^{(\sigma_2)}_\mathbb R|_{0}\is\Gr_\mathbb R
}
  \end{equation}
where we view $G_\bbR \subset LG^{(2)}_\bbR$ as the constant group-scheme.
One defines $\psi=i^*j_*\pi^*f_\bbR$ where we write  $f_\bbR:D(G(\calO) \backslash\Gr) \to D(G_\bbR \backslash\Gr)$ 
for the forgetful functor. 

Note the domain and codomain of $\psi$ both have natural convolution monoidal structures. 
To equip $\psi$ with a monoidal structure, we proceed as follows. 

Let $\Gr^{(2)} \tilde\times \Gr^{(2)}$ be  the moduli of $x_1, x_2 \in \bbP^1$, $\calE_1, \calE_2$ $G$-torsors on $\bbP^1$, $\phi$ a trivialization of $\calE_1$ over $\bbP^1 \setminus \{x_1, x_2\}$, and $\alpha$ an isomorphism from $\calE_1$ to $\calE_2$ over  $\bbP^1 \setminus \{x_1, x_2\}$. 
Let $\Gr^{(\sigma_2)}_\bbR \tilde\times \Gr^{(\sigma_2)}_\bbR$ be the real form of $\Gr^{(2)} \tilde\times \Gr^{(2)}$ with respect to the twisted conjugation that exchanges $x_1$ and $x_2$. 

Then
there is a  canonical diagram of  $G_\bbR$-equivariant maps
\begin{equation} \label{eq: nearby + conv}
  \xymatrix{
 G(\calK) \times^{G(\calO)} \Gr & \ar[l]_-\pi G(\calK) \times^{G(\calO)} \Gr\times i\bbR_{>0} \is\Gr^{(\sigma_2)}_\bbR \tilde\times \Gr^{(\sigma_2)}_\bbR|_{
i\mathbb R_{>0}} \ar@{^(->}[r]^-j &  }
\end{equation}
$$
\xymatrix{
\Gr^{(\sigma_2)}_\bbR \tilde\times \Gr^{(\sigma_2)}_\bbR|_{i\mathbb R_{\geq 0}} & \ar@{_(->}[l]_-i 
\Gr^{(\sigma_2)}_\bbR \tilde\times \Gr^{(\sigma_2)}_\bbR |_{0}\is   G_\bbR(\calK_\bbR)  \times^{G_\bbR(\calO_\bbR)} \Gr_\mathbb R
}
$$
Moreover, the  convolution maps on the end terms naturally extend to the entire diagram. By standard identities, we arrive at a canonical isomorphism $\psi(\calF_1 \star \calF_2) \simeq \psi(\calF_1) \star \psi(\calF_2)$. By using iterated versions of the above moduli spaces, we may likewise equip $\psi$ with the associativity constraints of a monoidal structure.


\subsection{Compatibility of actions}\label{ss:actions}

Note we can view the  Radon equivalence  $\Upsilon_\bbR$ as an equivalence of $D(G(\calO) \backslash\Gr)$-modules 
via the monoidal functor 
\[
\psi:D(G(\calO) \backslash\Gr)\lra D(G_\bbR(\mO_\bbR)\backslash\Gr_{\bbR})
\]

Now we have the following further compatibility of our constructions.

\begin{thm}\label{conv comp}
Via the monoidal functor
\[\psi:D(G(\calO) \backslash\Gr)\lra D(G_\bbR(\mO_\bbR)\backslash\Gr_{\bbR})
\]
the  equivalences 
\[\Psi:D(K(\calK)\backslash\Gr)\stackrel{\sim}\lra D(G_\bbR(\mO_\bbR)\backslash\Gr_{\bbR}),\] 
\[\Psi_{\bbR}:D(LG_\mathbb R\backslash\Gr)\stackrel{\sim}\lra D(G_\mathbb R(\mathbb R[t^{-1}])\backslash\Gr_{\mathbb R})\]
of Theorem~\ref{equ} and commutative square
\[
\xymatrix{D(K(\calK)\backslash\Gr)\ar[r]^{\Psi}\ar[d]^{\Upsilon}&D(G_\mathbb R(\mO_\mathbb R)\backslash\Gr_{\mathbb R})
\ar[d]^{\Upsilon_\bbR}\\
D(LG_{\mathbb R}\backslash\Gr)\ar[r]^{\Psi_{\bbR}\ \ }&D(G_\mathbb R(\mathbb R[t^{-1}])\backslash\Gr_{\mathbb R}).}
\]
of Theorem~\ref{diagram} are naturally of  $D(G(\calO) \backslash\Gr)$-modules.
\end{thm}

\begin{proof}
We will focus on the compatibility for the top row and indicate the moduli spaces needed. We  leave it to the reader to pass to sheaves and apply standard identities.  The  compatibility for the bottom row and entire square can be argued  similarly.

Let $\QM^{(2)}(\bP^1,X) \tilde\times \Gr^{(2)}$ be  the moduli of $x_1, x_2 \in \bbP^1$, $\calE_1, \calE_2$ $G$-torsors on $\bbP^1$, $\sigma$ a section of $\calE_1 \times^G X$ over $\bbP^1 \setminus \{x_1, x_2\}$, and $\alpha$ an isomorphism from $\calE_1$ to $\calE_2$ over  $\bbP^1 \setminus \{x_1, x_2\}$. 
Let $\QM^{(\sigma_2)}(\bP^1,X)_{\bbR}  \tilde\times \Gr^{(\sigma_2)}_{\bbR} $ be the real form of $\QM^{(2)}(\bP^1,X) \tilde\times \Gr^{(2)}$ with respect to the twisted conjugation that exchanges $x_1$ and $x_2$.

Then
there is a  canonical diagram of  $K_c$-equivariant maps
\begin{equation}\label{eq: nearby + action}
  \xymatrix{
LK_c\backslash G(\calK) \times^{G(\mO)} \Gr  & \ar[l]_-\pi  LK_c\backslash G(\calK) \times^{G(\mO)} \Gr  \times i\bbR_{>0} \is \QM^{(\sigma_2)}(\bP^1,X)_{\bbR}  \tilde\times \Gr^{(\sigma_2)}_{\bbR}|_{
i\mathbb R_{>0}} \ar@{^(->}[r]^-j &  }
\end{equation}
$$
\xymatrix{
\QM^{(\sigma_2)}(\bP^1,X)_{\bbR}  \tilde\times \Gr^{(\sigma_2)}_{\bbR}|_{i\mathbb R_{\geq 0}} & \ar@{_(->}[l]_-i 
\QM^{(\sigma_2)}(\bP^1,X)_{\bbR}  \tilde\times \Gr^{(\sigma_2)}_{\bbR} |_{0}\is   K_c\backslash G_\bbR(\calK_\bbR)  \times^{G_\bbR(\calO_\bbR)} \Gr_\mathbb R
}
$$
Note we could equivalently obtain diagram  \eqref{eq: nearby + action} by taking diagram  \eqref{eq: nearby + conv}
and quotienting by the left action of  the group-scheme $LK^{(\sigma_2)}_\bbR$.

As with the convolution maps in diagram \eqref{eq: nearby + conv}, the  actions maps on the end terms of diagram  
 \eqref{eq: nearby + action} naturally extend to the entire diagram. By standard identities, we arrive at a canonical isomorphism $\Psi(\calM \star \calF) \simeq \Psi(\calM) \star \psi(\calF)$. By using iterated versions of the above moduli spaces, we may likewise equip $\Psi$ with the associativity constraints of a module map.
\end{proof}

\section{Compatibility with fusion product}\label{fusion}
We show that the real-symmetric equivalences 
in Theorem \ref{real-symmetric}
are compatible with the natural fusion products

\subsection{Fusion product for $\Gr_\bbR$}
We first define fusion product for the real Satake category.
Consider the 
family
\[\Gr_{\bbR}^{(2)}|_{\bbR_{\geq0}}\lra \bbR_{\geq0}\]
obtained by the restriction of the family
$\Gr_\bbR^{(2)}\to\bbR^2$  
along 
embedding
$\bbR_{\geq0}\hookrightarrow \bbR^2$ sending $t\to (t,-t)$. 
The natural action of $G(\calO)^{(2)}_\bbR$ on $\Gr_{\bbR}^{(2)}$
is compatible with the factorization isomorphisms in Section \ref{real Gr} and it follows that 
there are 
Cartesian diagrams
\[\xymatrix{(G_\bbR(\calO_\bbR)\backslash\Gr_\bbR)^2\times \bbR_{>0}\ar[r]^{\ \ \ \ \ \ j_\bbR}\ar[d]&(G(\calO)^{(2)}_\bbR\backslash\Gr^{(2)}_\bbR)|_{\bbR_{\geq0}}\ar[d]&G_\bbR(\calO_\bbR)\backslash\Gr_\bbR\ar[d]\ar[l]_{i_\bbR}\\
\bbR_{>0}\ar[r]&\bbR_{\geq0}&\{0\}\ar[l]}\]
For any $\mF,\mF'\in D(G_\bbR(\calO_\bbR)\backslash\Gr_\bbR)$, we define the fusion product $\mF\star_\bbR\mF'$
as the following nearby cycles:
\[\mF\star_f\mF':= (i_\bbR)^*(j_\bbR)_*(\mF\boxtimes\mF'\boxtimes\underline{\bC}_{\bbR_{>0}})\]
where $\underline{\bC}_{\bbR_{>0}}$ is the constant sheaf on $\bbR_{>0}$.
Like in the case of complex reductive groups,
there is a natural monoidal structure on 
$D(G_\bbR(\calO_\bbR)\backslash\Gr_{\bbR})$
given by the convolution product:
consider the convolution diagram
\beq\label{convolution diagram}
\Gr_{\bbR}\times\Gr_{\bbR}\stackrel{p}\la
G_\bbR(\calK_\bbR)\times\Gr_\bbR\stackrel{q}\ra
\Gr_{\bbR}\tilde\times\Gr_{\bbR}:=G_\bbR(\calK_\bbR)\times^{G_\bbR(\calO_\bbR)}\Gr_{\bbR}\stackrel{m}\to\Gr_{\bbR}
\eeq
where $p$ and $q$ are the natural quotient maps and
 $m(x,y\mod G_{\bbR}(\calO_\bbR))= xy\mod G_{\bbR}(\calO_\bbR)$.
 For any $\mF_1,\mF_2\in D(G_\bbR(\calO_\bbR)\backslash\Gr_{\bbR})$,
 the convolution 
 is defined as 
 \[\mF_1\star\mF_2=m_!(\mF_1\tilde\boxtimes\mF_2)\]
 where $\mF_1\tilde\boxtimes\mF_2\in D(G_\bbR(\calO_\bbR)\backslash\Gr_{\bbR})$ is the unique complex
 such that $q^*(\mF_1\tilde\boxtimes\mF_2)\is p^*(\mF_1\boxtimes\mF_2)$.

\begin{prop}\label{real convolution}
(1) There is a natural isomorphism 
\[\star_f\is\star:D(G_\bbR(\calO_\bbR)\backslash\Gr_\bbR)\times D(G_\bbR(\calO_\bbR)\backslash\Gr_\bbR)\to D(G_\bbR(\calO_\bbR)\backslash\Gr_\bbR)\] of functors.
(2) The fusion and convolution product are $t$-exact with respect to the 
perverse $t$-structure.
%

\end{prop}
\begin{proof}
Parts (1)  follows from the same proof as in the complex case using the real global convolution diagram 
in \cite[Section 6.3]{N2}.
Part (2) follows from part (1) and the fact that 
the real convolution map $m$ is a stratified semismall map, see \cite[Section 3.8]{N2}.

\end{proof}

\subsection{
Fusion product 
 for $X(\calK)$}\label{fusion product for X}
We define and study the fusion product for the relative Satake category.

Consider the base change 
$\Gr^{(2)}|_{\bC}\ra \bC$
of $\Gr^{(2)}\to(\bP^1)^2$ along the embedding 
$
\bC\hookrightarrow(\bP^1)^2,\ \ t\to (t,-t).
$
Since the action of the ind-group scheme $K(\calK)^{(2)}$ on
$\Gr^{(2)}$ is compatible with the factorization isomorphisms
 it follows that 
there are 
Cartesian diagrams
\[\xymatrix{(X(\calK)/G(\calO))^2\times \bbR_{>0}\is (K(\calK)\backslash\Gr)^2\times \bbR_{>0}\ar[r]^{\ \ \ \ \ \ \ \ \ \ \ \ \ \ j_\bbR}\ar[d]^{s|_{\bC^\times}}&(K(\calK)^{(2)}\backslash\Gr^{(2)})|_{ \bbR_{\geq0}}\ar[d]^{s}&K(\calK)\backslash\Gr\is X(\calK)/G(\calO)\ar[d]^{id}_{\simeq}\ar[l]_{i_\bbR\ }\\
(X(\calK)/G(\calO))^2\times \bC^\times\is (K(\calK)\backslash\Gr)^2\times \bC^\times\ar[r]^{\ \ \ \ \ \ \ \ \ \ \ \ \ \ j}\ar[d]&(K(\calK)^{(2)}\backslash\Gr^{(2)})|_{ \bC}\ar[d]^{}&K(\calK)\backslash\Gr\is X(\calK)/G(\calO)\ar[d]\ar[l]_{i\ }\\
\bC^\times\ar[r]&\bC&\{0\}\ar[l]}\]
where $s$ is the natural closed embedding.

Since  $j$ is a fp-open embedding and $i$ is an fp-closed embedding,
\cite[Lemma 5.4.1]{BKV} and Proposition \ref{key base change}
implies that the functors $j^!$ and $i^!$ admits 
a (continuous) right adjoint $j_*$ and a left adjoint $i^*$ respectively.
 
 For any $\mF,\mF'\in D(X(\calK)/G(\calO))$ (resp. $\mF,\mF'\in D(K(\calK)\backslash\Gr))$ we define the fusion product 
$\mF\star_f\mF'$ 
as the following nearby cycles 
\[\mF\star_f\mF':= i^*j_*(\mF\boxtimes\mF'\boxtimes\underline{\bC}_{\bbR_{>0}})\]
here we view $\underline{\bC}_{\bbR_{>0}}$ as constant sheaf 
supported on $\bbR_{>0}\subset\bC^\times$.

Note that the left adjoint $(i_\bbR)^*$  of $(i_\bbR)_*$ 
and right adjoint $(j_\bbR)_*$ of $(j_\bbR)^!$
exist and are isomorphic to 
$(i_\bbR)^*\is i^*s_*$ and $(j_\bbR)_*\is s^!j_*(s|_{\bC^\times})_*$.
Indeed, we have 
\[\on{Hom}(i^*s_*\mF,\mF')\is\Hom(s_*\mF,i_*\mF')\is\Hom(s_*\mF,s_*(i_\bbR)_*\mF')\is\Hom(\mF,(i_\bbR)_*\mF')\]
\[\on{Hom}(\mM,s^!j_*(s|_{\bC^\times})_*\mM')\is\Hom(s_*\mM,j_*(s|_{\bC^\times})_*\mM')\is\Hom(j^!s_*\mM,(s|_{\bC^\times})_*\mM')\is\]
\[\is\Hom((s|_{\bC^\times})_*(j_\bbR)^!\mM,(s|_{\bC^\times})_*\mM')\is
\Hom((j_\bbR)^!\mM,\mM')\]
where the last isomorphisms follow from the 
fact that the functors $s_*$ and $(s|_{\bC^\times})_*$ are  fully-faithful. 
There is an isomorphism
\[\mF\star_f\mF'\is (i_\bbR)^* (j_\bbR)_*(\mF\boxtimes\mF'\boxtimes\underline{\bC}_{\bbR_{>0}})\]

\begin{thm}\label{fusion compatibility}
\
\begin{enumerate}
\item
The equivalences  
$D(X(\calK)/G(\calO))\is 
D(K(\calK)\backslash\Gr)\is D(G_\bbR(\calO_\bbR)\backslash\Gr_{\bbR})
$
are compatible with the fusion products 
\item  
The fusion product  for
$D(X(\calK)/G(\calO))$ (resp. $D(K(\calK)\backslash\Gr)$)
is $t$-exact with respect to the perverse $t$-structure.
\end{enumerate}
\end{thm}
\begin{proof}
Proof of (1). The compatibility for   $D(X(\calK)/G(\calO))\is 
D(K(\calK)\backslash\Gr)$ follows from the definition.
Consider the family of 
 embeddings
\[\iota_a:\bC\to(\bP^1)^2\ \ \ t\to (t+ai,-t+ai)\]
parametrized by $a\in\bbR$. 
There is a natural isomorphism
\beq\label{translation}
(K(\calK)^{(2)}\backslash\Gr^{(2)})|_{\iota_a(\bC)}
\is (K(\calK)^{(2)}\backslash\Gr^{(2)})|_{\iota_0(\bC)}=(K(\calK)^{(2)}\backslash\Gr^{(2)})|_{\bC}
\eeq
induced by the translation isomorphism 
$\bbP^1\is\bbP^1, x\to x+ai$ and 
the stratified homeomorphism  in Theorem \ref{Quillen}
restricts to a  stratified homeomorphism 
\beq\label{Quillen R>=0}
\Gr_\bbR^{(2)}|_{\bbR_{\geq0}}\is (\Omega K_c^{(2)}\backslash\Gr^{(2)})|_{\iota_1(\bbR_{\geq0})}
\eeq
Let us consider the following Cartesian diagram 
\[\xymatrix{ 
(\Gr_\bbR)^2\times \bbR_{>0}\ar[d]^\simeq\ar[r]^{\ \ j_\bbR}&\Gr_\bbR^{(2)}|_{\bbR_{\geq0}}\ar[d]_{\eqref{Quillen R>=0}}^\simeq&\Gr_\bbR\ar[l]_{i_\bbR}\ar[d]^\simeq\\
(\Omega K_c^{(2)}\backslash\Gr^{(2)})|_{\bbR_{>0}}\ar[r]^{\ \ j_\bbR}\ar[d]^{h^{(2)}|_{\bbR_{>0}}}&(\Omega K_c^{(2)}\backslash\Gr^{(2)})|_{\iota_1(\bbR_{\geq0})}\ar[d]^{h^{(2)}}&\Omega K_c\backslash\Gr\ar[d]^{h}\ar[l]_{\ \ \ \ \ i_\bbR}\\
(K(\calK)\backslash\Gr)^2\times \bbR_{>0}\ar[r]^{ j_\bbR}\ar[d]&(K(\calK)^{(2)}\backslash\Gr^{(2)})|_{\bbR_{\geq0}}\ar[d]&K(\calK)\backslash\Gr\ar[d]\ar[l]_{i_\bbR}\\
\bbR_{>0}\ar[r]&\bbR_{\geq0}&\{0\}\ar[l]}\]
where 
\[h^{(2)}:(\Omega K_c^{(2)}\backslash\Gr^{(2)})|_{\iota_1(\bbR_{\geq0})} 
\stackrel{\pr}\lra
(K(\calK)^{(2)}\backslash\Gr^{(2)})|_{\iota_1(\bbR_{\geq0})}\stackrel{\eqref{translation}}\is(K(\calK)^{(2)}\backslash\Gr^{(2)})|_{\bbR_{\geq0}} \]
and 
\[\pr:(\Omega K_c^{(2)}\backslash\Gr^{(2)})\to (K(\calK)^{(2)}\backslash\Gr^{(2)})\] 
is  
the natural  map 
induced by the natural inclusion $\Omega K_c^{(2)}\to K(\calK)^{(2)}$. 
Note that there is an isomorphism
\[(\Omega K_c^{(2)}\backslash\Gr^{(2)})\is K(\calK)^{(2)} \backslash (\Gr^{(2)}\times_{(\bP^1)^2}
K(\calK)^{(2)}/\Omega K_c^{(2)})\]
where  $K(\calK)^{(2)}$-acts diagonally, and 
the functor 
\[\pr^!\is (h^{(2)})^!: D(K(\calK)^{(2)}\backslash\Gr^{(2)}|_{\bbR_{\geq0}})\is\on{colim}D((\Gr^{(2)})^{[n]}|_{\bbR_{\geq0}})\lra
 D((\Omega K_c^{(2)}\backslash\Gr^{(2)})|_{\bbR_{\geq0}})\is\]
 \[\is\on{colim}D((\Gr^{(2)}\times_{(\bP^1)^2}
K(\calK)^{(2)}/\Omega K_c^{(2)})^{[n]})\]
is induced by the pullback functors $(\pr^{[n]})^!$ along the 
projection maps between
the terms of the $\check{\on{C}}$ech complexes:
\[\pr^{[n]}:(\Gr^{(2)}\times_{(\bP^1)^2}
K(\calK)^{(2)}/\Omega K_c^{(2)})^{[n]}=(\Gr^{(2)}\times_{(\bP^1)^2} (K(\calK)^{(2)}/\Omega K_c^{(2)}))\times_{(\bP^1)^2} (K(\calK)^{(2)})^n\stackrel{}\ra\]
\[\ra(\Gr^{(2)})^{[n]}\is
\Gr^{(2)}\times_{(\bP^1)^2}  (K(\calK)^{(2)})^n\] 
Note that 
the Gram-Schmidt factorization in Proposition \ref{multi Gram-Schmidt}
implies that quotient
$K(\calK)^{(2)}/\Omega K_c^{(2)}\is K(\calO)^{(2)}$
 is strongly pro-smooth, that is, $K(\calO)^{(2)}$ is a projective limit of smooth schemes with smooth affine transition maps.
Thus the map $\pr^{[n]}$ is also strongly pro-smooth and it follows from \cite[Proposition 5.2.7 (d) and Lemma 5.4.5]{BKV} 
that  $(\pr^{[n]})^!$ and hence $\pr^!\is (h^{(2)})^!$ commutes with nearby cycles, that is,
we have 
\[h^!(i_\bbR)^*(j_\bbR)_*(-)\is
(i_\bbR)^*(j_\bbR)_*(h^{(2)}|_{\bbR^\times})^!(-).\]
Since  $h^!:D(K(\calK)\backslash\Gr)\to D(\Omega K_c\backslash\Gr)\is D(\Gr_\bbR)$ gives rise to the equivalence 
$D(K(\calK)\backslash\Gr)\is D(G_\bbR(\calO_\bbR)\backslash\Gr_\bbR)$
in Theorem \ref{real-symmetric}, 
for any $\mF,\mF'\in D(K(\calK)\backslash\Gr)$, there is a natural isomorphism
\[h^!(\mF\star_f\mF')\is h^!((i_\bbR)^*(j_\bbR)_*(\mF\boxtimes\mF'\boxtimes\underline\bC_{\bbR_{>0}}))\is
(i_\bbR)^*(j_\bbR)_*(h^{(2)}|_{\bbR^\times})^!(\mF\boxtimes\mF'\boxtimes\underline\bC_{\bbR_{>0}})\is\]
\[\is 
(i_\bbR)^*(j_\bbR)_*(h^!\mF\boxtimes h^!\mF'\boxtimes\underline\bC_{\bbR_{>0}})\is h^!\mF\star_f h^!\mF'.\]
Part (1) follows.
 
Since the equivalences in (1) are $t$-exact, 
Part (2) follows from part (1) and Lemma \ref{real convolution}.

\end{proof}

\quash{
\begin{corollary}
There are natural equivalences 
\[\on{Perv}(X(\calK)/G(\calO))\is \on{Perv}(K(\calK)\backslash\Gr)\is\on{Perv}(G_\bbR(\calO_\bbR)\backslash\Gr_\bbR)\]
compatible with fusion products.

\end{corollary}
}

\section{Applications}\label{applications}
We provide numerous applications  of the main results to real and relative 
Langlands duality.

\subsection{$t$-exactness criterion and semi-simplicity of Hecke actions }
\begin{thm}\label{t-exact}
The Hecke actions on $D(X(\calK)/G(\calO))$ and  $D(K(\calK)\backslash\Gr)$ (resp. 
$D(G_\bbR(\calO_\bbR)\backslash\Gr_\bbR)$)
are $t$-exact if and only if $X$ is quasi-split (resp. $G_\bbR$ is quasi-split).
\end{thm}
\begin{proof}
It was shown in \cite[Theorem 1.2.3]{N2} that
the nearby cycles functor 
$\psi:D(G(\calO)\backslash\Gr)\to D(G_\bbR(\calO_\bbR)\backslash\Gr_\bbR)$
is $t$-exact if and only if $G_\bbR$ is quasi-split.
Since the 
Hecke action on the real Satake category 
is given by the convolution $\mF\star\psi(\mF)$, 
the $t$-exactness for the convolution in Lemma \ref{real convolution}
 implies the desired claim for the case $D(G_\bbR(\calO_\bbR)\backslash\Gr_\bbR)$.
We deduce the case of $D(X(\calK)/G(\calO))$ from 
Theorem \ref{real-symmetric}
and Theorem \ref{conv comp}.

\end{proof}

\begin{remark}
It would be nice is one can find a  
direct geometric proof for the $t$-exactness criterion
for the relative Satake category.
\end{remark}

\begin{thm}\label{semisimplicity}
The Hecke action on the real and relative Satake categories
satisfies the conclusion of decomposition theorem. That is,
for any semi-simple complexes $\mM\in D(G(\calO)\backslash\Gr)$ and
$\mF\in D(G_\bbR(\calO_\bbR)\backslash\Gr_\bbR)$ or $D(K(\calK)\backslash\Gr)$ or
$D(X(\calK)/G(\calO))$, the convolutions  
$\mF\star\mM$ is again semi-simple.

In particular,
the nearby cycles functor 
$\psi:D(G(\calO)\backslash\Gr)\to D(G_\bbR(\calO_\bbR)\backslash\Gr_\bbR)$
preserves semi-simplicity.

\end{thm}
\begin{proof}
By the real-symmetric equivalence Theorem \ref{real-symmetric} and Theorem \ref{conv comp}, 
it suffices to prove the case of relative Satake categories.
Pick a placid presentation $X(\calK)=\on{colim}_{i\in I}\on{lim}_{j\in J} Y^i_j$. 
We can assume 
the twisted product $\mF\tilde\boxtimes\mM\in D(X(\calK)\times^{G(\calO)}\Gr)$
is supported on $Y^i\times^{G(\calO)}\overline{G(\calO)t^\lambda G(\calO)}/G(\calO)$
for some $i\in I$ and $\lambda\in\Lambda_T^+$. One can find 
an (large enough) index $i'\in I$ such that the image of the action map 
$a:Y^i\times^{G(\calO)}\overline{G(\calO)t^\lambda G(\calO)}\to X(\calK)$
lands in $Y^{i'}$. Note that $Y^i\times^{G(\calO)}\overline{G(\calO)t^\lambda G(\calO)}$ is $G(\calO)$-placid and 
the induced map
\[a:Y^i\times^{G(\calO)}\overline{G(\calO)t^\lambda G(\calO)}\to Y^{i'}.\]
is a proper morphism. 
Thus by \cite[Theorem 8.10.5]{EGA IV}, one can find 
a placid presentation $Y^i\times^{G(\calO)}\overline{G(\calO)t^\lambda G(\calO)}\is\on{lim}_{j\in J}Z_j$, an index $j\in J$,
and a proper morphism $\bar a: Z_j\to Y_{j}^{i'}$ 
such that there is a Cartesian diagram
\[\xymatrix{Y^i\times^{G(\calO)}\overline{G(\calO)t^\lambda G(\calO)}/G(\calO)\ar[r]^{\ \ \ \ \ \ \ \ \ \ \ \ a}\ar[d]^{h}&Y^{i'}/G(\calO)\ar[d]^p\\
Z_j/G(\calO)\ar[r]^{\ \ \ \ \ \ \ \ \ \ \bar a}& Y^{i'}_j/G(\calO)}\]
Moreover, we can assume 
$\mF\tilde\boxtimes\mM\is h^!\mF'$ for some semisimple object 
$\mF'\in D(Z_j/G(\calO))$. 
Since $\bar a$ is proper, 
 the Decomposition Theorem implies $\bar a_*(\mF)$ is semi-simple and 
the proper base change theorem implies that 
\[\mF\star\mM\is a_*(\mF\boxtimes\mM)\is a_*h^!(\mF')\is p^!(\bar a)_*(\mF')\]
is semi-simple.

\end{proof}

\subsection{Formality and commutativity of  dg Ext algebras}\label{FC}
We have the monoidal abelian Satake equivalence 
\[\on{Rep}(G^\vee)\is\on{Perv}^{}(G(\calO)\backslash\Gr_G):\ \ \ V\to \ \on{IC}_V.\]
By restricting the Hecke action to the subcagegory $\on{Rep}(G^\vee)\is \on{Perv}^{}(G(\calO)\backslash\Gr_G)$, we obtain 
a monoidal action of $\on{Rep}(G^\vee)$ on $D^{}(X(\calK)/G(\calO))$ and $D(G_\bbR(\calO_\bbR)\backslash\Gr_\bbR)$.
Let $\omega_{X(\calO)/G(\calO)}\in\on{Perv}(X(\calK)/G(\calO))$ be the 
dualizing complex on 
the closed $G(\calO)$-orbit
$X(\calK)^0\is X(\calO)$ and $\delta_\bbR\in\on{Perv}(G_\bbR(\calO_\bbR)\backslash\Gr_\bbR)$
be the $\IC$-complex of the closed $G_\bbR(\calO_\bbR)$-orbit $S^0_\bbR$.

Let $\IC_{reg}:=\on{IC}_{\calO(G^\vee)}$
(an ind-object in $\on{Perv}^{}(G(\calO)\backslash\Gr_G)$)
be the image of the regular representation 
$\calO(G^\vee)$
under the abelian Satake equivalence.
Since $\calO(G^\vee)$ is a ring object in $\on{Rep}(G^\vee)$, 
$\IC_{reg}$ is naturally a ring object 
in $D(G(\calO)\backslash\Gr)$, that is, there is a natural 
homomorphism 
\beq
m:\IC_{reg}\star\IC_{reg}\to \IC_{reg}
\eeq
satisfying the unit and associativity properties. It follows that the
RHom spaces 
\[A_X:=R\Hom_{D^{}(X(\calK)/G(\calO))}(\omega_{X(\calO)/G(\calO)},\omega_{X(\calO))/G(\calO)}\star\on{IC}_{reg})\]
\[\ A_\bbR:=R\Hom_{D(G_\bbR(\calO_\bbR)\backslash\Gr_\bbR)}(\delta_\bbR,\delta_\bbR\star\on{IC}_{reg})\]
are  naturally  dg-algebras with natural $G^\vee$-actions, known as the the de-equivariantized Ext algebras for the symmetric and real Satake categories respectively.
We denote by $H^\bullet (A_X)=\on{Ext}^\bullet_{D^{}(X(\calK)/G(\calO))}(\omega_{X(\calO)/G(\calO)},\omega_{X(\calO))/G(\calO)}\star\on{IC}_{reg}))$
and $H^\bullet(A_\bbR)=\on{Ext}^\bullet_{D(G_\bbR(\calO_\bbR)\backslash\Gr_\bbR)}(\delta_\bbR,\delta_\bbR\star\on{IC}_{reg})$
the corresponding cohomology algebras  with trivial differentials.

\begin{thm}\label{formality and comm}
(1) There is a $G^\vee$-equivariant isomorphism of 
dg algebras 
$A_X\is A_\bbR$ inducing a $G^\vee$-equivariant isomorphism
of algebras $H^\bullet (A_X)\is H^\bullet (A_\bbR)$.

(2) The dg algebra $A_X$ (resp. $A_\bbR$)  is 
formal, that is, they are quasi-isomorphic to the 
cohomology algebras  
$H^\bullet(A_X)$ (resp. $H^\bullet (A_\bbR)$)
with trivial differential.

(3) 
The algebra $H^\bullet(A_X)$ (resp. $H^\bullet (A_\bbR)$) is commutative.
\end{thm}
\begin{proof}
Part (1) follows from Theorem \ref{real-symmetric}
and Theorem \ref{conv comp}.

For part (2) and (3), using (1), it suffices to prove 
that $A_X$ is formal and $H^\bullet (A_\bbR)$ is commutative.
The formality of $A_X$ is proved in \cite{CY} using a pointwise  purity result for 
$\IC$-complexes of $G(\calO)$-obits in $X(\calK)$.
The proof of the commutativity of $H^\bullet(A_\bbR)$
 is similar to the case of complex groups.
Note that $\delta_\bbR\star\IC_{reg}\is \psi(\IC_{reg})$
where $\psi$ is the monoidal functor in~\eqref{psi}.
Recall that $\IC_{reg}$ is a \emph{commutative ring object} in $D(G(\calO)\backslash\Gr)$, that is, 
there is a natural isomorphism 
\[m\circ\sigma\is m:\IC_{reg}\star\IC_{reg}\to \IC_{reg}\]
where
\beq\label{commutative constraint}
\sigma:\IC_{reg}\star\IC_{reg}\is \IC_{reg}\star\IC_{reg}
\eeq
is the commutativity  constraint 
of the convolution product. 
The monoidal structure of $\psi$
gives rise to a multiplication morphism of $\psi(\IC_{reg})$:
\[
m_\bbR:\psi(\IC_{reg})\star\psi(\IC_{reg})\is\psi(\IC_{reg}\star\IC_{reg})\stackrel{\psi(m)}\lra \psi(\IC_{reg})
\]
satisfying the unit and associativity properties 
and there is a natural isomorphism
\beq\label{commutativity}
m_\bbR\circ\sigma_\bbR\is m_\bbR:\psi(\IC_{reg})\star\psi(\IC_{reg})\to \psi(\IC_{reg})
\eeq
where
\[\sigma_\bbR:\psi(\IC_{reg})\star\psi(\IC_{reg})\is\psi(\IC_{reg}\star\IC_{reg})\stackrel{\psi(\sigma)}\is\psi(\IC_{reg}\star\IC_{reg})\is\psi(\IC_{reg})\star\psi(\IC_{reg}).\]
Note that 
the multiplication of  
\[H^\bullet(A_\bbR)=\on{Ext}^\bullet_{D(G_\bbR(\calO_\bbR)\backslash\Gr_\bbR)}(\delta_\bbR,\delta_\bbR\star\on{IC}_{reg})\is \on{Ext}^\bullet_{D(G_\bbR(\calO_\bbR)\backslash\Gr_\bbR)}(\delta_\bbR,\psi(\IC_{reg}))\] is induced from the multiplication morphism 
$m_\bbR$ of $\psi(\IC_{reg})$:
for $x:\delta_\bbR\to\psi(\IC_{reg})[i]\in\on{Ext}^i_{D(G_\bbR(\calO_\bbR)\backslash\Gr_\bbR)}(\delta_\bbR,\psi(\IC_{reg}))$, $y:\delta_\bbR\to\psi(\IC_{reg})[j]\in\on{Ext}^j_{D(G_\bbR(\calO_\bbR)\backslash\Gr_\bbR)}(\delta_\bbR,\psi(\IC_{reg}))$ 
there exists an unique map
\[x\tilde\boxtimes y:\delta_\bbR\tilde\boxtimes\delta_\bbR\to \psi(\IC_{reg})[i]\tilde\boxtimes\psi(\IC_{reg})[j]\in
\on{Ext}^{i+j}_{D(G_\bbR(\calO_\bbR)\backslash\Gr_\bbR\tilde\times\Gr_\bbR)}(\delta_\bbR\tilde\boxtimes\delta_\bbR,\psi(\IC_{reg})\tilde\boxtimes\psi(\IC_{reg}))\]
such that 
\[q^*(x\tilde\boxtimes y)=p^*(x\boxtimes y)\in \on{Ext}^{i+j}_{D(G_\bbR(\calO_\bbR)\backslash G_\bbR(\calK_\bbR)\times\Gr_\bbR)}(\delta_\bbR\boxtimes\delta_\bbR,\psi(\IC_{reg})\boxtimes\psi(\IC_{reg}))\]
(where $p$ and $q$ are the maps in the convolution diagram~\eqref{convolution diagram})
and the product $x\star y\in \on{Ext}^{i+j}_{D(G_\bbR(\calO_\bbR)\backslash\Gr_\bbR)}(\delta_\bbR,\psi(\IC_{reg}))$ is given by 
\[x\star y= m_!(x\tilde\boxtimes y):\delta_\bbR\is\delta_\bbR\star\delta_\bbR\to
\psi(\IC_{reg})[i]\star\psi(\IC_{reg})[j]\stackrel{m_\bbR}\to\psi(\IC_{reg})[i+j]\]
Since the commutativity  constraint $\sigma$ 
exchanges the factors of $\IC_{reg}\star\IC_{reg}$, it follows that 
$\sigma_\bbR$ exchanges the factors of 
$\psi(\IC_{reg})\star\psi(\IC_{reg})$ and the isomorphism~\eqref{commutativity}
implies that the algebra 
$\on{Ext}^\bullet_{D(G_\bbR(\calO_\bbR)\backslash\Gr_\bbR)}(\delta_\bbR,\psi(\IC_{reg}))$
is commutative.

\end{proof}

Theorem \ref{formality and comm} implies the following spectral descriptions of the 
full subcategories
$D_c(X(\calK)/G(\calO))_0\subset D(X(\calK)/G(\calO))$ and $D_c(G_\bbR(\calO_\bbR)\backslash\Gr_\bbR)_0\subset
D(G_\bbR(\calO_\bbR)\backslash\Gr_\bbR)$ in Section \ref{duals}.
Consider the Hamiltonian duals $M^\vee_X=\on{Spec}(H^\bullet(A_X))$ and $M_\bbR^\vee=\on{Spec}(H^\bullet(A_\bbR))$
of $X$ and $G_\bbR$.
  
\begin{thm}\label{real-sym satake}
(1) There is a $G^\vee$-equivariant isomorphism 
$M^\vee_X\is M^\vee_\bbR$.
(2)There are equivalences of categories
\beq\label{derived Satake}
D_c(X(\calK)/G(\calO))_0\is \on{Coh}(M^\vee_X/G^\vee)\ \ \ \ D_c(G_\bbR(\calO_\bbR)\backslash\Gr_\bbR)_0\is \on{Coh}(M_\bbR^\vee/G^\vee).
\eeq
\end{thm}
\begin{proof}
Part (1) follows from Theorem \ref{formality and comm}  and part (2)
 follows from the formality of dg Ext algebras $A_X$ and $A_\bbR$
and the general Barr-Beck-Lurie theorem, see the details in \cite[Theorem 5.5]{CMNO}.
\end{proof}

\subsection{Identification of dual groups}\label{identification}
In this section we show that there is an isomorphism 
between the dual 
group $H_{real}^\vee\subset G^\vee$ of $G_\bbR$
introduced in \cite{N2} and the dual group $H_{sph}^\vee\subset G^\vee$ 
of $X$ introduced in \cite{GN1}.
\subsubsection{Construction of $H_{real}^\vee$}

\begin{definition}
 Let $Q_\bbR\subset\on{Perv}(G_\bbR(\calO_\bbR)\backslash\Gr_\bbR)$ be 
be the full subcategory whose objects are isomorphic to 
direct sum of perverse sheaves that appears in the summand 
of $\delta_\bbR\star\IC_V$ for some $V\in\on{Rep}(G^\vee)$.
Let $Q_X\subset\on{Perv}(X(\calK)/G(\calO))$ be 
be the full subcategory whose objects are isomorphic to 
direct sum of perverse sheaves that appears in the summand 
of $\omega_{X(\calO)/G(\calO)}\star\IC_V$ for some $V\in\on{Rep}(G^\vee)$.
Let $Q_K\subset\on{Perv}(K(\calK)\backslash\Gr)$ be 
be the full subcategory whose objects are isomorphic to 
direct sum of perverse sheaves that appears in the summand 
of $\omega_{K(\calK)\backslash\mO_K^0}\star\IC_V$ for some $V\in\on{Rep}(G^\vee)$.
\end{definition}

Consider the perverse Hecke actions
\beq\label{perverse Hecke actions}
\star^p:\on{Rep}(G^\vee)\times Q\to Q\ \ \  \ \ \IC_V\star^p\mF:=\bigoplus_{i\in\bZ} {^p}H^i(\IC_V\star\mF)
\eeq
where $Q=Q_\bbR$, $Q_X$ or $Q_K$.

\begin{prop}\label{Q_X=Q_R}
(1) The abelian category  $Q_\bbR$ is semi-simple 
 and irreducible objects are intersection cohomology sheaves on the 
 orbits closures of strata $S^\lambda_\bbR$, $\lambda\in\sigma(\Lambda_T)$ (see~\eqref{sigma}), with 
 coefficients in trivial local systems.
There exists a unique associativity and commutativity constraints 
for the 
category $Q_\bbR$ equipped with the convolution product $\star$
such that $(Q_\bbR,\star)$ is a neutral Tannakian category with 
fiber functor 
$\on{H}^*:Q_\bbR\to\on{Vect}$ given by cohomology.
The Tannakian group of $Q_\bbR$ is isomorphic to the 
 connected complex reductive subgroup $H_{real}^\vee\subset G^\vee$ of the dual group associated to $G_\bbR$ in \cite{N2} and there is a horizontal tensor equivalence in the following diagram of tensor functors
 \[\xymatrix{&\on{Rep}(G^\vee)\ar[rd]\ar[ld]&\\
 Q_\bbR\ar[rr]^\simeq&&\on{Rep}(H_{real}^\vee)}\]
 where left vertical arrow is given by the perverse Hecke action on $\delta_\bbR$.
 

(2)
There exists unique associativity and commutativity constraints 
for the 
category $Q_X$ (resp. $Q_K$) equipped with the fusion product $\star_f$ such that 
 there are horizontal tensor equivalences 
in the following commutative diagram of 
 tensor functors
  \beq\label{tensor compatibility}
 \xymatrix{&\on{Rep}(G^\vee)\ar[rd]^{}\ar[ld]_{}\ar[d]&\\
 Q_X\ar[r]^\simeq&Q_K\ar[r]^\simeq&Q_\bbR}.
 \eeq
where vertical arrows are given by the perverse Heck-actions $\star^p$ of $\on{Rep}G^\vee$ \eqref{perverse Hecke actions} on $\omega_{X(\calO)/G(\calO)}$, $\omega_{K(\calK)\backslash\mO_K^0}$,
and $\delta_\bbR$ respectively.

\end{prop}
\begin{proof}
The semi-simplicity of the Hecke action in
Theorem \ref{semisimplicity}
implies that 
the category $Q_\bbR$ is the same as the full subcategory 
$Q(\Gr_\bbR)\subset\on{Perv}(G_\bbR(\calO_\bbR)\backslash\Gr_\bbR)$ introduced \cite{N2} whose objects are isomorphic to 
subquotients of perverse sheaves that appears in the summand 
of $\delta_\bbR\star\IC_V$ for some $V\in\on{Rep}(G^\vee)$.
Now part (1) follows from the main results in \cite{N2}.

Proposition \ref{fusion compatibility}
implies that there are equivalence 
$Q_\bbR\is Q_X\is Q_K$ compatible with the 
Hecke action, and the convolution product on $Q_\bbR$
and the fusion products on $Q_X$ and $Q_K$.
Thus we can
transport the 
associativity and commutativity constraints of $(Q_\bbR,\star)$
through the above equivalence to 
$(Q_X,\star_f)$ and $(Q_K,\star_f)$. Now the commutativity of~\eqref{tensor compatibility}
follows again from  part (1) and Theorem \ref{conv comp}.

\end{proof}

\begin{corollary}\label{description of Q_X}
There are tensor equivalences $Q_X\is Q_K\is \on{Rep}(H_{real}^\vee)$.
The abelian category $Q_X$ (resp. $Q_K$) is semi-simple with irreducible objects 
$\IC_X^\lambda$ (resp. $\IC_K^\lambda$), $\lambda\in\sigma(\Lambda_T)$.
\end{corollary}

\subsubsection{Construction of $H_{sph}^\vee$}
Fix a pole point $0\in\bbP^1$
and consider 
 the stack of quasi-maps  $Z:=QM^{(1)}(\bbP^1,0,X)$
classifying a $G$-bundle $\mE$ on $\bbP^1$
 and a section 
$\phi:\bbP^1\setminus\{0\}\to \mE\times^GX$, or equivalently, a $K$-reduction $\mE_K$ of 
$\mE$ on $\bbP^1\setminus\{0\}$. 
We have the uniformization isomorphism
\beq\label{uniformization of Z}
Z\is LK^{(1)}\backslash\Gr^{(1)}|_{0}\is K(\bC[t^{-1}])\backslash\Gr
\eeq
in Section~\eqref{uniform of quasi maps},
here $t$ is the local coordinate at $0$.
The stratum $\mO_K^{\lambda}\subset\Gr$, $\lambda\in\Lambda_S^+$ in Section \ref{BD Gr} decends to
a stratum $Z^\lambda=K(\bC[t^{-1}])\backslash\mO_K^{\lambda}\subset Z$, which is
 smooth, locally closed sub-stack of $Z$ and 
the collection $\calZ=\{Z^\lambda\}$ forms a stratification of $Z$ (see \cite[Section 3.4]{GN1}).
For example, the closed stratum $Z^0$ is isomorphic to 
\[Z^0\is K(\bC[t^{-1}])\backslash K(\calK)/K(\calO)\is\on{Bun}_K(\bP^1).\]

Following \cite{GN1}, we define the ind-stack $\on{Hecke}_Z$ of 
\emph{generic Hecke modifications} 
to the ind-stack classifying data
\[(\calE_1,\calE_2,\phi_1,\phi_2,\underline z,\tau)\]
where $(\calE_i,\phi_i)\in Z$, $\underline z\in\on{Sym} \mathbb(\bP^1)$ is a divisor on 
$\bbP^1$ with support  contained in $\bP^1\setminus\{0\}$,
and $\tau$ is an isomorphism of $G$-bundles 
$\tau:\mE_{1}\is\mE_2$ on $\mathbb P^1\setminus\underline z$
compatible with the $K$-reductions of $\mE_{1,K}$ and $\mE_{2,K}$
on $\bP^1\setminus\{0,\underline z\}$.
We have natural projection maps
\[Z \stackrel{p_1}\longleftarrow \on{Hecke}_Z\stackrel{p_2}\lra Z\ \ \ \ \]
given by 
\[p_i((\calE_1,\calE_2,\phi_1,\phi_2,\underline z,\tau))=(\mE_i,\phi_i).\]
We define a \emph{smooth generic Hecke correspondence} to be any stack $Y$ equipped with
smooth maps 
\[Z \stackrel{h_{1,Y}}\longleftarrow Y\stackrel{h_{2,Y}}\lra Z\]
such that there exists a map $Y\to\on{Hecke}_Z$ such that the following diagram commutes
\[\xymatrix{&Y\ar[ld]_{h_{1,Y}}\ar[rd]^{h_{2,Y}}\ar[d]&\\
Z&\on{Hecke}_Z\ar[r]^{p_1}\ar[l]_{p_1}&Z}\]

We also have the ind-stack $Z\tilde\times\Gr_{}=\tilde Z\times^{G(\calO_{})}\Gr_{}$ of Hecke modifications at $\{0\}$ 
where $\tilde Z\to Z$ is the $G(\calO_{})$-torsor classiyfing the data
\[(\mE,\phi,\sigma)\]
where $(\mE,\phi)\in Z$ and $\sigma$ is a trivialization 
$\mE|_{\on{Spec}(\calO_{})}\is G\times\on{Spec}(\calO_{})$
of $\mE$ on the formal neighborhood of $\{0\}$. 
We have natural projection maps
\[Z \stackrel{h_1}\longleftarrow Z\tilde\times\Gr_{}\stackrel{h_2}\lra Z\ \ \ \ \]

Let $\on{Perv}(Z)$ be the category of perverse sheaves on $Z$.
Let  $\on{Perv}_{\calZ}(Z)\subset \on{Perv}(Z)$ the full subcategory 
of perverse sheaves which are locally constant with respect to the stratification 
$\calZ=\{Z^\lambda\}$.
We define a \emph{generic Hecke-equivariant} perverse sheaf 
on $Z$ to a perverse sheaf $\mF\in\on{Perv}(Z)$ on $Z$ equipped with 
isomorphisms 
\[\phi_Y:h_{1,Y}^!\mF\is h_{2,Y}^!\mF\]
for every smooth generic Hecke correspondence $Y$, satisfying some uatural 
conditions, see \cite[Section 3.2]{GN1}.
We denote by 
$\on{Perv}(Z)^{\on{Hecke}}$ the category of generic Hecke-equivariant
perverse sheaves on $Z$.
Since we assume $K$ is connected, the condition of generic-Hecke equivariance is a property, not 
addition structure of a perverse sheaf on $Z$ by \cite[Proposition 3.5.2]{GN1}, we see that 
the natural forgetful map 
$\on{Perv}(Z)^{\on{Hecke}}\to \on{Perv}(Z)$ is fully-faithful
and induces an equivalence 
$\on{Perv}(Z)^{\on{Hecke}}\is \on{Perv}_{\calZ}(Z)$.

Following \cite[Section 4.2]{GN1},
consider the perverse Hecke action 
\[\star^p:
\on{Rep}(G^\vee)\times\on{Perv}(Z)\to\on{Perv}(Z)\ \ \ \ \ \IC_V\star^p\mF:=\bigoplus_{i} {^p}H^i((h_2)_!(\mF\tilde\boxtimes\IC_V))
\]
where $\mF\tilde\boxtimes\IC_V\in\on{Perv}(Z\tilde\times\Gr)$
is the twisted product of $\mF\boxtimes\IC_V$
with respect to the projections $h_1$ and $h_2$.
Since the generic Hecke modifications commute with 
Hecke modification at the pole point $\{0\}$, the perverse Hecke action descends to 
a well defined functor
\beq\label{perverse HK for X}
\star^p:\on{Rep}(G^\vee)\times\on{Perv}(Z)^{\on{Hecke}}\to\on{Perv}(Z)^{\on{Hecke}}\ \ 
\eeq

For any $\lambda\in\Lambda_A^+$, let 
$\IC_{Z^\lambda}\in\on{Perv}(Z)$ be the intersection cohomology 
complex of the stratum $Z^\lambda$ (with constant coefficient).
By \cite[Proposition 3.5.1]{GN1},
we have $\IC_{Z^\lambda}\in\on{Perv}(Z)^{\on{Hecke}}$.

\begin{definition}\cite[Definition 4.2.3]{GN1}
Let $Q_K^{glob}\subset \on{Perv}(Z)^{\on{Hecke}}$ be the full
subcategory of $\on{Perv}(Z)^{\on{Hecke}}$ whose objects are isomorphic 
to direct summands of perverse sheaves appear in
$\IC_V\star^p\IC_{Z^0}$ for some $V\in\on{Rep}(G^\vee)$.

\end{definition}

We have the 
following description of $Q_K^{glob}$. 
\begin{prop}\cite[Theorem 1.2.1]{GN1}\label{Q_X^glob}
$Q_K^{glob}$ is a semi-simple abelian category
and every irreducible object is isomorphic to 
$\IC_{Z^\lambda}$ for some $\lambda\in\Lambda_A^+$.
\end{prop}

\begin{remark}
The proof of the proposition in \emph{loc. cit.}  
is quite involved. We will give 
an another proof of it using the results of the paper, see Corollary \ref{Weyl groups}.
In particular, we will show that in fact  the irreducible objects of 
$Q_K^{glob}$ are isomorphic to 
$\IC_{Z^\lambda}$, $\lambda\in\sigma(\Lambda_T)$, confirming a conjecture in \emph{loc. cit.}.
\end{remark}

We now recall the construction of fusion product on 
$Q_K^{glob}$ following \cite[Section 6.3]{GN1}.
Consider the base change
\[Z^{(2)}=QM^{(2)}(\mathbb P^1,X)\times_{(\bP^1)^2}\bC\]
of $QM^{(2)}(\mathbb P^1,X)\to(\bP^1)^2$
along the diagonal embedding 
$
\bC\lra (\mathbb P^1)^2$, $z\to (z,-z)$ 
We have 
\[Z^{(2)}|_{0}\is Z\]
and there is a 
Cartesian diagram
\[\xymatrix{Z^{(2)}|_{\bC^\times}\ar[d]\ar[r]^j&Z^{(2)}\ar[d]^f&Z\ar[d]\ar[l]_{\ \ i}\\
\bC^\times\ar[r]&\bC&\{0\}\ar[l]}\]

The stratification $\{\mO_K^{(2),\lambda_\fp}\}$ of $\Gr^{(2)}$
in Section \ref{BD Gr} descends to a stratification 
$\{LK^{(2)}\backslash\mO_K^{(2),\lambda_\fp}\}$
of $QM^{(2)}(\mathbb P^1,X)\is LK^{(2)}\backslash\Gr^{(2)}$
which restricts to a stratification 
$\{Z^{(2),\lambda_\fp}\}$ of the base change $Z$, where  
\beq\label{stratum Z^2}
Z^{(2),\lambda_\fp}=(LK^{(2)}\backslash\mO_K^{(2),\lambda_\fp})\cap Z=(LK^{(2)}\backslash\mO_K^{(2),\lambda_\fp})\times_{(\bP^1)^2}\bC
\eeq
We denote by $\IC_{Z^{(2),\lambda_\fp}}\in\on{Perv}(Z^{(2)})$
the $\IC$-complex of the stratum $Z^{(2),\lambda_\fp}$.

For any $\lambda_1,\lambda_2\in\Lambda_A^+$
let $\lambda_{1\cup2}:\fp=\{1\}\cup\{2\}\to\Lambda_A^+$ be the 
the map $\lambda_{1\cup2}(i)=\lambda_i$.
Note that 
$Z^{(2),\lambda_{1\cup2}}\subset Z^{(2)}|_{\bC^\times}$.
For any $\IC_{Z^{\lambda_1}},\IC_{Z^{\lambda_2}}\in\on{Perv}(Z)$
we define the (global) fusion product of 
as 
the following nearby cycles
\beq\label{global fusion}
\IC_{Z^{\lambda_1}}\star_f\IC_{Z^{\lambda_2}}=\psi_f(\IC_{Z^{(2),\lambda_{1\cup2}}})\in\on{Perv}(Z)
\eeq
along the projection map $f:Z^{(2)}\to\bC$
\begin{prop}\label{H_sph}
(1) For any $\IC_{Z^{\lambda_1}},\IC_{Z^{\lambda_2}}\in Q_K^{glob}$, we have $\IC_{Z^{\lambda_1}}\star_f\IC_{Z^{\lambda_2}}\in Q_K^{glob}$.
(2) There exists unique associativity and commutativity constraints 
for the 
category $Q_K^{glob}$ equipped with the fusion product $\star_f$
such that $(Q_K^{glob},\star_f)$ is a neutral Tannakian category with 
Tannakian group isomorphic to the
reductive subgroup  $H^{\vee}_{sph}\subset G^\vee$ 
associated to $X$ in \cite{GN1}.
Moreover, 
we have the following commutative diagram of tensor functors
\[\xymatrix{&\on{Rep}(G^\vee)\ar[ld]\ar[rd]&\\
Q_K^{glob}\ar[rr]^\simeq&&\on{Rep}(H^\vee_{sph})}\]
where the left vertical arrow is given by the perverse Hecke action 
on $\IC_{Z^0}\in Q_K^{glob}$.
\end{prop}
\begin{proof}
\cite[Corollary 4.2.6]{GN1}
and \cite[Lemma 6.3.1]{GN1} imply
that $\IC_{Z^{(2),\lambda_{1\cup2}}}\is j_{!*}j^*(\IC_{Z^{(2),\lambda_{1\cup2}}})$
 is ULA with respect to the projection 
$f:Z^{(2)}\to\bC$. Thus by \cite[Theorem A.2.6]{Z1} there is an isomorphism
\[\psi_f(\IC_{Z^{(2),\lambda_{1\cup2}}})\is i^*j_{!*}(j^*\IC_{Z^{(2),\lambda_{1\cup2}}})[-1].\]
It follows that the fusion product in~\eqref{global fusion} is the same as the one defined in 
\cite[Section 6.3]{GN1} and the proposition follows from the main results of \cite{GN1}.
\end{proof}

\subsubsection{The identification $H^\vee_{real}=H_{sph}^\vee$}
The uniformization map~\eqref{uniformization of Z} induces a map 
\[r:Z\is K(\bC[t^{-1}])\backslash\Gr\to K(\calK)\backslash\Gr.\]

\begin{prop}
The functor $r^{!}[\dim K]:
D(K(\calK)\backslash\Gr)\lra D(Z)$ is 
$t$-exact 

\end{prop}
\begin{proof}
The proof is similar to the one in Theorem \ref{real-symmetric}.
For any $\lambda\in\Lambda_A^+$,
consider the following Cartesian diagram
\[\xymatrix{Z^\lambda\ar[r]^{j_Z^\lambda}\ar[d]^{r^\lambda}&Z\ar[d]^r\\
 K(\calK)\backslash\mO^\lambda_K\ar[r]^{j_K^\lambda}&  K(\calK)\backslash\Gr}\]
We need to check that for any 
$\mF\in\on{Perv}(K(\calK)\backslash\Gr)$, we have
\beq\label{t-exact claim}
(j_Z^\lambda)^*r^{!}[\dim K](\mF)\in {^{p_{cl}}}D^{\leq0}(Z^\lambda)\ \ \ 
(j_Z^\lambda)^!r^{!}[\dim K](\mF)\in {^{p_{cl}}}D^{\geq0}(Z^\lambda).
\eeq

By Proposition \ref{key base change} (in the case $n=1$), the map
$r$ is strongly pro-smooth and the functor $r^!$
satisfies base change along $*$-pullback  along any fp-locally closed 
map $\calS\to K(\calK)\backslash\Gr$. It follows that 
\beq\label{dim Z^lambda}
\dim Z^\lambda=\dim\calO^\lambda_K-\dim\calO^0_K+\dim Z^0=
\langle\lambda,\rho\rangle+\dim Z^0=\langle\lambda,\rho\rangle-\dim K
\eeq
(note that $Z^0\is\Bun_K(\bP^1)$
and $\dim \Bun_K(\bP^1)=-\dim K$)
and 
\[ (j_Z^\lambda)^! r^{!}(\mF)[\dim K]\is (r^\lambda)^!(j_K^\lambda)^!(\mF)[\dim K]\]
\[(j_Z^\lambda)^* r^{!}(\mF)[\dim K]\is (r^\lambda)^!(j_K^\lambda)^*(\mF)[\dim K].\]
Since  $\mF$ is perverse sheaf on $K(\calK)\backslash\Gr\is X(\calK)/G(\calO)$ with respect to the 
perversity function in~\eqref{perversity}, we have 
\[(j_K^\lambda)^!(\mF)\in {^p}D^{\geq0}( K(\calK)\backslash\mO^\lambda_K)[-\langle\lambda,\rho\rangle]\ \ \ \ 
(j_K^\lambda)^*(\mF)\in {^p}D^{\leq0}( K(\calK)\backslash\mO^\lambda_K)[-\langle\lambda,\rho\rangle].\]
On the other hand, since $r^\lambda$ is strongly pro-smooth, 
 the characterization of $!$-adapted $t$-structures in~\eqref{characterization'} implies that 
$(r^\lambda)^!$ is $t$-exact with respect to the $!$-adapted $t$-structures  \cite[Proposition 6.3.3 (c)]{BKV} and
hence
\[(r^\lambda)^!((j_K^\lambda)^!(\mF))[\dim K]\in {^p}D^{\geq0}(Z^\lambda)[-\langle\lambda,\rho\rangle+\dim K]\stackrel{~\eqref{dim Z^lambda}}=
{^p}D^{\geq0}(Z^\lambda)[-\dim Z^\lambda]=
{^{p_{cl}}}D^{\geq0}(Z^\lambda)\]
\[(r^\lambda)^!((j_K^\lambda)^*(\mF))[\dim K]\in {^p}D^{\leq0}(Z^\lambda)[-\langle\lambda,\rho\rangle+\dim K]\stackrel{~\eqref{dim Z^lambda}}=
{^p}D^{\leq0}(Z^\lambda)[-\dim Z^\lambda]=
{^{p_{cl}}}D^{\leq0}(Z^\lambda)\]
The desired claim~\eqref{t-exact claim} follows.

\end{proof}

It follows from the proposition above that $r^![\dim K]$
restricts to a functor
\[r^![\dim K]:\on{Perv}(K(\calK)\backslash\Gr)\to\on{Perv}(Z)\]
on the category of perverse sheaves.

\begin{thm}\label{local=global}
The functor $r^![\dim K]$
restricts to
the horizontal tensor equivalence in a commutative
diagram of tensor functors
\[\xymatrix{&\on{Rep}G^\vee\ar[rd]\ar[ld]&\\
Q_K\ar[rr]^\simeq&&Q_K^{glob}}\]

\end{thm}
\begin{proof}
Since the functor $r^![\dim K]$ is $t$-exact and commutes with the convolution action of the Hecke category 
$D(G(\calO)\backslash\Gr)$
from the right, it restricts to a functor  $r^![\dim K]:Q_K\to Q_K^{glob}$.
We first show that $r^![\dim K]$  induces an
equivalence of semi-simple abelian categories
\beq\label{Phi}
\Phi:Q_K\is Q_K^{glob}
\eeq
By the description of $Q_K$ and $Q_K^{glob}$ in 
Proposition \ref{description of Q_X} and Proposition \ref{Q_X^glob}, it suffices to 
show  that, for any $\lambda\in\Lambda_A^+$, there is an isomorphism
\beq\label{IC=IC}
r^{!}[\dim K](\IC_X^\lambda)\is \IC_{Z^\lambda}
\eeq
Since $j_{K}^\lambda$ is fp-locally closed embedding 
the base change isomorphisms in
 \cite[Lemma 5.4.5]{BKV}
and 
Lemma \ref{key lemma}
 imply
\[r^!\circ (j_{K}^\lambda)_*\is (j_Z^\lambda)_*(r^\lambda)^!\ \ \ \ \ r^!\circ (j_{K}^\lambda)_!\is (j_Z^\lambda)_!(r^\lambda)^! \]
and~\eqref{IC=IC} follows from 
the $t$-exactness of $r^![\dim K]$ and 
\[\IC_{Z^\lambda}\is \on{Im}({^{p_{cl}}}H^0((j_{Z}^\lambda)_!(\omega_{Z^\lambda}[-\dim Z_\lambda]))\to {^{p_{cl}}}H^0((j_{Z}^\lambda)_*(\omega_{Z^\lambda}[-\dim Z_\lambda])))\]
\[\IC_K^{\lambda}\is \on{Im}({^p}H^0((j_{K}^\lambda)_!(\omega_{K(\calK)\backslash\mO_K^\lambda}[-\langle\lambda,\rho\rangle]))\to {^p}H^0((j_{K}^\lambda)_*(\omega_{K(\calK)\backslash\mO_K^\lambda}[-\langle\lambda,\rho\rangle]))).\]

We shall show that $\Phi:=r^{!}[\dim K]:Q_K\is Q_K^{glob}$ is a tensor equivalence.
We first show that for any  
$\IC^{\lambda_1}_K,\IC^{\lambda_2}_K\in Q_K$, there is a canonical isomorphism 
\beq\label{c}
c_{}:\Phi(\IC^{\lambda_1}_K\star_f\IC^{\lambda_1}_K)\is\Phi(\IC^{\lambda_1}_K)\star_f\Phi(\IC^{\lambda_1}_K)\is \IC_{Z^{\lambda_1}}\star_f\IC_{Z^{\lambda_2}}.
\eeq
called the \emph{monoidal structure} on $\Phi$.

The uniformization isomorphism 
$QM^{(2)}(\bP^1,X)\is LK^{(2)}\backslash\Gr^{(2)}$
gives rise to 
a map 
\[r^{(2)}:Z^{(2)}\is LK^{(2)}\backslash\Gr^{(2)}|_\bC\to K(\calK)^{(2)}\backslash\Gr^{(2)}|_\bC.\]
Consider the following Cartesian diagram 
\[\xymatrix{Z^{(2)}|_{\bC^\times}\ar[r]^{j}\ar[d]^{r^{(2)}|_{\bC^\times}}&Z^{(2)}\ar[d]^{r^{(2)}}& Z\ar[d]^r\ar[l]_{i}\\
(K(\calK)\backslash\Gr)^2\times \bC^\times\ar[r]^{ j}\ar[d]&(K(\calK)^{(2)}\backslash\Gr^{(2)})|_{\bC}\ar[d]&K(\calK)\backslash\Gr\ar[d]\ar[l]_{\ \ \ \ i}\\
\bC^\times\ar[r]&\bC&\{0\}\ar[l]}\]
Proposition \ref{key base change} implies that the map 
$r^{(2)}$ is strongly pro-smooth and 
Lemma \ref{key lemma} 
implies that 
\beq\label{nearby cycles 1}
\Phi(\IC^{\lambda_1}_K\star_f\IC^{\lambda_2}_K)=r^![\dim K]i^* j_*(\IC^{\lambda_1}_K\boxtimes\IC^{\lambda_2}_K\boxtimes\underline{\bC}_{\bbR_{>0}}) )\is
\eeq
\[\is i^* j_*(r^{(2)}_{>0})^!(\IC^{\lambda_1}_K\boxtimes\IC^{\lambda_2}_K\boxtimes\underline{\bC}_{\bbR_{>0}}))[\dim K].
\]
On the other hand, the same argument of proving~\eqref{IC=IC} show that 
\[(r^{(2)}|_{\bC^\times})^!(\IC^{\lambda_1}_K\boxtimes\IC^{\lambda_2}_K\boxtimes\underline{\bC}_{\bbR_{>0}}[1]))[\dim K]\is\IC_{Z^{(2),\lambda_{1\cup 2}}}|_{\bbR_{>0}}\]
where $\IC_{Z^{(2),\lambda_{1\cup 2}}}$ is the $\IC$-complex on the stratum 
$Z^{(2),\lambda_{1\cup 2}}\subset Z^{(2)}$ in~\eqref{stratum Z^2}.
Thus we have
\beq\label{nearby cycles 2}
i^* j_*(r^{(2)}_{>0})^!(\mF\boxtimes\mF'\boxtimes\underline{\bC}_{\bbR_{>0}}))[\dim K]\is 
i^* j_*(\IC_{Z^{(2),\lambda_{1\cup 2}}}|_{\bbR_{>0}})[-1]\is
\psi_f(\IC_{Z^{(2),\lambda_{1\cup 2}}})
\eeq
\[\is
\IC_{Z^{\lambda_1}}\star_f\IC_{Z^{\lambda_2}}\]
Combining~\eqref{nearby cycles 1} and~\eqref{nearby cycles 2}, we obtain the desired monoidal structure  
\[c_{}:\Phi(\IC^{\lambda_1}_K\star_f\IC^{\lambda_2}_K)\is i^* j_*(r^{(2)}|_{\bC^\times})^!(\IC^{\lambda_1}_K\boxtimes\IC^{\lambda_2}_K\boxtimes\underline{\bC}_{\bbR_{>0}}))[\dim K]\is \IC_{Z^{\lambda_1}}\star_f\IC_{Z^{\lambda_2}}\]
in~\eqref{c}.

Finally, we show that  $\Phi$ is compatible with associativity and commutativity constraints of 
$Q_K$ and $Q_K^{glob}$.
Consider the following diagram of functors
\beq\label{Tannakian}
\xymatrix{&\on{Rep}G^\vee\ar[dr]\ar[dl]&\\
Q_K\ar[rr]^\Phi_\simeq\ar[dr]_\omega&&Q_K^{glob}\ar[dl]^{\omega^{glob}}\\
&\on{Vect}&}\eeq
where $\omega$ and $\omega^{glob}$ are the fiber functors.
The upper triangle and the outer square are 
commutative.
We claim that the lower triangle is commutative, that is, there is natural isomorphism 
$\omega^{glob}\circ\Phi\is\omega$.
For this we observe that 
the composition
\[\on{Rep}G^\vee\to Q_K\stackrel{\omega^{glob}\circ\Phi}\to\on{Vect}\]
is the fiber functor for $\on{Rep}G^\vee$ and, for any $\mF,\mF\in Q_K$, the 
the monoidal structure for $\Phi$ defines an isomorphism
\[\omega^{glob}\circ\Phi(\mF\star_f\mF')\is \omega^{glob}\circ\Phi(\mF)\otimes\omega^{glob}\circ\Phi(\mF').\]
Thus by \cite[Lemma 3.3]{Z2}, $\omega^{glob}\circ\Phi:Q_K\to \on{Vect}$
has a structure of a fiber functor and
since $Q_K$  is a neutral Tannakian category over $\bC$,
 the uniqueness of fiber functors in
\cite[Theorem 3.2]{DM} implies that there is an isomorphism $\omega^{glob}\circ\Phi\is\omega$. 

Let $R$,
$B_K$, $B_K^{glob}$ be the Hopf algebra corresponding to 
$\on{Rep} G^\vee$, $Q_K$, and $Q_K^{glob}$
under the Tannakian dictionary. 
By \cite[Proposition 2.16]{DM},
the commutative diagram
of functors in~\eqref{Tannakian} gives rise to a maps 
\[
\xymatrix{&R\ar[dr]\ar[dl]&\\
B_K\ar[rr]^{\Phi^*}&&B_K^{glob}}
\]
where the vertical arrows are \emph{surjective} morphisms of Hopf algebras
and the horizontal arrow $\Phi^*$ is an \emph{isomorphism}  of co-algebras. 
By \cite[Lemma  9.2.1]{N2}, the monoidal structure for $\Phi$
implies that $\Phi^*$ respects multiplication
and the surjectivity of the vertical arrows implies that 
$\Phi^*$ is an isomorphism of 
Hopf algebras.
The theorem follows.

\end{proof}

The following corollary follows immediately from 
Proposition \ref{Q_X=Q_R}, Proposition \ref{H_sph}, confirming \cite[Conjecture 7.3.2.]{GN1} in the case of symmetric varieties. 

\begin{corollary}\label{Weyl groups}
(1)
There is an isomorphism $H^\vee_{sph}\is H^\vee_{real}$. 
In particular, the Weyl group of $H^\vee_{sph}$ is isomorphic to the small Weyl group 
of $X$. 
(2) The 
irreducible objects in $Q_K^{glob}$  are 
 intersection cohomology sheaves $\IC_{Z^\lambda}$
on the closures of strata 
$Z^\lambda$, $\lambda\in\sigma(\Lambda_T)$, with
coefficients in trivial local systems.
\end{corollary}

\appendix
\section{Semi-analytic stacks}\label{Real stacks}
\subsection{Basic definitions}
Recall that a subset $Y$ of a 
real analytic manifold $M$ 
is called semi-analytic if
any point $y\in Y$ has a open neighbourhood $U$ 
such that the intersection $Y\cap U$ is a finite union of sets of the form
\[\{y\in U|f_1(y)=\cdot\cdot\cdot=f_r(y)=0, g_1(y)>0,..., g_l(y)>0\},\]
where the $f_i$ and $g_j$ are real analytic functions on 
$U$. A map $f:Y\to Y'$ between two semi-analytic sets is called semi-analytic if
it is continuous and its graph is a semi-analytic set.

Let $\on{Grpd}$ be the $\infty$-category of spaces, which are often referred as $\infty$-groupoids.
Let $\on{RSp}$ be the site of semi-analytic sets where the coverings are 
\'etale (=locally bi-analytic) maps $\{S_i\to S\}_{i\in I}$ such that the map $\bigsqcup S_i\to S$ is surjective. 

\begin{definition}
A semi-analytic pre-stack is a functor $\calY:\on{RSp}\to\on{Grpd}$ 
 and a semi-analytic stack is a 
pre-stack which is a sheaf. 
\end{definition}
We will view any semi-analytic set as a semi-analytic stack via the Yoneda embedding.

A morphism $\calX\to\calY$ between  stacks is called representable if
for any morphism from a semi-analytic set
$Y\to\calY$, the fiber product $\calX\times_\calY Y$ is representable by a semi-analytic set. We say that a representable morphism $\calX\to\calY$ has property 
P if it has property P after base change along any morphism 
from a semi-analytic sets.

Let $\calX\to\calY$ be a  morphism of semi-analytic stacks.
One can associate its $\check{\on{C}}$ech complex with terms 
$\calX^{[n]}=\calX\times_\calY\calX\times\cdot\cdot\cdot\calX\times_\calY\calX$
the $(n+1)$-times fiber product of $\calX$ over $\calY$. When $f$ is surjective, there is a natural isomorphism 
\beq\label{Cech complex}
\xymatrix{
\calY\is \on{colim} (\cdot\cdot\cdot\ar@<-.5ex>[r] \ar@<.5ex>[r]  \ar@<.2ex>[r]\ar@<-.2ex>[r]& \calX^{[2]} \ar@<-.5ex>[r] \ar@<.5ex>[r]  \ar@<.0ex>[r]& \calX^{[1]} \ar@<-.5ex>[r] \ar@<.5ex>[r]  & \calX)}
\eeq

A semi-analytic stack $\calY$ is a semi-analytic space if for every $S\in\on{RSp}$, 
$\calY(S)$ is isomorphic to a set, that is, each connected component of $\calY(S)$ is contractible.
Let $\Gamma\rightrightarrows Y$ be a groupoid object 
in the category of semi-analytic space.
The  quotient stack 
$\Gamma\backslash Y$ is defined as the colimit
\beq\label{quotient stack}
\xymatrix{
\Gamma\backslash Y\is \on{colim} (\cdot\cdot\cdot\ar@<-.5ex>[r] \ar@<.5ex>[r]  \ar@<.2ex>[r]\ar@<-.2ex>[r]&\Gamma\times_Y\Gamma  \ar@<-.5ex>[r] \ar@<.5ex>[r]  \ar@<.0ex>[r]& \Gamma \ar@<-.5ex>[r] \ar@<.5ex>[r]  & Y)}
\eeq

\subsection{From  stacks to 
semi-analytic stacks}
Let $F=\bbR$ or $\bbC$.
For any scheme $Y$ over $F$ of finite type, its $F$-points  
$Y(F)$ is naturally a semi analytic set, denoted by $Y_F$.
For any strict ind-scheme $Y\is\on{colim}_{i\in I} Y^i$
over $F$, its $F$-points is naturally a semi-analytic space, denoted by 
$Y_F$.

In the paper, we will mainly consider $F$-stacks of the form 
$\calY\is Y/G$ where $Y$ is a strict ind-scheme acted on by a group ind-scheme 
$G$. 
A surjective morphism $f:Y\to\calY$ from a strict ind-scheme $Y$ to a $F$-stack
$\calY$ is called a 
$F$-surjective presentation if it induces a surjective map  
$Y(F)\to |\calY(F)|$ on the set of isomorphism classes of objects. 
Note that when $F=\bC$ any surjective morphism is a $\bC$-surjective 
presentation.

\begin{lemma}\label{image}
Let $f_1:Y_1\to\calY$ and $f_2:Y_2\to\calY$ be 
two $F$-surjective presentations of $\calY$.
Let $\Gamma_i=Y_i\times_\calY Y_i\rightrightarrows Y_i$ be the corresponding 
groupoid. 
Then
there is a canonical isomorphism of semi-analytic stacks
\[\Gamma_{1,F}\backslash Y_{i,F}\is\Gamma_{2,F}\backslash Y_{2,F}.\]

\end{lemma}
\begin{proof}
Let $Y=Y_1\times_\calY Y_2$ and $\Gamma=Y\times_\calY Y$
be the corresponding groupoid. Then 
the natural map $\Gamma_F\backslash Y_F\to\Gamma_{i,F}\backslash Y_{i,F}$ is an isomorphism. The lemma follows.

\end{proof}

\begin{definition}
Given a   stack $\calY$ over $F$ which admits a $F$-surjective presentation, we define the associated 
semi-analytic stack to be 
\[\calY_F:=\Gamma_F\backslash Y_F\]
where $Y\to\calY$ is a $F$-presentation of $\calY$.
\end{definition}
By the lemma above $\calY_F$ is well-defined and the assignment 
$\calY\to\calY_F$ defines a functor from the category of 
stacks over $F$
 which admit $F$-presentations to the category of 
semi-analytic stacks.

\begin{example}\label{BG}
Consider the case $F=\bbR$.
Let $Y$ be a $\bbR$-scheme and $G$ be an algebraic group over $\bbR$ acting 
on $Y$. Consider the algebraic stack $\calY=G\backslash Y$.
Let $T_1,...,T_s\in H^1(\on{Gal}(\bbC/\bbR),G(\bbC))$ be the isomorphism classes of 
$G$-torsors.
  Define $G_i:=\on{Aut}_G(T_i)$ and the $\bbR$-scheme 
$Y_i:=\Hom_{G}(T_i,Y)$. Note that 
$G_i$ acts on $Y_i$ and 
the collection $\{G_1,...,G_s\}$ gives all the pure-inner 
forms of $G$. 
Consider the real algebraic stack $G_i\backslash Y_i$.  
We have $G_i\backslash Y_i\is\calY$ and the map 
$\bigsqcup_{i=1}^sY_i\to\calY$ is a $\bbR$-surjective presentation. In addition, 
the $\bbR$-surjective presentation above induces an isomorphism of semi-analytic stacks 
$\bigsqcup_{i=1}^sG_{i,\bbR}\backslash Y_{i,\bbR}\is \calY_\bbR$.

\end{example}

One can regard real  stacks as complex stacks with real structures and the discussion above has an obvious generalization to 
this setting.  Let $\calY$ be a complex  stack and let
$\sigma$ be a real structure on $\calY$, that is, a complex conjugation (or a semi-linear involution) $\sigma:\calY\to\calY$. Then 
a surjective morphism $f:Y\to\calY$ of $\calY$
is called a 
$\bbR$-surjective presentation if it satisfies the following properties.
(1) There is a real structure $\sigma$ on $X$ such that 
$f$ is compatible with the real structures on $Y$ and $\calY$.
(2) The map $f$  
induces a surjective map  
$Y(\bC)^{\sigma}\to |\calY(\bC)^{\sigma}|$. One can check that 
Lemma \ref{image} still holds in this setting, 
thus 
for a pair $(\calY,\sigma)$ as above which admits a $\bbR$-surjective presentation, there is a well-defined 
semi-analytic stack $\calY_\bbR$ given by 
$\calY_\bbR:=\Gamma(\bC)^\sigma\backslash Y(\bC)^\sigma$, where 
$Y\to\calY$ is a $\bbR$-surjective presentation, 
$\Gamma=Y\times_{\calY}Y$ is the corresponding groupoid (Note that 
$\Gamma$ has a canonical real structure $\sigma$ coming from 
$Y$ and $\calY$).

\subsection{Constructible complexes }\label{sheaves}
We will be working with $\bC$-linear dg-categories. Unless specified otherwise, all dg-categories will be assumed cocomplete, i.e., containing all small colimits,
and all functors between dg-categories will be assumed continuous, i.e., preserving all small colimits.

For any semi-analytic set $S$, 
we define $D(Y)=\on{Ind}(D_c(Y))$ to be the ind-completion
of the 
 bounded dg-category $D_c(Y)$ of $\bC$-constructible sheaves on 
$Y$. For any semi-analytic stack $\calY$ 
we define $D(\calY):=\underset{}{\on{lim}}^!\ D(Y)$
where the index category is that of semi-analytic sets equipped with a semi-analytic map
to $\calY$, and the transition functors are given by $!$-pullback.
Since we are in the constructible context, $!$-pullback admits a left adjoint, given by 
$!$-pushforward, and it follows that 
$D(\calY)=\underset{}{\on{colim}_!}\ D(Y)$. In particular,
$D(\calY)$ is compactly generated. 
For an ind semi-analytic stacks $\calY=\on{colim}\calY_i$ we denote by 
$D_c(\calY)\subset D(\calY)$ the (non co-complete) full subcategory
consisting of complexes that are extensions by
zero off of substacks.

Let $\calX\to\calY$ be a surjective morphism of semi-analytic stacks.
Then the isomorphism~\eqref{Cech complex} induces an equivalence
\beq\label{limit}
\xymatrix{
D(\calY)\is\on{lim} (D(\calX)\ar@<-.5ex>[r] \ar@<.5ex>[r]  & D(\calX^{[1]}) \ar@<-.5ex>[r] \ar@<.5ex>[r]  \ar@<.0ex>[r]& D(\calX^{[2]}) \ar@<-.5ex>[r] \ar@<.5ex>[r] \ar@<.2ex>[r]\ar@<-.2ex>[r] & \cdot\cdot\cdot)}
\eeq

Let $F=\bC$ or $\bbR$.
For any $F$-stack $\calY$  admitting a 
$F$-surjective presentation, we define 
$D(\calY)=D(\calY_F)$ where $\calY_F$ is the associated semi-analytic stack. 


\end{document}